\documentclass[11pt]{amsart}
\usepackage{graphicx} 
\usepackage{subcaption}
\usepackage{booktabs}
\usepackage[utf8]{inputenc}
\usepackage[margin=1in]{geometry}
\usepackage{amsfonts,amssymb,amsmath,amsthm,tikz,comment,mathtools,setspace,float,stmaryrd,datetime,bbm,adjustbox}
\usepackage[toc]{appendix}
\usepackage{enumitem}
\usepackage{xcolor}

\usepackage[hidelinks]{hyperref}
\usepackage{scalerel} 
\numberwithin{equation}{section}
\usepackage{chngcntr}
\counterwithin{figure}{section}

\allowdisplaybreaks

\newcommand{\N}{\mathbb{N}}
\newcommand{\R}{\mathbb{R}}
\newcommand{\Q}{\mathbb{Q}}
\newcommand{\Z}{\mathbb{Z}}
\newcommand{\F}{\mathcal{F}}

\newcommand{\X}{\mathcal{X}}

\newcommand{\ve}{\varepsilon}
\newcommand{\Ll}{\mathcal{L}}
\newcommand{\Pp}{\mathbb P}
\newcommand{\f}{\frac}
\newcommand{\vep}{\varepsilon}

\newcommand{\wt}{\widetilde}

\newcommand{\Rup}{\R_{{\scaleobj{0.7}{\uparrow}}}^4}
\newcommand{\wh}{\widehat}

\newcommand{\W}{W}

\newcommand{\dir}{\xi}

\newcommand{\graph}{\mathcal G}

\newcommand{\sig}{{\scaleobj{0.8}{\boxempty}}} 
\newcommand{\sigg}{{\scaleobj{0.9}{\boxempty}}} 

\newcommand{\kpzs}{\mathfrak{h}} 
\newcommand{\fun}{\Phi} 
\newcommand{\intc}{\tau}
\newcommand{\sn}{\mathfrak{sn}}
\newcommand{\es}{\mathcal{B}}
\newcommand{\hm}{\mathcal{H}}

\newcommand{\be}{\begin{equation}}
\newcommand{\ee}{\end{equation}}

\newcommand{\Wf}[1]{W\star#1\,}

\DeclareMathOperator*{\argmax}{arg\,max}   

\def\dirs{\pmb{\xi}} 

\def\tex{t_{\text{ext}}}
\def\xex{x_{\text{ext}}}
\def\pex{p}
\def\clr{\mathfrak{C}}
\def\ls{\mathfrak{a}}
\def\rs{\mathfrak{b}}
\def\Xs{\mathfrak{X}}
\def\Rm{\hat{\mathbb{R}}}

\newcommand{\Dom}{\operatorname{Dom}}
\newcommand{\supp}{\operatorname{supp}}

\newtheorem{theorem}{Theorem}[section]
\newtheorem{proposition}[theorem]{Proposition}
\newtheorem{corollary}[theorem]{Corollary}
\newtheorem{lemma}[theorem]{Lemma}

\theoremstyle{definition}
\newtheorem{definition}[theorem]{Definition}
\newtheorem{example}[theorem]{Example}

\theoremstyle{remark}
\newtheorem{remark}[theorem]{Remark}

\title[Eternal solutions for the KPZ fixed point]{Classification of the eternal solutions and multiple coalescing shocks in the KPZ fixed point}

\author{Sudeshna Bhattacharjee, Ofer Busani, and Evan Sorensen}
\address{Sudeshna Bhattacharjee\\University of Maryland, College Park\\Office: 4117, William E. Kirwan Hall\\4176 Campus Dr, College Park\\MD 20742, USA}
\email{sudeshna@umd.edu}
\urladdr{https://sites.google.com/view/sudeshnabhattacharjee/home}
\address{Ofer Busani \\ University of Edinburgh\\ School of Mathematics \\ 5321, James Clerk Maxwell Building\\
Peter Guthrie Tait Road\\
Edinburgh, EH9 3FD}
\email{obusani@ed.ac.uk}
\urladdr{https://oferbusani.github.io/}
\address{Evan Sorensen \\ Indiana University \\ Department of Mathematics \\ Room 311 Rawles Hall, 831 E 3rd St, Bloomington, IN 47405, USA}
\email{elsorens@iu.edu}
\urladdr{https://sites.google.com/view/evan-sorensen}

\begin{document}
\maketitle

\begin{abstract}
    We give a complete classification of the eternal solutions for the KPZ fixed point. Each of these is a (possibly infinite) max-plus convolution of the known eternal solutions, called Busemann functions. Specifically, we show that the space of eternal solutions is homeomorphic to a certain space of upper semicontinuous functions that encodes the weights of each of the Busemann functions. As a result, we show that the Busemann process gives a spectral decomposition of eternal solutions of the KPZ fixed point analogous to that for Hamilton-Jaccobi equations. The resulting evolution of the KPZ fixed point exhibits a shock at each of the boundaries between the different Busemann functions. Moving forward in time, the shocks coalesce, while moving backwards in time, additional shocks can form.  We describe several geometric properties of this tree of shocks. This completes the study of eternal solutions initiated in earlier work of the second and third authors with Sepp\"al\"ainen and continued in the previous work of the authors and in recent work of Rassoul-Agha and Sweeney.
\end{abstract}
\tableofcontents

\section{Introduction} 

\subsection{Growth models and the Kardar-Parisi-Zhang (KPZ) universality class} The KPZ universality class contains a broad class of random growth models with one time and one space dimension, each exhibiting universal scaling exponents and limiting statistics. The last ten years has seen tremendous progress in this area. In particular, in the case where the spatial domain is the full line, the complete space-time scaling limits have been constructed. These objects are known as the KPZ fixed point \cite{KPZfixed} and the directed landscape \cite{Directed_Landscape}, which built off many previous results in capturing the marginal distributions of these limiting processes \cite{Baik-Deift-Johansson-1999,Johansson-2000,Prahofer-Spohn-02,CorwinHammond}. Each of the models in the KPZ universality class exhibits an evolving height function converging to the KPZ fixed point. These height functions are often described by a Hamilton-Jacobi PDE with stochastic forcing, or are a discrete approximation of such a PDE. Viscous Hamilton-Jacobi equations can be solved by first solving a linear equation, then taking its logarithm. This is known as the Cole-Hopf-transform. The corresponding linear equation has a Green's function, and in the case of stochastic forcing, this Green's function is random. Under an appropriate scaling, the logarithm of this Green's function converges to the directed landscape. This was proved in the case of the KPZ equation in \cite{Wu-23}. Thus, the directed landscape can be thought of as the random medium through which solutions of the KPZ fixed point evolve. Indeed, the KPZ fixed point can be described via a variational formula involving the directed landscape, recorded below as \eqref{eq:KPZFP_DL}. This is analogous to the solution of a inviscid Hamilton-Jacobi PDE, although it has been shown that the KPZ fixed point does not itself solve a stochastic PDE \cite{Himwich-Parekh-2024}. Convergence to the directed landscape has also been shown for many metric-like models, including exactly solvable last-passage percolation models \cite{Dauvergne-Virag-21}, the asymmetric simple exclusion process \cite{Aggarwal-Corwin-Hegde-2024}, and the inverse-gamma polymer \cite{Zhang-2025}. More recently, a characterization of the directed landscape from the KPZ fixed point was given in \cite{Dauvergne-Zhang-2024}.

\subsection{The directed landscape}
Let $\Rup = \{(x,s;y,t) \in \R^4: s < t\}$. The directed landscape is a random continuous function $\Ll:\Rup \to \R$ originally constructed in \cite{Directed_Landscape} as the scaling limit of last-passage percolation. It was shown in \cite{reflected_KPZfixed} that the KPZ fixed point with initial data $f$ from initial time $s$ may be represented as
\be \label{eq:KPZFP_DL}
\kpzs_t(x;f,s) = \sup_{y \in \R}[f(y) + \Ll(y,s;x,t)],\quad t > s,x \in \R.
\ee
A natural state space for the KPZ fixed point consists of upper semi-continuous functions such that the supremum in \eqref{eq:KPZFP_DL} is finite for all $(x,t) \in \R^2$. In \cite{reflected_KPZfixed}, it was shown that the identity in \eqref{eq:KPZFP_DL} is valid for upper-semicontinuous initial conditions $f:\R \to [-\infty,\infty)$ such that $f(x) \le A + B|x|$ for some constants $A,B > 0$ and $f(x) > -\infty$ for some $x \in \R$. However, we will consider a larger space of all functions such that the supremum in \eqref{eq:KPZFP_DL} is finite for all $(x,t) \in \R \times \R_{>s}$. Specifically, the solution is finite for all $(x,t) \in \R \times \R_{>s}$ if and only if 
\[
\sup_{y \in \R}[f(y) - ay^2] < \infty,\quad\text{for all }a > 0.
\]
Furthermore, when this condition holds, the solution is continuous on the domain $\R \times \R_{>s}$ (see Lemmas \ref{lem:f-ay_bd}-\ref{lem:general_restrict_max}). With the directed landscape constructed, we take \eqref{eq:KPZFP_DL} as our definition of the KPZ fixed point, as an alternative to the definition originally given in \cite{KPZfixed} in terms of Fredholm determinant formulas for the transition probabilities. We let $(\Omega,\F,\Pp)$ be a probability space on which the directed landscape is defined, and we will let $\omega$ denote a generic sample point from this space.  

\subsection{Invariant measures and eternal solutions}
As a process in time, the KPZ fixed point is a Markov process, so it is natural to ask about its invariant measures. Strictly speaking, there are no invariant measures for the KPZ fixed point since it is a model of random growth. However, if one considers the recentered solution 
\[
\kpzs_t(x;f,s) - \kpzs_t(0;f,s),
\]
then there is a one-parameter family of invariant measures, namely Brownian motions with diffusion coefficient $\sqrt 2$ and arbitrary drift. This invariance has been known since the construction of the KPZ fixed point \cite{KPZfixed} since, for the prelimiting object (the totally asymmetric exclusion process), i.i.d. Bernoulli measures are invariant. In \cite{Busa-Sepp-Sore-22a}, a one force--one solution principle was shown for the KPZ fixed point. That is, when started with an arbitrary initial condition with asymptotic slopes as $y \to -\infty$ and $y \to +\infty$, the recentered solution at large times converges in distribution to a Brownian motion with drift, where the drift is determined by a phase diagram of the initial slopes. The only case not considered in that work is when the slope to $+\infty$ is $\theta > 0$ and the slope to $-\infty$ is $-\theta$. It was shown in \cite{Dunlap-Sorensen-2025} that there are no invariant measures of this form. This was used in that work to prove that Brownian motions with drift are the only extremal invariant measures for the KPZ fixed point. 

The convergence result in \cite{Busa-Sepp-Sore-22a} actually proved more than just convergence to an invariant measure. Specifically, if we start the process at a time in the distant past, then the recentered solution across all times converges almost surely to an eternal solution whose spatial increments are stationary in time. We define eternal solutions now. 
\begin{definition} \label{def:eternal}
An \textbf{eternal solution} for the KPZ fixed point is any function $b:\R^2 \to [-\infty,\infty)$ such that, for every $x \in \R$ and $s < t$, we have 
\begin{equation} \label{eq:bL_def}
b(x,t) = \sup_{y \in \R}[b(y,s) + \Ll(y,s;x,t)]. 
\end{equation}
We say that $b$ is a \textbf{finite eternal solution} if $b(x,t) > -\infty$ for all $(x,t) \in \R^2$.
\end{definition}
\begin{remark} \label{rmk:b=-inf}
Allowing $b$ to take the value $-\infty$ is done for topological convenience. Since $\Ll$ is real-valued, one can readily see that if $b(x,t) = -\infty$ for some $(x,t) \in \R^2$, then $b(y,s) = -\infty$ for every $s < t$ and $y \in \R$. Hence, $b \equiv -\infty$ is the only eternal solution that is not a finite eternal solution.  
\end{remark}

The eternal solutions obtained as limits in \cite{Busa-Sepp-Sore-22a} are called \textbf{Busemann functions}. In the context of the KPZ fixed point, these were first constructed in \cite{Rahman-Virag-21}. The Busemann functions are indexed by a parameter $\xi \in \R$ related to the slope of the initial condition. The Busemann function is denoted as $W^\xi(0,0;x,t)$ and satisfies
\be \label{eq:slope}
\lim_{|x| \to \infty} \f{W^\dir(0,0;x,t)}{x} = 2\dir,
\ee
for all $t \in \R$.
The work \cite{Busa-Sepp-Sore-22a} extended this to a \textbf{Busemann process} defined for all $\xi \in \R$. However, care must be taken in doing this, as this extension introduces a random, countably dense set of discontinuities, denoted as $\Xi = \Xi_\omega$ (noting that the set $\Xi$ is random). We therefore construct an extended parameter space of the Busemann process, denoted as $\Xs = \Xs_\omega$. The set $\Xs$ consists of elements $\xi = \eta^\sig$, where $\eta \in \R$ and $\sigg \in \{-,0,+\}$. Specifically, we define the set  
\begin{equation} \label{eq:X_space}
\Xs:=\{\eta^{0}:\eta\in \R\setminus \Xi\}\cup \{\eta^{\sig}:\eta\in \Xi,\sigg \in \{-,+\}\}.
\end{equation}
The subscript $\omega$ is used to emphasize that the set is random (depending on the random countably infinite dense set $\Xi$).  We define the projections $\mathcal P:\Xs \to \R$ and $\sn:\Xs \to \{-,0,+\}$ as follows:
\be \label{eq:P_s_funs}
\mathcal P(\xi) = \eta,\quad\text{and}\quad \sn(\xi) = \sigg,\quad\text{when}\quad \xi = \eta^\sig.
\ee
We call the elements of $\Xs$ \textbf{colors}. Define the subspaces
\begin{equation}
    \begin{aligned}
        \Xi^\sig&=\{\dir\in \Xs: \sn(\dir)=\sigg\},\quad \text{and}\quad \Xi^c = \{\dir \in \Xs: \sn(\dir) = 0\}.
    \end{aligned}
\end{equation}
where we have made a slight abuse of notation by identifying the complement $\Xi^c$, which is nominally a subset of $\R$, as a subset of $\Xs$. We can therefore write $\Xs$ as the disjoint union $\Xi^c \cup \Xi^- \cup \Xi^+$.

We define a total order on the set $\Xs$ as follows: we say that $\xi_1 < \xi_2$ if $\mathcal P(\xi_1) < \mathcal P(\xi_2)$ and we also say that $\eta^- < \eta^+$ for $\eta \in \Xi$. We readily see that this defines a total ordering on $\Xs$, so we can make sense of the intervals $(\xi_1,\xi_2), [\xi_1,\xi_2),(\xi_1,\xi_2]$, and $[\xi_1,\xi_2]$.

 We now define a (random) topology on $\Xs$ as follows: First, for $\ve > 0$, we define the $\ve$-neighborhood of $\xi \in \Xs$ as follows:
 \begin{equation}\label{eq:o36}
I_\ve(\dir)=
    \begin{cases}
        (\dir-\ve,\dir+\ve) & \dir\in\Xi^c\\
        [\dir,\dir+\ve) & \dir\in\Xi^+\\
         (\dir-\ve,\dir] & \dir\in\Xi^-.
    \end{cases}
\end{equation}
where if $\xi = \eta^\sig \in \Xs$, $\xi - \ve$ denotes the element $(\eta -\ve)^+$ if $\eta -\ve \in \Xi$ and $(\eta - \ve)^0$ otherwise. Similarly, $\xi + \ve$ denotes $(\eta + \ve)^-$ if $\eta + \ve \in \Xi$ and $(\eta + \ve)^0$ otherwise. The topology on $\Xs$ is then the topology generated by the sets $\{I_\ve(\xi): \ve > 0, \xi \in \Xs\}$. 
 \begin{remark} \label{rmk:convergence}
  One can readily see that convergence of sequences $\{\dir_n\}$ in this topology is then characterized as follows:
\begin{enumerate}
    \item If $\dir \in \Xi^c$, then $\dir_n \to \dir$ if and only if $\mathcal P(\xi_n) \to \mathcal P(\xi)$. 
    \item If $\dir \in \Xi^-$, then $\dir_n \to \dir$ if and only if $\mathcal P(\dir_n) \nearrow \mathcal P(\dir)$.
    \item If $\dir \in \Xi^+$, then $\dir_n \to \dir$ if and only if $\mathcal P(\dir_n) \searrow \mathcal P(\dir)$. 
 \end{enumerate}      
 \end{remark}
We show in Lemma \ref{lem:X_metric} that this topology makes $\Xs$ a metric space. We then write the Busemann process as 
\[
\Bigl(W^{\dir}(0,0;x,t): (x,t) \in \R^2, \dir \in \Xs\Bigr),
\]
When $\eta \in \Xi$, the Busemann functions $W^{\eta-}$ and $W^{\eta+}$  denote the left and right limits of the Busemann process at $\eta$.

 When $\xi \notin \Xi$, the Busemann function is the only eternal solution satisfying \eqref{eq:bL_def} (up to addition by a constant). This is known as the \textbf{one force--one solution principle}. In our previous work \cite{bhat-bus-soren}, we studied the set of eternal solutions satisfying \eqref{eq:bL_def} when $\xi \in \Xi$. From the construction of the Busemann process  we showed that there are uncountably many of these eternal solutions. Each is a ``patching" together of the solutions $(x,t) \mapsto W^{\dir-}(0,0;x,t)$ and $(x,t) \mapsto W^{\dir +}(0,0;x,t)$, separated by a bi-infinite competition interface.  Alternatively, one can write these eternal solutions as 
\be \label{b_max_comb}
b(x,t) = \max\Bigl(W^{\xi-}(0,0;x,t) + \varphi_-, W^{\xi+}(0,0;x,t) + \varphi_+\Bigr),
\ee
where $\varphi_-,\varphi_+ \in [-\infty,\infty)$ (not both $-\infty$), and the associated bi-infinite interface is a continuous path $(\tau_t,t)_{t \in \R}$ along which
\[
W^{\xi-}(0,0;\tau_t,t) + \varphi_- = W^{\xi+}(0,0;\tau_t,t) + \varphi_+.
\]
It is fairly immediate from the definition of eternal solutions that any combination of the form \eqref{b_max_comb} produces another eternal solution. Hence, the novelty of \cite{bhat-bus-soren} came from showing that all eternal solutions satisfying \eqref{eq:bL_def} are of this form. This came from a careful analysis of the associated semi-infinite geodesics and Busemann process.  

One can also see quickly that we can produce eternal solutions from this maximal combination with \textit{any} two eternal solutions. For example, if $\xi_1 < \xi_2$ (ignoring the $\pm$ distinction for now), then for any $\varphi_1,\varphi_2 \in \R$, the combination
\[
b(x,t) := \max\Bigl(W^{\xi_1}(0,0;x,t) + \varphi_1, W^{\xi_2}(0,0;x,t) + \varphi_2\Bigr)
\]
gives an eternal solution satisfying 
\be \label{eq:left_right_slopes}
\lim_{x \to -\infty} \f{b(x,t)}{x} = \xi_1,\quad\text{and}\quad \lim_{x \to +\infty} \f{b(x,t)}{x} = \xi_2.
\ee
It is known from \cite{Busa-Sepp-Sore-22a,Dunlap-Sorensen-2025} that there are no invariant measures with different slopes as $x \to -\infty$ and $x \to \infty$. In fact, Brownian motions with drift are the only extremal invariant measures. Hence, since we have an eternal solution of this form, the interface must diverge as $t \to \infty$, and in any compact window of space (values of $x$), the solution converges to exactly one of the eternal solutions as $t \to \infty$.

Furthermore, one is not restricted to using only two eternal solutions in creating a new eternal solution. The maximum of any collection of eternal solutions (even an infinite collection) will also produce an eternal solution, provided that the resulting solution is finite  (see Lemma \ref{lem:interchange_sups}). This creates a tree of interfaces that is more interesting and difficult to analyze. Furthermore, the first limit in \eqref{eq:left_right_slopes} can be $-\infty$, and the second limit can be $+\infty$. Since no invariant measure has this property, for each fixed compact window of space, the solution in that window will not converge to any particular Busemann function; instead, infinitely many of the Busemann functions will pass through the window as $t \to \infty$.

We introduce the following notation: For a function $\varphi:\Xs \to [-\infty,\infty)$, define the function $W\star \varphi:\R^2 \to [-\infty,\infty]$ by
\[
(W\star \varphi)(x,t) := \sup_{\dir \in \Xs}[\varphi(\dir) + W^\dir(0,0;x,t)].
\]

Let $\es_\omega$ denote the set of eternal solutions. A priori, we do not assume anything about the regularity of the eternal solutions; we only assume that they satisfy  Definition \ref{def:eternal}. However, in Corollary \ref{cor:cts_bounded} we show that $(x,t) \mapsto b(x,t)$ is continuous for any finite eternal solution $b$ (See the recent preprint \cite[Theorem 3.7]{RassoulAghaSweeney2026} for an alternative proof of the continuity of $x \mapsto b(x,t)$ for each $t$). Hence, $\es_\omega$ is a topological space when considered as a subspace of the space $C(\R^2,[-\infty,\infty))$ equipped with the uniform-on-compact topology. Let $\mathcal T_{\es_\omega}$ denote the topology on this space. It is clear from the discussion above that the space  $\es_\omega$ is closed under the $\text{max}$ operator $\vee$. 

Next, let $\operatorname{USC}(\Xs)$ denote the space of all upper semicontinuous functions $\varphi:\Xs \to [-\infty,\infty)$. Here, upper semicontinuity is defined in the usual way by the condition that 
\[
\limsup_{\eta \to \xi} \varphi(\eta) \le \varphi(\xi),
\]
for all $\xi \in \Xs$, where the convergence $\eta \to \xi$ is with respect to the topology on $\Xs$. Next, define
\begin{equation} \label{eq:Phi_omega_def}
\fun_\omega:=\Bigl\{\varphi\in \operatorname{USC}(\Xs):\limsup_{|\mathcal P(\xi)| \to \infty} \f{\varphi(\xi)}{\mathcal P(\xi)^2} = -\infty\Bigr\}.
\end{equation}

We equip the space $\fun_\omega$ with a topology $\mathcal{T}_{\fun_\omega}$ determined by convergence of hypographs and uniform boundedness of $\f{\varphi(\dir)}{\mathcal P(\dir)^2}$ outside a compact set of $\dir$. This topology is defined precisely in Definition \ref{def:fun_convergence}. The space $\Phi$ also comes equipped with its own version of the max operator $\vee$, defined in the obvious way by $(\varphi_1 \vee  \varphi_2)(\dir) := \varphi_1(\dir) \vee \varphi_2(\dir)$. Our main result is the following
\begin{theorem}\label{thm:mr}
For almost-every $\omega$, the mapping $\hm_\omega[\varphi] := W\star \varphi$ is a homeomorphism $(\fun_\omega,\mathcal T_{\fun_\omega}) \to (\es_\omega,\mathcal T_{\es_\omega})$ that satisfies
    \be \label{eq:hom_max}
    \hm[\varphi_1] \vee \hm[\varphi_2] = \hm[\varphi_1 \vee \varphi_2],\quad \forall \varphi_1,\varphi_2 \in \Phi,
    \ee
    and whose inverse is given by
    \be \label{eq:hom_inv}
    \hm^{-1}[b](\dir) = \inf_{(x,t)\in\R^2}\{b(x,t)-W^\dir(0,0;x,t)\},\quad \forall\; b \in \es.
    \ee
\end{theorem}
Some remarks are in order.
\begin{remark} \label{rmk:non_unique}
    In the definition of $\Phi_\omega$ \eqref{eq:Phi_omega_def}, the superquadratic decay condition ensures that $(W\star \varphi)(x,t) < \infty$ for all $(x,t) \in \R^2$. In general, we have that $W\star \varphi < \infty$ if and only if the superquadratic decay condition holds and $\varphi$ is bounded from above on bounded sets (see Proposition \ref{prop:Wstarphi_bd}). The  condition that $\varphi \in \operatorname{USC}(\Xs)$ in \eqref{eq:Phi_omega_def} gives the uniqueness of the representation $W\star \varphi$ so that the map $\hm$ in Theorem \ref{thm:mr} is a bijection. This uniqueness is one of the key consequences of Theorem \ref{thm:mr}. Without the condition that $\varphi \in \operatorname{USC}(\Xs)$, the uniqueness does not hold.  For example, if one takes $\varphi(\dir) = 0$ for $\dir \in [0,1]$ and $-\infty$ otherwise, Proposition \ref{prop:exists_varphi} below states that $W\star \varphi = W\star \wh \varphi$ for some $\wh \varphi$ with countable support. However, imposing the condition of countable support is not enough to give uniqueness. As an example, consider the function $\varphi$ defined by $ \varphi(n^{-1}) = -\f{1}{n}$ for $n \in \N$ and $\varphi(\xi) = -\infty$ otherwise. We immediately see that $\varphi$ has countable support, and by using continuity of the Busemann process at $\dir = 0$, one can see that $W\star \varphi = W\star \wh \varphi$ if we define $\wh \varphi(\dir) = \varphi(\dir)$ for $\dir \neq 0$ and $\wh \varphi(0) = 0$. We therefore see that the uniqueness statement is quite delicate. It is shown in Section \ref{sec:unique}.
\end{remark}

\begin{remark}
    Since $W^\xi(0,0;x,t)$ grows as $2 \xi x$ for each fixed $t \in \R$, we see that if we choose $\varphi(\xi) \sim -\xi^{\alpha}$ as $\dir \to \infty$ along some subsequence for $\alpha > 2$, then for $x > 0$,
    \[
    b(x,t) \sim \sup_{\xi}[2\xi x - \xi^\alpha] \sim x^{\f{\alpha}{\alpha -1}},
    \]
    so we see a continuum of spatial growth rates for $b(x,t)$ greater than or equal to $O(x)$ and less than $O(x^2)$, depending on the choice of $\alpha$.
\end{remark}

\subsection{Consequences of Theorem \ref{thm:mr}} \label{sec:consequences}
We now discuss some important and very interesting consequences to Theorem \ref{thm:mr}.

\medskip \noindent {\textbf{The Busemann process as a spectral decomposition of the KPZ fixed point eternal solutions}}. A Hamilton Jacobi equation (HJE) with concave Hamiltonian $H$ is defined by 
 \begin{equation}\label{eq:HJE}
     u_t(x,t)+H(t,\nabla u(t,x))=0.
 \end{equation}
 The Lagrangian is given by 
 \begin{equation}
     L(t,v)=\inf_p\{p\cdot v-H(t,p)\},
 \end{equation}
 and the action by 
 \begin{equation} \label{eq:int_Lagrange}
     B_{s,t}(y,x)=\sup_{\substack{\gamma(s)=y\\ \gamma(t)=x}}\int_s^t L(r,\gamma(s))dr.
 \end{equation}
 The solution to equation \eqref{eq:HJE} is then given by 
 \begin{equation} \label{eq:u_B_evolve}
     u(x,t)=S_{s,t}[u(\cdot,s)](x)=\sup_y\{u(y,s)+B_{s,t}(y,x)\}, \qquad \forall s<t.
 \end{equation}
 Many PDEs are amenable to a spectral decomposition which allows one to go from a PDE to a system of infinite ODE's. For example, let $D$ be the differential operator on $\R^d$ and consider the equation 
 \begin{equation}
     u_t=P(D)u, \qquad \text{ on $\R^d$},
 \end{equation}
 for some polynomial $P$. Letting $\{h_p\}=\{e^{ip\cdot x}\}_{p\in\R^d}$, $\oplus=\int$, and $\odot$ be the usual product  on $\R$, then
 \begin{equation}  u(x,t)=\int_{\R^d}\varphi(p,t)e^{ip\cdot x}, \qquad \forall t\in\R
 \end{equation}
 and
 \begin{equation}
     \varphi_t(p,t)=P(ip)\varphi(p,t),\qquad \forall p\in\R^d.
 \end{equation}
 More generally, a PDE is amenable to a spectral decomposition if there exists a family of modes $\{h_p\}_{p\in\mathcal{I}}$ and operations $\odot$ and $\oplus$ such that
 \begin{equation}
     u(\cdot,s)=\bigoplus_p \varphi(p,s)\odot h_p, \qquad \varphi(p,s)\in\R.
 \end{equation}
 and for every $s<t$ and $p\in\mathcal{I}$, $h_p$ is an eigenvector of the solution operator $\{S_{s,t}\}_{s<t}$ i.e.\
 \begin{equation}
     S_{s,t}(c\odot h_p)=c\odot \beta(p;s,t)\odot h_p, \qquad \beta(p;s,t)\in\R.
 \end{equation}
 Thus,
 \begin{equation}  S_{s,t}u(\cdot,s)=\bigoplus\beta(p;s,t)\odot\varphi(p,s)\odot h_p,
 \end{equation}
 In other words, in order to understand the evolution of $S_{s,t}$ it is enough to understand the evolution of $\beta(p;s,t)$ for all $p\in\mathcal{I}$. 
 
 For Hamilton-Jacobi equations, things are a bit more delicate and the decomposition is applicable to a certain initial conditions, namely convex functions. When the Hamiltonian $H$ is strictly concave, one can show that any eternal solution $u$ is convex at each time $t$ i.e.\ $u(\cdot,t)$ is convex. Now, set $\oplus=\vee$ and $\odot=+$ and $h_p(x)=p x$, from the convexity of $u(\cdot,t)$ we see that 
 \begin{equation} \label{eq:o611}
     u(x,t)=\sup_{p \in \R}(\varphi(p,t)+px),
 \end{equation}
  where
 \begin{equation}\label{eq:o61}
     \varphi(p,t)=\inf_{x\in\R}(u(x,t)-px).
 \end{equation}
 Now, let $x^*=x(p,t)$ be the minimizer in \eqref{eq:o61}. Then, by optimality, we have
 \begin{equation}
     \nabla u(x(p,t),t)=p.
 \end{equation}
 Differentiating $\varphi(p,t)=u(x^*,t)-px^*$ with respect to $t$ and using the last display, we obtain
 \begin{equation}
     \varphi_t(p,t)=u_t(x^*,t)
 \end{equation}
 But from the HJE $-u_t(x^*,t)=H(t,p)$, we conclude that 
 \begin{equation} \label{eq:betaH}
     \beta_{\text{\tiny HJE}}(p;s,t)=\varphi(p,t)-\varphi(p,s)=-\int_s^tH(r,p)dr.
 \end{equation}
We now compare this with the results of Theorem \ref{thm:mr}. For each $t\in\R$ and $p\in\Xs$ set $h^t_p(x):=W^p(0,t,x,t) := W^p(0,0;x,t) - W^p(0,0;0,t)$. By a proof similar to that in  Proposition \ref{prop:Legendre}, we see that for $b\in\es$ and $t \in \R$, we may write
\begin{equation} \label{eq:o612}
    b(x,t)=\sup_{p \in \Xs}(\varphi(p,t)+h^t_p(x)),
\end{equation}
where
\begin{equation} \label{eq:o6122}
    \varphi(p,t)=\inf_{x \in \R}(b(x,t)-h^t_p(x)).
\end{equation}
Then, since $b$ and $(x,t) \mapsto h_p^t(x)$ are eternal solutions, for $s < t$, we have 
\begin{align*}
b(x,t)&=\sup_{z \in \R} [b(z,s)+\Ll(z,s;x,t)]=\sup_{z \in \R} \Bigl [\sup_{p \in \Xs}[\varphi(p,s)+h^s_p(z)]+\Ll(z,s;x,t) \Bigr]\\
&=\sup_{p \in \Xs} \Bigl[\varphi(p,s)+\sup_{z \in \R} [h^s_p(z)+\Ll(z,s;x,t)] \Bigr] =\sup_{p \in \Xs } [\varphi(p,s)+W^p(0,s;x,t)]\\&=\sup_{p \in \Xs }  [\varphi(p,s)+W^p(0,s;0,t)+h^t_p(x)].
\end{align*}
From here, we see that

\begin{equation} \label{eq:betaW}
    \beta_{\text{\tiny KPZ}}(p;s,t)=\varphi(p,t)-\varphi(p,s)=W^p(0,s;0,t).
\end{equation}

We now see that as in the HJE where the time-integrated Hamiltonian carries the information of the coefficients of the modes, for the KPZ fixed-point this is done by the evolution of the Busemann process in time. The values of the Busemann process at a fixed time gives the family $\{h_p\}$ that allows us to decompose any eternal solution evaluated at a fixed time. Thus, at least heuristically, by comparing \eqref{eq:o611}-\eqref{eq:o61} to \eqref{eq:o612}-\eqref{eq:o6122}, comparing  \eqref{eq:u_B_evolve} to \eqref{eq:KPZFP_DL} and comparing \eqref{eq:betaH} to \eqref{eq:betaW}, we see that for eternal solutions, we have the following comparison between the deterministic HJE and the KPZ fixed point.
\begin{table}[ht]
\centering
\renewcommand{\arraystretch}{1.25}
\begin{tabular}{ccc}
\textbf{Description}
& \textbf{Hamilton--Jacobi equation}
& \textbf{KPZ fixed point} \\[2mm]
Time evolution
& Integrated Lagrangian \(L\)
& Directed landscape \\[1mm]
Spectral evolution
& Integrated Hamiltonian \(H\)
& The Busemann process
\end{tabular}
\end{table}

\medskip \noindent 
{\textbf{The space  $\es$ has a rich structure}.} From Definition \ref{def:eternal}, it is not hard to see that the set $\es$ is closed under the maximum operation $\vee$. However, it is not easy to find other natural operations under which $\es$ is closed. Theorem \ref{thm:mr} unveils a much richer structure in $\es$ than that of $(\es,\vee)$. Indeed, the space $\text{USC}(\Xs)$ is well understood; it is closed under summation and the max and min operations, and the interaction between the different operations is clear. Any operation $\odot$ in $\fun$ gives rise to one in $\es$ via
\begin{equation}
    b_1\odot b_2= \hm\bigl[\hm^{-1}[b_1]\odot \hm^{-1}[b_2]\bigr].
\end{equation}
Moreover, any algebraic relation between such operations in $\fun$ and their topological properties has a dual picture in $\es$.

Let us see a concrete application of Theorem \ref{thm:mr}. Consider the min operation $\wedge$ on $\Phi$ defined as $(\varphi_1 \wedge \varphi_2)(\dir) = \varphi_1(\dir) \wedge \varphi_2(\dir)$. One can readily see that $\varphi_1 \wedge \varphi_2 \in \Phi$ for all $(\varphi_1,\varphi_2) \in \Phi^2$, as the space of upper semicontinuous functions is closed under $\wedge$. We can also consider the min operator $\wedge$ on the space of functions $C(\R^2,[-\infty,\infty))$. While $\wedge$ preserves the space $C(\R^2,[-\infty,\infty))$, it does not preserve the subspace $\es$. In general, we define the operation $\wt{\wedge}:\es\times \es \to \es$ by 
\[
b_1\, \wt{\wedge}\, b_2 := \hm\bigl[\hm^{-1}[b_1] \wedge \hm^{-1}[b_2]\bigr]
\]

Furthermore, setting $b_i = \hm[\varphi_i]$ for $i = 1,2$, we can readily see that if $b_1 \wedge b_2 \in \es$, then 
\begin{equation}\label{eq:o60}
    b_1 \,\wt{\wedge}\,b_2 = b_1 \wedge b_2.
\end{equation}
Indeed, from \eqref{eq:hom_inv},
\[
    \begin{aligned}
        &\quad \;\hm^{-1}[b_1](\dir)\wedge \hm^{-1}[b_2](\dir) = \inf_{(x,t)\in\R^2}\{b_1(x,t)-W^\dir(0,0;x,t)\}\wedge \inf_{(x,t)\in\R^2}\{b_2(x,t)-W^\dir(0,0;x,t)\}\\
    &=\inf_{(x,t)\in\R^2}\{b_1\wedge b_2(x,t)-W^\dir(0,0;x,t)\}=\hm^{-1}[b_1\wedge b_2](\dir),
    \end{aligned}
\]
 where in the last equality, we used the assumption that $b_1\wedge b_2\in\es$.
  
 It is an interesting open problem to characterize the pairs $(b_1,b_2)$ for which $b_1\wedge b_2\in\es$. Another interesting problem is to characterize all points of continuity of $\wedge$ in $\es\times\es$; i.e.\ all points $(b_1,b_2)\in\es$ such that  whenever $(b_1^n,b_2^n)\rightarrow (b_1,b_2)$ then $b_1^n\wedge b_2^n\rightarrow b_1\wedge b_2$. Note that for a point $(b_1,b_2)$ to be a continuity point  it must have a neighborhood $(b_1,b_2)\in U$ such that $\wedge$ is defined on $U$. We can also ask the same question about the set of continuity points of $\wt \wedge$, which is easier since this operation is defined on all of $\es$. 
The following result and its proof demonstrate how Theorem \ref{thm:mr} can be used to study the space $\es$. 
\begin{proposition}
   The pair $(b_1,b_2)$ is a continuity point of $\wt \wedge$ if and only if $b_1 \equiv - \infty$ or $b_2 \equiv -\infty$. The  operator $\wedge$ has no continuity points in $\es\times\es$.
\end{proposition}
\begin{proof}
 By $\eqref{eq:o60}$ and Theorem \ref{thm:mr} we see that $(b_1,b_2)$ is a continuity point of $\wt \wedge$ in $\es$ if and only if $(\hm^{-1}(b_1),\hm^{-1}(b_2))$ is a continuity point of  $(\fun, \wedge)$, and the space $(\fun, \wedge)\subseteq (\text{USC}(\Xs),\wedge)$ is much easier to work with. Indeed, Lemma \ref{lm:mincon} says that $(\varphi_1,\varphi_2)$ is a continuity point if and only if $\supp(\varphi_1)\cap\supp(\varphi_2)=\varnothing$.  We conclude that $\varphi_1\wedge\varphi_2\equiv-\infty$. But this implies that $\hm[\varphi_1\wedge\varphi_2]\equiv-\infty$, so 
   if we set $b_i := \hm[\varphi_i]$, then $b_1\wedge b_2= \hm[\varphi_1]\wedge\hm[\varphi_2]\equiv -\infty$. By Remark \ref{rmk:b=-inf}, we conclude that that $(b_1,b_2)$ is a continuity point of $(\es,\wt \wedge)$ if and only if either $b_1\equiv-\infty$ or $b_2\equiv -\infty$.
   
   Now, for $(b_1,b_2)$ to be a continuity point of $\wedge$ in $\es$, \eqref{eq:o60} implies that $(b_1,b_2)$ must be a continuity point of $\wt \wedge$, and $\wedge$ must be defined in a neighborhood of $(b_1,b_2)$. Lemma \ref{lm:o6} says that for any $b\in\es$ and any neighborhood $U$ of $(b,-\infty)$ there exists $(b_1,b_2)\in U$ such that $b_1\wedge b_2\notin \es$. This shows that $(b,-\infty)$ is not a continuity point for any $b \in \es$, and by symmetry, $(-\infty,b)$ is also not a continuity point of $\wedge$. 
 \end{proof}

\subsection{Colors and semi-infinite geodesics}
Intimately tied with the Busemann process is the notion of semi-infinite geodesics in the directed landscape. Each Busemann function $W^\dir$ is associated to a family of backward coalescing semi-infinite geodesics known as $\dir$-geodesics. For $\dir \in \Xs$, each semi-infinite geodesic rooted at a point $(x,t)$ is encoded by a continuous function $g_{(x,t)}^\dir:(-\infty,t] \to \R$, and when $\dir = \eta^\sig$, the geodesics have asymptotic direction $\eta$:
\[
\lim_{s \to -\infty} \f{g_{(x,t)}^{\dir}(s)}{|s|} = \eta. 
\]
Each of these families of semi-infinite geodesics is distinct, so when $\eta \in \Xi$, the $\eta^-$ and $\eta^+$ families are non-coalescing families of semi-infinite geodesics in the same direction. It is shown in \cite[Theorem 2.5]{Busa-Sepp-Sore-22a} that every semi-infinite geodesic has a direction, and it was later shown in \cite[Theorem 1.5]{Busani_N3G} that the only semi-infinite geodesics in the directed landscape are those associated with the Busemann functions. See Section \ref{sec:SIG_Buse} for more precise statements.

Given a finite eternal solution $b$, there is a family of semi-infinite geodesics associated to it. Simply stated, the location of a semi-infinite geodesic from $(x,t)$ at time $s < t$ is a maximizer of 
\[
y \mapsto b(y,s) + \Ll(y,s;x,t). 
\]
Some care needs to be taken to ensure that one obtains a continuous function, as there are exceptional points where the function above has more than one maximizer. One consistent choice is to always choose the leftmost maximizer (see Definition \ref{def:b-geodesics} below). Since all of the semi-infinite geodesics in the directed landscape come from the Busemann process, we can define the \textbf{color} of a point $(x,t)$ as the color $\dir \in \Xs$ of the leftmost geodesic emanating from $(x,t)$. Note that, unlike the Busemann functions, the semi-infinite geodesics from an arbitrary eternal solution can have different colors across points $(x,t)$. For $b \in \es$, we define $\dirs^b$ to be the set of colors associated to $b$ across all points $(x,t) \in \R^2$. See \eqref{eq18} and \eqref{eq:colors} below for the precise definition.   For $\dir \in \dirs^b$, we let $\clr^b(\dir)$ be the subset of $\R^2$ consisting of points with color $\dir$. In the case where $b \equiv - \infty$, there are no such semi-infinite geodesics, and we set $\dirs^b = \varnothing$. We call the associated map $\R^2 \to \Xs$ the \textbf{coloring map}. Our next main result states that we can read off the colors associated to $b$ from the function $\varphi$. We first make the following definition. For this, recall the intervals $I_\ve(\dir)$ from \eqref{eq:o36}. 
\begin{definition} \label{def:dominating}
For $\varphi \in \fun$, we call a point $\dir\in\supp(\varphi)$ a (left) dominating color if there exists $\ve_0>0$ such that for every $\ve<\ve_0$
\begin{align}
     \varphi(\dir')& <\varphi(\dir), \qquad \forall \dir'\in I_\ve(\dir),\dir'< \dir\label{eq:o17}\\
      \varphi(\dir')& \leq \varphi(\dir), \qquad \forall \dir'\in I_\ve(\dir),\dir'\geq \dir\label{eq:o18}
\end{align}
Note that \eqref{eq:o17} (resp. \eqref{eq:o18}) is vacuously true for points in $\Xi^+$ (resp. $\Xi^-$).
We denote by $\Dom(\varphi)$  the set of dominating colors.  See Figure \ref{fig:function}. 
\end{definition}

\begin{example}
Let $a<b$ be elements in $\Xs$. Define the function 
\begin{equation}
\varphi(\dir)=
    \begin{cases}
        \mathcal{P}(\dir) & \dir\in [a,b]\\
        -\infty & \dir\in [a,b]^c
    \end{cases}
\end{equation}
The function $\varphi$ is in $\fun_\omega$, and 
\begin{equation}
    \Dom(\varphi)=\Xi^- \cap [a,b]
\end{equation}
\end{example}

\begin{theorem} \label{thm:cc}
    For almost-every $\omega$, the following hold. 
    \begin{enumerate} [label=(\roman*), font=\normalfont]
        \item For every $b \in \es_\omega$, the set $\dirs^b$ is at most countable. \label{it:dirs_count}
        \item For every $\varphi \in \fun_\omega$, $\dirs^{W \star \varphi} = \Dom(\varphi)$, and this set is dense in $\supp(\varphi)$. \label{it:dominating}
        \item \label{it:equals_p+W} For every $\varphi \in \fun_\omega$, $\dir \in \dirs^{W\star \varphi}$, and $(x,t) \in \clr^{W\star \varphi}(\dir)$,
        \[
        (W\star \varphi)(x,t) = \varphi(\dir) + W^\dir(0,0;x,t). 
        \]
    \end{enumerate}
\end{theorem}
Theorem \ref{thm:cc} is proved as separate items throughout the text.  Specifically, Item \ref{it:dirs_count} is shown as part of Proposition \ref{prop:exists_varphi}, Item \ref{it:dominating} is Lemma \ref{lm:o5}, and Item \ref{it:equals_p+W} is Corollary \ref{cor:o1}. These are then used as inputs in the proof of Theorem \ref{thm:mr}.

\begin{remark}
   The definition of dominating colors is quite delicate. The asymmetry in Definition \ref{def:dominating} of left dominating colors comes because we define the color of a point as the color of the leftmost semi-infinite geodesics. One can readily get an analogous definition of right dominating colors, which correspond to the colors of rightmost semi-infinite geodesics  by a symmetric proof. An illustrative example comes by considering the function in $\Phi$ defined by 
    \[
    \varphi(\dir) = \begin{cases}
        0 &\mathcal |P(\dir)| \le 1 \\
        -\infty &\text{otherwise}.
    \end{cases}
    \]
    With probability one, $\pm 1 \notin \Xi$, and we represent the elements $\pm 1^0 \in \Xi$ as  $\pm 1$ for shorthand. Then, the set of left dominating colors is the countable set $(\Xi^+ \cap [-1,1]) \cup \{-1\}$, while the set of right dominating colors is the set $(\Xi^- \cap [-1,1]) \cup \{1\}$. Thus, the set of colors of rightmost semi-infinite geodesics associated to $b := W\star \varphi$ is disjoint from the set of colors of leftmost semi-infinite geodesics associated to $b$. This is in sharp contrast to the setting of our previous work \cite{bhat-bus-soren} where we consider eternal solutions with two colors $\eta^-$ and $\eta^+$. In that case, the the cutoff points between the two colors may differ on some time horizons, but set of colors of leftmost and rightmost geodesics is always the same 
\end{remark}

\begin{figure}
    \centering
    \begin{tikzpicture}
     \begin{scope}
        \node[anchor=south west, inner sep=0] (image) at (0,0)
        { \includegraphics[page=5, scale=0.9, trim=100 300 100 350, clip]{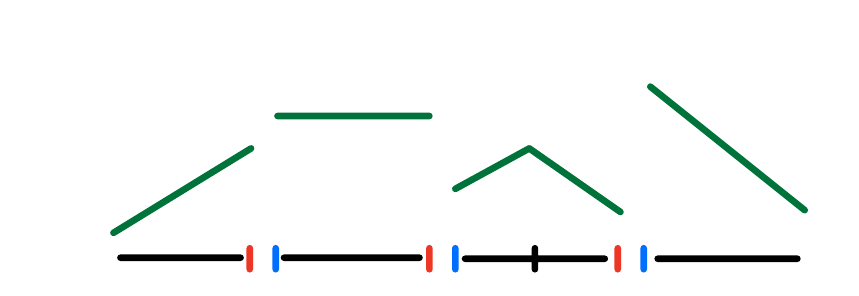}};

\node[red] at (3.8,-0.5) {$\eta^-_1$};
\node[blue] at (4.4,-0.5) {$\eta^+_1$};
\node[red] at (6.5,-0.5) {$\eta^-_2$};
\node[blue] at (7.1,-0.5) {$\eta^+_2$};
\node[red] at (9.4,-0.5) {$\eta^-_3$};
\node[blue] at (10,-0.5) {$\eta^+_3$};
\node[black] at (8.25,-0.5) {$\eta_4$};
        \end{scope}
    \end{tikzpicture}
    
    \caption{\small A conceptual graph of a continuous element $\varphi$ in $\Phi$ where we annotated 7 points in the support of $\varphi$. The colors $\eta_1^-,\eta_1^+,\eta_4$ and $\eta_3^+$ are all in $\Dom(\varphi)$ and will appear in the coloring map of $\hm[\varphi]$, while the colors $\eta_2^-,\eta_2^+,\eta_3^-$ are not dominating colors.}
    \label{fig:function}
\end{figure}

\subsection{Multiple coalescing shocks} \label{sec:mult_shocks}
It is a classical fact about the deterministic inviscid Burgers equation that shock solutions exist. The Burgers' equation gives a nice deterministic comparison to the stochastic models in the KPZ universality class because the spatial derivative of the KPZ equation formally solves the stochastic Burgers' equation. The shocks solutions in the inviscid Burgers' equation consist of two constant solutions that are separated by a discontinuity. When integrated out, the solution forms the shape of the capital letter $V$, as we have two linear functions with a discontinuity at the bottom of the $V$. On the level of the KPZ fixed point, these correspond to two noisy linear functions. Shocks for the KPZ equation and KPZ fixed point have been studied recently in \cite{Dunlap-Ryzhik-2020,Rahman-Virag-21,Dunlap-Sorensen-2024,Dunlap-Sorensen-2025} (see also \cite{Ferrari-Fontes-1994b,Ferrari-Ghosal-Nejjar-2019} for related works in particle systems).

It is also well-known that the Burgers' equation can have solutions with multiple coalescing shocks. See also \cite{Ferrari-Nejjar-2020} for discussion of an analogous phenomenon in the totally asymmetric simple exclusion process. As mentioned above, the description of the eternal solutions in Theorem \ref{thm:mr} may be thought of as a bi-infinite collection of shock solutions. For an eternal solution, each point in the plane corresponds to a backwards infinite geodesic. Each of those geodesics has a direction corresponding to one of the parameters of the Busemann process (See Section \ref{sec:SIG_Buse}). Monotonicity of the semi-infinite geodesics across different directions implies that the plane is partitioned into disjoint connected sets corresponding to the directions of the semi-infinite geodesics (more specifically, the color of the semi-infinite geodesic, described more precisely in Section \ref{sec:SIG_Buse}). Within each connected component, the semi-infinite geodesics coalesce. The borders between these directions are the shocks in this setting. Our next main theorem describes this family of shocks as the interfaces between colors. Refer to Figure \ref{fig:color_map_intro} for a visual representation.
\begin{figure}
    \centering
    \includegraphics[width=0.5\linewidth]{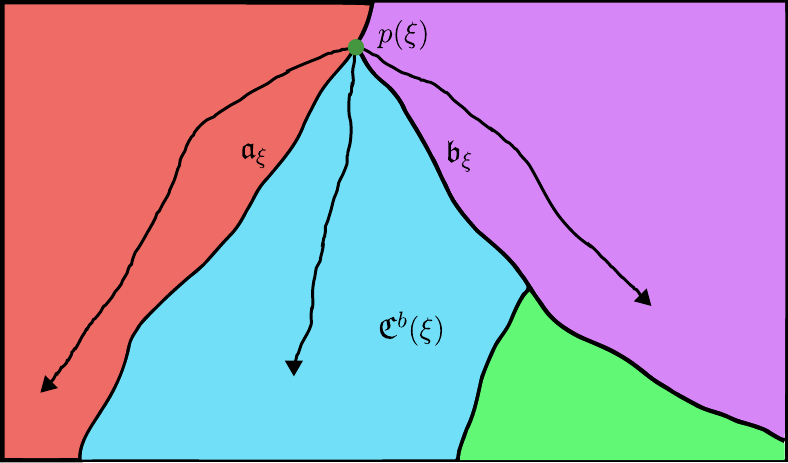}
    \caption{\small A depiction of the coloring map. Each color represents the set $\mathfrak C^b(\xi)$ for a different color $\xi$. In this figure, the labeled set $\mathfrak C^b(\xi)$ is in blue. We see the point $p(\xi)$ where the blue color goes extinct, as well as the three disjoint semi-infinite geodesics emanating from it. The paths $\ls_\xi$ and $\rs_\xi$ denote the borders between colors. }
    \label{fig:color_map_intro}
\end{figure}

\begin{theorem} \label{thm:shocks_coal}
For almost-every $\omega$, the following holds for all $b \in \es_\omega$. 
\begin{enumerate} [label=(\roman*), font=\normalfont] 
    \item \label{it:Cb_set} For each $\xi \in \dirs^b$, the set $\clr^b(\dir)$ is connected and unbounded. Specifically, we may write
    \be \label{eq:Cb_intro}
    \clr^b(\dir)=\bigcup_{s<\tex(\dir)} (\ls^s_\dir,\rs^s_{\dir}]\times \{s\},
    \ee
    for numbers $\ls^s_\dir \in [-\infty,\infty)$ and $\rs_s^\dir \in (-\infty,\infty]$ with $\ls^s_\dir<\rs^s_{\dir}$ for all $s < \tex(\dir)$, and $\tex(\dir) \in (-\infty,\infty]$ is called the \textbf{extinction time} for the color $\xi \in \dirs^b$. If $\ls_\xi^{s'} > -\infty$ for some $\xi \in \dirs^b$ and $s' <\tex(\dir)$, then $\ls_\dir^s > -\infty$ for all $s <\tex(\dir)$, and the function $s \mapsto \ls_\dir^s$ is continuous. Similarly, if $\rs_\dir^{s'} < \infty$, for some $s' <\tex(\dir)$, then $\rs_\dir^s < \infty$ for all $s <\tex(\dir)$, and $s \mapsto \rs_\dir^s$ is continuous.

     \item \label{it:ext_pt} For $\dir \in \dirs^b$, $\tex(\dir)<\infty$  if and only if $\ls_\dir^s > -\infty$ and $\rs_\dir^s < \infty$ for some $s \in \R$. That is, a color survives forward in time if and only if it is not bounded between two other colors. If this is the case, the following limits exist and are equal:
    \be \label{eq:extinction}
    \xex(\xi) :=\lim_{s\uparrow \tex}\ls_\dir^s=\lim_{s\uparrow \tex}\rs_\dir^s
    \ee
    We call $\pex(\xi) :=  (\xex(\xi),\tex(\xi))$ the \textbf{extinction point} of the color $\xi$. 
  
    \item \label{it:locally_finite} No two colors share the same extinction point. From every extinction point, there exists three disjoint semi-infinite geodesics \rm{(}i.e., they are distinct except at the initial point\rm{)}. Furthermore, the set of extinction points is locally finite. 
\end{enumerate}
\end{theorem}
\begin{remark}
    The paths $s \mapsto \ls_\xi^s$ and $s \mapsto \rs_\xi^s$  are borders between different colors and therefore can be interpreted as shocks. These shocks are also competition interfaces between two colors (see Lemma \ref{lemma 4.1}). Item \ref{it:ext_pt} states that at most two colors can survive forever forward in time. By definition of extinction points in Item \ref{it:ext_pt} and the distinctness of extinction points for different colors in Item \ref{it:locally_finite}, when a color goes extinct, this corresponds to two interfaces coalescing. Two colors cannot have the same extinction point because the same proof would then show that there are $4$ disjoint semi-infinite geodesics from that extinction point, a contradiction to a result in \cite{Dauvergne-23}.   See Figure \ref{fig:color_map_intro}. 
    
\end{remark}

\begin{remark}
    It is helpful to make a comparison to a deterministic equation. For each $\xi \in \R$, the function $w^\dir(x,t): \xi^2 t + 2\xi x$ is an eternal solution to the Hamilton-Jacobi PDE
    \[
    \partial_t h(x,t) = \f{1}{4}\Bigl(\partial_x h(x,t)\Bigr)^2,
    \]
    which is solved from a general initial condition $f$ at time $s$ as 
    \[
    h(x,t) = \sup_{y \in \R}\Bigl[f(y) - \f{(x-y)^2}{t-s}\Bigl].
    \]
    Thus, by Lemma \ref{lem:Landscape_global_bound}, we may think of the KPZ fixed point as a noisy version of this equation.

    Consider the function $\varphi:\R \to [-\infty,\infty)$ such that $\varphi(\xi) = 0$ for $\xi \in [-1,1]$, and $\varphi(\xi) = -\infty$ otherwise. Then, consider the following eternal solution to the deterministic equation:
    \[
    b(x,t) = \sup_{\xi \in \R}[\varphi(\xi) +w^\dir(x,t) ] = \begin{cases}
        w^{-1}(x,t) &x \le t \wedge 0 \\
        w^{1}(x,t) &x \ge (-t) \vee 0 \\
        w^\dir(x,t) & t < 0, \dir = -\f{x}{t} \in [-1,1]
    \end{cases}
    \]
    From here, we see that all colors in $[-1,1]$ are present in the coloring map, a distinct behavior from the countable set of colors in Theorem \ref{thm:shocks_coal}. We also see that all colors in the interval $(-1,1)$ go extinct at the space-time point $(0,0)$; in the noisy case, Item \ref{it:locally_finite} states that no two colors go extinct at the same point. However, we do see a qualitative similarity in that only the two extremal colors, $-1$ and $1$ survive as $t \to \infty$ (compare to Item \ref{it:ext_pt}).  
\end{remark}

\subsection{Methods, contribution, and related work}
The previous works \cite{Akian-Gaubert-Walsh-2005,Walsh-2009} developed similar spectral decompositions via Busemann functions for discrete time and space evolutions in the max-plus algebra. Less than a week before the posting of the first version of this paper, the article \cite{RassoulAghaSweeney2026} was posted on arXiv. The authors obtained, among other results, a result similar to our Proposition \ref{prop:exists_varphi} below. In particular, they showed that all eternal solutions of the KPZ fixed point are of the form $W \star \varphi$ for a function $\varphi$ with countable support. In both works, this is an intermediate result that leads to different stronger consequences.  The work \cite{Rassoul-Sweeney-2026} uses this result in the context of the Martin boundary of the directed landscape and connects to results in hyperbolic geometry and potential theory. By contrast, our Theorem \ref{thm:mr} is new in that it completely characterizes the space of eternal solutions $\es$ with its spectral decomposition $\fun$ and shows that the bijection is a homeomorphism.

    Some preliminaries are presented in Section \ref{sec:SIG_Buse}, and then 
    Proposition \ref{prop:exists_varphi} is proved in Section \ref{sec:decomp}. The main novelty and technical innovation then comes in the later sections. To classify the set of eternal solutions, one first needs to understand the set of functions $\varphi$ such that $W\star \varphi < \infty$. This is handled in Section \ref{sec:finiteness} by developing control on the Busemann functions across all $\dir \in \Xs$ (Lemma \ref{lem:Buse_global}), then comparing this to the shape of the directed landscape (Lemma \ref{lem:Landscape_global_bound}). Proposition \ref{prop:Wstarphi_bd} identifies the largest set of functions $\mathcal D$ such that $W\star (\mathcal D) = \es$.

    However, we can see from Remark \ref{rmk:non_unique} that uniqueness of the representation $W\star \varphi$ does not hold across $\varphi \in \mathcal D$. Section \ref{sec:unique} identifies the appropriate subset $\Phi \subseteq \mathcal D$ on which $\varphi \mapsto W\star \varphi$ is a bijection onto $\es$. After establishing some preliminaries about the space $\Phi$, Corollary \ref{cor:o1} shows that any function $\varphi \in \Phi$ such that $b = W\star \varphi$ must be an extension of the function with countable support from Proposition \ref{prop:exists_varphi}. There is a standard way to give such an extension (see Lemma \ref{lem:one_extension}), but it remains to show that no other extension exists. Then, Lemmas \ref{lm:o8}-\ref{lm:o5} give a comparative analysis  of competition interfaces across colors and points and introduce the notion of dominating colors. This is used to prove the uniqueness in Proposition \ref{prop:uniq}. This uniqueness is then used in Proposition \ref{prop:Legendre} to show the representation of $\varphi$ in \eqref{eq:hom_inv}. Finally, in Section \ref{sec:main_proof}, we characterize the topology on the space $\fun$ and prove that $\Wf$ and its inverse are continuous. The characterization of the decay rate required for $W\star \varphi$ in Proposition \ref{prop:Wstarphi_bd} informs the choice of topology, as we need to impose a uniform bound on $\f{\varphi_n(\dir)}{\mathcal P(\dir)^2}$ outside a compact window in defining the notion of convergence $\varphi_n \to \varphi$. 
     
    Regarding our two other results above, Theorem \ref{thm:cc}\ref{it:dominating} is new.  The continuity of the interfaces (Theorem \ref{thm:shocks_coal}\ref{it:Cb_set}) and the fact that no  two colors go extinct at the same time (Theorem \ref{thm:shocks_coal}\ref{it:locally_finite}) is also new. The proof of Theorem \ref{thm:shocks_coal} is geometric in nature and can be read independently of the other parts of the paper. It is proved in Section \ref{sec:coal_proofs}. As mentioned previously, Theorem \ref{thm:cc} is proved in several results throughout the paper.

\subsection{Acknowledgments}
The first version of this article was completed while E.S. was affiliated with Columbia University, where he was partially supported by Ivan Corwin's Simons Investigator Grant \#929852. Part of this work was completed while S.B. and E.S. were visiting the workshop on stochastic interacting particle systems and random matrices at the R\'enyi institute in Budapest, Hungary in June 2025. The travel of E.S. during that trip was supported by AMS-Simons travel grant AMMS CU23-2401. The travel of S.B. during that trip was funded by Indian Institute of Science (IISc) GARP support. Other parts of this work were completed during a visit by S.B. to Columbia University, where she was supported by Ivan Corwin's Simons Investigator Grant \#929852. A significant portion of this work was done during a visit by S.B. to the University of Edinburgh. The visit was supported by O.B.'s  ERC Starting Grant UnivKPZ (grant agreement No.101219600). S.B. was partially supported by scholarship from National Board for Higher Mathematics (NBHM) (ref no: 0203/13(32)/2021-R\&D-II/13158). S.B. was also supported by EPSRC grant EP/W032112/1 Standard Grant of the UK. S.B. thanks Riddhipratim Basu for discussion during the initial part of the project. O.B. acknowledges support from the European Research Council (ERC) under the European Union’s Horizon Europe research and innovation programme through the ERC Starting Grant UnivKPZ (grant agreement No.101219600). O.B. Would like to thank Duncan Dauvergne for a helpful discussion on the existence of three stars and infinite geodesics in the directed landscape. We also thank Alexander Dunlap, Firas Rassoul-Agha, and Mikhail Sweeney for helpful comments and discussions of an earlier version of the manuscript. OpenAI's Chat GPT 5.6 gave the idea for Remark \ref{rmk:topology} to show that Definition \ref{def:fun_convergence} indeed defines a topology, but the formulation of the notion of convergence in Definition \ref{def:fun_convergence}, as well as all other content and writing of the paper, came from the authors.

\section{Semi-infinite geodesics and Busemann functions}  \label{sec:SIG_Buse}   

In the present paper, we make use of the Busemann process and their associated geodesics. We record several  key preliminary facts in this section.

\subsection{Directed landscape}

The directed landscape may be thought of as an anti-metric in the plane. Indeed, it satisfies the reverse triangle inequality:
\be \label{triangle}
\Ll(x,s;y,t) \ge \Ll(x,s;z,r) + \Ll(z,r;y,t).
\ee
For a continuous function $\gamma:[s,t] \to \R$, we define its $\Ll$-length as
\[
\Ll(\gamma) = \inf_{k \in \Z_{>0}} \; \inf_{s = t_0 < t_1 < \cdots < t_k = t} \sum_{i = 1}^k \Ll\bigl(\gamma(t_{i - 1}),t_{i - 1};\gamma(t_i),t_i\bigr).
\]
A path $\gamma$ is called a \textbf{point-to-point geodesic} from $(x,s)$ to $(y,t)$ if its $\Ll$-length is maximal among all paths $\gamma:[s,t] \to \R$ with $\gamma(s) = x$ and $\gamma(t) = y$. Equivalently, 
\[
\Ll(\gamma) = \sum_{i = 1}^k \Ll\bigl(\gamma(t_{i - 1}),t_{i - 1};\gamma(t_i),t_i\bigr)
\]
for all $k \in \Z_{>0}$ and partitions $s = t_0 < t_1 < \cdots < t_k = t$. In particular, when $s < r <t$, we have  $\gamma(r) = z$ for some geodesic from $(x,s)$ to $(y,t)$ if and only if equality holds in \eqref{triangle}. Sometimes, when referring to the geodesic, we will identify it with its path in the plane $(\gamma(r),r)_{r \in [s,t]}$.

We will use the following fact about geodesics, which will refer to as the property that geodesics do not form interior bubbles. 
\begin{lemma}\cite[Theorem 1]{Bha24}, \cite[Lemma 3.3]{Dauvergne-23} \label{lem:no_bubbles}
With probability one, there exist no points $(x,s;y,t) \in \Rup$ and distinct geodesics $\gamma_1,\gamma_2$ from $(x,s)$ to $(y,t)$, such that, for some $\delta > 0$, $\gamma_1$ and $\gamma_2$ agree on the set $[s,s+\delta] \cap [t-\delta,t]$.
 \end{lemma}

A \textbf{semi-infinite geodesic} rooted at the point $(x,t) \in \R^2$ is a continuous function $g:(-\infty,t] \to \R$ such that $g(t) = x$ and the restriction $g|_{[s,t]}$ is a point-to-point geodesic for any $s < t$. It was shown in \cite[Theorem 2.5]{Busa-Sepp-Sore-22a} that each semi-infinite geodesic has a direction. That is, 
\[
\text{The limit }\lim_{s \to -\infty} \f{g(s)}{|s|}\quad\text{exists in }\R.
\]
We say that two semi-infinite geodesics $g_1$ and $g_2$ \textbf{coalesce} if there exists $S \in \R$ such that, for all $s \le S$, $g_1(s) = g_2(s)$. It was shown in \cite[Theorem 3.19]{Rahman-Virag-21} that for any fixed direction, with probability one, all semi-infinite geodesics in that direction coalesce. In \cite{Busa-Sepp-Sore-22a}, it was further shown that there is a random exceptional set $\Xi$ of directions such that there are two coalescing families of semi-infinite geodesics, which we call the $-$ geodesics and the $+$ geodesics. Then, in \cite{Busani_N3G}, it was shown that in these directions, there are exactly two families of coalescing geodesics. These geodesics are then indexed by their \textbf{colors} in $\Xs$: each geodesic with color $\xi = \eta^\sig$ has a direction $\eta\in \R$ and a sign $\sigg \in \{-,0,+\}$. The colors in $\Xs$ index both the Busemann process and the semi-infinite geodesics. To discuss this connection, we first cite the following result from \cite{Busa-Sepp-Sore-22a}:
\begin{proposition} \cite[Theorem 5.1, Theorem 5.3, Lemma 5.12]{Busa-Sepp-Sore-22a} \label{prop:Buse_basic_properties}
On the space $(\Omega,\F,\Pp)$, there exists a process
\[
\Bigl(W^{\dir}(x,s;y,t): \dir \in \Xs, (x,s;y,t) \in \R^4\Bigr)
\]
such that for each fixed $t \in \R$ and $\dir \in \Xs$, $\dir \in \Xi^c$, and $x \mapsto W^\dir(0,t;x,t)$ is a Brownian motion with diffusion coefficient $\sqrt 2$ and drift $2\mathcal P(\dir)$. Furthermore, the process satisfies the following properties on a single event of probability one:
\begin{enumerate}[label=(\roman*), font=\normalfont]
\item{\rm(Continuity)} \label{itm:general_cts} The $\R^4 \times \Xs \to \R$ function  $(x,s;y,t;\dir) \mapsto \W^{\dir}(x,s;y,t)$ is  continuous. Furthermore, if $K \subseteq \R^4$ is a compact set, then for all $\dir \in \Xs$ and sequences $\dir_n \to \dir$, there exists $N \in \N$ so that for all $n \ge N$, we have 
\[
W^{\dir_n}(x,s;y,t) = W^{\dir}(x,s;y,t),\quad \text{for all }(x,s;y,t) \in K.
\]
 \item {\rm(Additivity)} \label{itm:DL_Buse_add} For all $\dir \in \Xs$ and $p,q,r \in \R^2$, 
    $\W^{\dir}(p;q) + \W^{\dir}(q;r) = \W^{\dir}(p;r)$.   In particular, \\ $\W^{\dir}(p;q) = -\W^{\dir}(q;p)$ and $\W^{\dir}(p;p) = 0$.
    \item {\rm(Monotonicity along a horizontal line)}
    \label{itm:DL_Buse_gen_mont} Whenever $\dir_1< \dir_2$ in $\Xs$, $x < y$, and $t \in \R$,
    \[
  \W^{\dir_1}(x,t;y,t) \le \W^{\dir_2}(x,t;y,t).
    \]
    \item \rm{(Divergence to $\pm \infty$)} \label{it:to_pm_inf}
    For all $\dir_1 < \dir_2$ and $t \in \R$,
    \[
\lim_{x \to \pm \infty} W^{\dir_2} (0,0;x,t) -W^{\dir_1}(0,0;x,t) = \pm \infty.
\]
    \item {\rm(Eternal solution)}\label{itm:Buse_KPZ_description} For 
    all $\dir \in \Xs$, $x,y \in \R$, and $s < t$,
    \be\label{W_var}
    \W^{\dir}(x,s;y,t) = \sup_{z \in \R}\{\W^{\dir}(x,s;z,s) +
 \Ll(z,s;y,t)\}.
    \ee
\end{enumerate}
\end{proposition}

In the following, we will often refer to the properties of \textbf{continuity, additivity, monotonicity, and eternal solution}, without specific reference to the proposition. 

\begin{remark}
    In \cite{Busa-Sepp-Sore-22a}, the Busemann process is indexed by $\dir^\sig$ for $\dir \in \R$ and $\sig \in \{-,+\}$. In this sense, $W^{\dir-}$ is left-continuous, and $W^{\dir+}$ is right-continuous. However, because of our choice of topology of the index set $\Xi$ (see Remark \ref{rmk:convergence}), the Busemann process becomes continuous as a function of $\dir \in \Xs$.  
\end{remark}

For a given point $(x,t)$ and a color $\dir \in \Xs$, there are two particularly relevant geodesics corresponding to the Busemann function $W^{\dir}$. We denote these as $g_{(x,t)}^{\dir,L}$ and $g_{(x,t)}^{\dir,R}$. They are defined as follows:
\begin{definition} \label{def:SIG_xi}
For $(x,t) \in \R^2$ and $\dir \in \Xs$, define the functions $g_{(x,t)}^{\dir \sig,L}, g_{(x,t)}^{\dir \sig,R}:(-\infty,0] \to \R$ as follows: For $s < t$, let $g_{(x,t)}^{\dir,L/R}(s)$ be the leftmost/rightmost maximizers of the function
\[
y \mapsto W^{\dir \sig}(0,0;y,s) + \Ll(y,s;x,t) \quad\text{over } y \in \R,
\]
and for $s = t$, we define $g_{(x,t)}^{\dir,L/R}(t) = x$. 
\end{definition}
It is shown in \cite[Theorem 5.9]{Busa-Sepp-Sore-22a} that each of these is a semi-infinite geodesic with direction $\xi$ (possibly with a $\pm$ distinction). A key fact about these geodesics is that the weight of $\Ll$ along the geodesics is equal to the Busemann function:
\begin{lemma} \label{lem:W=landscape}\cite[Lemma 5.9]{Busa-Sepp-Sore-22a}
    For all $\dir \in \Xs$ and $s_1 < s_2 \le t$, we have that 
    \begin{equation} \label{L =landscape}
\Ll(y_1,s_1;y_2;s_2) = W^\dir(y_1,s_1;y_2,s_2), \quad\text{ where }  y_i = g_{(x,t)}^{\dir,L/R}(s_i).
\end{equation}
In general, for $\dir \in \Xs$, $s_1 < s_2$, and $y_1,y_2 \in \R$,
\be \label{eq:landscape_leq}
\Ll(y_1,s_1;y_2;s_2) \le  W^\dir(y_1,s_1;y_2,s_2)
\ee
\end{lemma}

Here, we have constructed the leftmost and rightmost semi-infinite geodesics associated to a particular Busemann function. For a fixed initial point, the leftmost and rightmost geodesics agree, but there is a random set where the geodesics disagree near the root point, to coalesce later \cite{Busa-Sepp-Sore-22a,Bha24}. There is also a smaller random set with three (locally distinct) semi-infinite geodesics having the same color.  (See, for example, \cite{Bha24,Rassoul-Sweeney-2026}). In general, it is shown in \cite[Theorem 5.9]{Busa-Sepp-Sore-22a} that continuous functions that satisfy \eqref{L =landscape} for all $s_1 < s_2 \le t$ are semi-infinite geodesics, and we call these $\dir$-geodesics. 
\begin{lemma} \cite[Lemma 7.1(i)]{Busa-Sepp-Sore-22a} \label{lem:all_coalesce}
    On a single event of probability one, for all $\dir \in \Xs$, all $\dir$-geodesics coalesce. 
\end{lemma}

One can also build semi-infinite geodesics from a general finite eternal solution, which we define now.  
    \begin{definition} \label{def:b-geodesics}
             If $b$ is a finite eternal solution, then for $(x,t) \in \R^2$, define the functions $g_{(x,t)}^{b,L},g_{(x,t)}^{b,R}:(-\infty,t] \to \R$ as follows: For $s < t$, let $g_{(x,t)}^{b,L/R}(s)$ be the leftmost/rightmost maximizers of the function
\be \label{eq:b+L_maxes}
y \mapsto b(y,s)+\Ll(y,s;x,t)
\ee
over $y \in \R$. For $s = t$, we define $g_{(x,t)}^{b,L/R}(t) = x$. 
    \end{definition}
\begin{remark}
    Corollary \ref{cor:cts_bounded} shows that $(x,t) \mapsto b(x,t)$ is continuous for any finite eternal solution $b$ and that the set of maximizers of the function \eqref{eq:b+L_maxes} is nonempty and bounded (see also \cite[Theorems 3.6-3.7]{Rassoul-Sweeney-2026} for an alternative proof). By continuity, any limit point of a maximizer is a maximizer, so the quantities $g_{(x,t)}^{b,L/R}(s)$ are well-defined and real-valued. 
\end{remark}

    Similarly as for the geodesics associated to the Busemann functions, the analogous equality to \eqref{L =landscape} holds (see \cite[Lemma A.14]{bhat-bus-soren})
    \be \label{b=Landscape}
    \Ll(y_1,s_1;y_2,s_2) = b(y_2,s_2) - b(y_1,s_1),\quad\text{for all }s_1 < s_2 \le t, \quad\text{ where }  y_i = g_{(x,t)}^{b,L/R}(s_i).
    \ee
    There may be some initial points where the eternal solution $b$ exhibits semi-infinite geodesics other than the leftmost and rightmost geodesics. In general, we call continuous functions $g:(-\infty,t] \to \R$ satisfying \eqref{b=Landscape} \textbf{$b$-geodesics}. By Lemma \ref{lem:geodesics_from_b} (originally from \cite{bhat-bus-soren}), with probability one, all $b$-geodesics are semi-infinite geodesics and hence are  continuous functions $(-\infty,t] \to \R$. 

We collect the following relevant facts:
\begin{lemma} \label{lem:SIG_facts_collect}
    The following hold on an event of probability one:
    \begin{enumerate}[label=(\roman*), font=\normalfont]
        \item \label{it:xi_g_mont} \cite[Theorem 6.3(i)]{Busa-Sepp-Sore-22a} For $\dir_1 < \dir_2$ in $\Xs$, $s < t$,  $x \in \R$, and $S \in \{L,R\}$,
        \[
        g_{(x,t)}^{\dir_1,S}(s) \le g_{(x,t)}^{\dir_2,S}(s).
        \]
        \item \label{it:some_g} \cite[Lemma 7.1]{bhat-bus-soren}  \rm{(}see also \cite{Busani_N3G}\rm{)} If $b$ is a finite eternal solution, then for each $(x,t) \in \R^2$ and $S \in \{L,R\}$, $g_{(x,t)}^{b,S} = g_{(x,t)}^{\dir,S}$ for some $\dir \in \Xs$.
        \item \label{it:g_b_mont} \cite[Corollary A.15]{bhat-bus-soren} If $b$ is a finite eternal solution, then for $x < y$ and $t \in \R$, we have $g_{(x,t)}^{b,S}(s) \le g_{(y,t)}^{b,S}(s)$ for all $x < y$.
        \item \cite[Corollary A.16]{bhat-bus-soren}\label{it:g_b_lim} If $b$ is a finite eternal solution, then for all $s < t$ and $x \in \R$,
        \[
        \lim_{y \nearrow x} g_{(y,t)}^{b,L}(s) = g_{(x,t)}^{b,L}(s),\quad\text{and}\quad \lim_{y \searrow x} g_{(y,t)}^{b,R}(s) = g_{(x,t)}^{b,R}(s).
        \]
        \item \cite[Theorem 6.3(v)]{Busa-Sepp-Sore-22a}\label{it:limits_to_infty} For all $\dir \in \Xs$ and $S \in \{L,R\}$,  
        \[
        \lim_{x \rightarrow \pm \infty} g_{(x,t)}^{\dir, S}(s)=\pm \infty.
        \]
        In particular, Items \ref{it:some_g} and \ref{it:g_b_mont} imply that  for every finite eternal solution $b$,
        \[
        \lim_{x \rightarrow \pm \infty} g_{(x,t)}^{b, S}(s)=\pm \infty.
        \]
    \end{enumerate}
\end{lemma}

     \begin{definition} \label{def:color}
         For a finite eternal solution $b$ and $(x,t) \in \R^2$, define  $\dir^{b}_{x,t} \in \Xs $ to be the color of $g_{(x,t)}^{b,L}$.
     \end{definition}
     We  now prove the following result.
\begin{lemma}\label{lem:5}
     The following holds for almost every $\omega$: for all finite eternal solutions $b$ and for all $t$, the function $x \mapsto \dir^b_{x,t}$ is non-decreasing and left continuous. 
\end{lemma}
\begin{proof}
    By Lemma \ref{lem:SIG_facts_collect}\ref{it:xi_g_mont},\ref{it:g_b_mont}, $x \mapsto \xi_{x,t}^b$ is non-decreasing. We now show left continuity at each $x$. By Lemma \ref{lem:SIG_facts_collect}\ref{it:g_b_lim}, we have $\lim_{y \nearrow x} g_{(y,t)}^{b,L}(s) = g_{(x,t)}^{b,L}(s)$ for all $s < t$  and $x \in \R$. Since geodesics do not form interior bubbles, we know that $g_{(x,t)}^{b,L}|_{[t-3,t-1]}$ is the unique geodesic between its points. Then, by \cite[Lemma 3.3]{Dauvergne-Sarkar-Virag-2022}, we must have that for $y < x$ sufficiently close to $x$, we have $g_{(y,t)}^{b,L}(t-2) = g_{(x,t)}^{b,L}(t-2)$. Hence, since these leftmost geodesics coalesce if they meet (Lemma \ref{lem:geodesics_from_b}\ref{itm: consis}), we have that $\xi_{y,t}^{b,L} = \xi_{x,t}^{b,L}$. But since $y \mapsto \xi_{y,t}^{b,L}$ is nondecreasing, we must have $\lim_{y \nearrow x} \xi_{y,t}^{b,L} = \xi_{x,t}^{b,L}$. 
\end{proof}

For an eternal solution $b$ and each $t\in\R$ define 
\begin{equation}\label{eq18}
\dirs^{b,t}=\left \{\dir^{b}_{x,t}: x \in \R \right \}.    
\end{equation}
When $b \equiv -\infty$, we have $\dirs^{b,t} = \varnothing$.

 From  Lemma \ref{lem:5}, we see that, whenever $b$ is a finite eternal solution,
\begin{equation}
\label{eq:s2}
  \dir^{b}_{x,t}=\sum_{\dir\in \dirs^{b,t}}\dir\,1_{(\ls^t_\dir,\rs^t_{\dir}]},
\end{equation}
where $-\infty \le \ls^t_\dir <\rs^t_{\dir} \le \infty$.

 \begin{lemma}\label{lem:2}
        For all $t\in\R$, the process $x\mapsto \dir^{b}_{x,t}$ is a jump process with only finitely many jumps in each compact interval.
    \end{lemma}
\begin{proof}
    We first break $\R$ into compact intervals $K_n$. Let us consider all leftmost geodesics starting from $(x,t)$ where $x \in K_n.$ We consider the set of intersection points $K_{n,1}:=\{g^{b,L}_{(x,t)}(t-1): x \in K_n \}.$ Clearly, $K_{n,1}$ is a bounded set. Further, define $K_{n,3}:=\{g^{b,L}_{(x,t)}(t-3): x \in K_n \}.$ Then $K_{n,3}$ is also a bounded set and as geodesics do not form bubbles \cite[Theorem 1]{Bha24}, for all $x \in K_n, g^{b.L}_{(x,t)}|_{[t-3,t-1]}$ is a unique geodesic between $g^{b,L}_{(x,t)}(t-3)$ and $g^{b,L}_{(x,t)}(t-1)$. Now as both $K_{n,1}$ and $K_{n,3}$ are bounded sets, by \cite[Lemma B.14]{Busa-Sepp-Sore-22a} we know that $K_{n,2}:=\{g^{b,L}_{(x,t)}(t-2): x \in K_n \}$ is a finite set. This implies that all leftmost $b$ geodesics starting from points of $K_n$ have coalesced into finitely many points by time $t-2.$ Further, once two leftmost $b$ geodesics meet, they do not separate. Thus $\left \{\dir^{b}_{x,t}: x \in \R  \right \}$ is countable and locally finite.
\end{proof}

The following lemma will be useful.
\begin{lemma} \label{lem:Buse=W}
    If $b$ is an eternal solution, and $(x,t),(y,s) \in \R^2$ satisfy $\xi_{x,t}^b = \xi_{y,s}^b = \xi$, then 
    \[
    b(x,t) - b(y,s) = W^\dir(y,s;x,t).
    \]
\end{lemma}
\begin{proof}
    By assumption, $g_{(x,t)}^{b,L} = g_{(x,t)}^{\dir,L}$ and $g_{(y,s)}^{b,L} = g_{(y,s)}^{\dir,L}$ have the same color $\xi$, so they coalesce by Lemma \ref{lem:all_coalesce}. Let $p \in \R^2$ be a common point along both geodesics. Then, since $b$ is an eternal solution and the locations of geodesics are maximizers by definition, we have 
    \begin{align*}
        b(x,t) = b(p) + \Ll(p;x,t),\quad\text{and}\quad b(y,s) = b(p) + \Ll(p;y,s).
    \end{align*}
    Next, by the equality of the directed landscape to the Busemann functions along a geodesic \eqref{L =landscape}, we have 
    \[
    b(x,t) = b(p) + W^\dir(p;x,t),\quad\text{and}\quad b(y,s) = b(p) + W^\dir(p;y,s).
    \]
    Subtracting the two equations and using the additivity of the Busemann function completes the proof. 
\end{proof}

\section{Decomposition of eternal solutions by colors} \label{sec:decomp}
In this section, we will show that every eternal solution can be represented as $W\star \varphi$ for a suitable function $\varphi$ (Proposition \ref{prop:exists_varphi} below). The uniqueness will come later in Section \ref{sec:unique}.

\subsection{KPZ evolution in a half-plane}
In this subsection, we describe how eternal solutions starting from a fixed time level $s$ evolves forward in time. We start with the following general result:
\begin{lemma} \label{lem:interchange_sups}
    Let $(f_\alpha)_{\alpha \in \mathcal I}$ be any indexed collection of functions $f_\alpha:\R \to [-\infty,\infty)$, and define the function $f:\R \to [-\infty,\infty]$ by
    \[
    f(y) = \sup_{\alpha \in I} f_\alpha(y).
    \]
    Then, for all $t > s$ and $x \in \R$, we have 
    \be \label{eq:interchange_sup}
    \kpzs_t(x;f,s) = \sup_{\alpha \in \mathcal I} \kpzs_t(x;f_\alpha,s).
    \ee
    Consequently, if $\varphi:\Xs \to [-\infty,\infty)$ is such that $b(x,t) := (W \star \varphi)(x,t) < \infty$ for all $(x,t) \in \R^2$, then $b$ is an eternal solution. 
\end{lemma}
\begin{proof}
    Equation \eqref{eq:interchange_sup} follows immediately from the definition of the KPZ fixed point \eqref{eq:KPZFP_DL} and by changing the order of the supremum (note that both sides of the equality may be $\infty$ for this general result). Next,  let $b := W\star \varphi$ for some $\varphi:\Xs \to [-\infty,\infty)$. Then, since each $(x,t) \mapsto W^\dir(0,0;x,t)$ is an eternal solution (Proposition \ref{prop:Buse_basic_properties}\ref{itm:Buse_KPZ_description}), we have that, for $s < t$ and $x \in \R$,
    \begin{align*}
    b(x,t) &= \sup_{\dir \in \Xs}[\varphi(\xi) + W^\dir(0,0;x,t)]  \\
    &= \sup_{\dir \in \Xs}\bigl[\varphi(\xi) + \sup_{y \in \R}[W^\dir(0,0;y,s) + \Ll(y,s;x,t)]\bigr] = \sup_{y \in \R}[b(y,s) +\Ll(y,s;x,t)], 
    \end{align*}
    where we again changed the order of the supremum. Hence, if $b(x,t) < \infty$ for all $(x,t) \in \R^2$, then $b$ is an eternal solution, as desired. 
\end{proof}

For $s \in \R$, recall $\dirs^{b,s}$ and $\bigl\{\ls_{\dir}^{s}\bigr \}_{\dir \in \dirs^{b,s}}$ and $\bigl \{\rs^s_{\dir}\bigr \}_{\dir \in \dirs^{b,s}}$, as defined in \eqref{eq18} and \eqref{eq:s2}. By Lemma \ref{lem:2}, $x \mapsto \dir^b_{x,s}$ is a discrete jump process. As such, we order the elements of $\dirs^{b,s}$ as follows.

Assume that $\dirs^{b,s}$ contains at least two colors for some $s \in \R.$ We choose and fix $\dir_0 \in \dirs^{b,s}$ such that $\ls^s_{\dir_0} \in \R.$ For $i \in \Z_{\ge 0}$, having defined $\dir_i^s$, we inductively define $\dir_{i+1}^s$ to be the unique $\dir' \in \dirs^{b,s}$ such that $\ls_{\dir'}^{s} = \rs_{\dir_i^s}^s$, whenever such a $\dir'$ exists. If $\mathfrak b_{\xi_i^s}^s = \infty$, then no such $\dir'$ exists, and we terminate the procedure. In this case, for the next two lemmas we shall formally write $\ls_{\xi_{i+1}^s}^s = \infty$, even though we have not defined $\xi_{i+1}^s$. 
 Similarly, for $i\in \Z_{\le 0}$, having defined $\dir_{i}^s$ we inductively define $\dir_{i-1}^s$ to be the unique $\dir' \in \dirs^{b,s}$ such that $\rs_{\dir'}^s=\ls_{\dir_i^s}^s$ whenever such a $\dir'$ exists. If If $\ls_{\xi_{i}}^s = -\infty$, then no such $\dir'$ exists, and we terminate the procedure. This gives us an integral interval of indices $\mathcal I^s$ that can be either infinite or finite and such that $\dirs^{b,s} = (\dir_i^s)_{i \in \mathcal I^s}$. The monotonicity in Lemma \ref{lem:5} implies that $\xi_i > \xi_j$ whenever $i > j$ in $\mathcal I^s$. We now prove the following lemma. 
\begin{lemma} 
\label{lem:s3}
Let $b$ be an eternal solution, and assume there exists $s \in \R$ such that $\dirs^{b,s}$ contains at least two colors. Then, for all $x \in \R$,
  \begin{equation}
  \label{4.32}
  \begin{aligned}
         &\quad \;b(x,s)-b(\ls_{\dir_0},s) \\
         &=
         \begin{cases}
            \sum_{j=0}^{i-1}W^{\dir_j}\left(\ls_{\dir_j},s;\ls_{\dir_{j+1}},s \right)+W^{\dir_{i}}\left({\ls_{\dir_{i}},s;x,s}\right) & x\in[\ls_{\dir_i},\ls_{\dir_{i+1}}), i \in \Z_{\ge 0}\\
             -\sum_{j=i+1}^{-1}W^{\dir_j}\left(\ls_{\dir_{j}},s;\ls_{\dir_{j+1}},s\right)+W^{\dir_i}(\ls_{\dir_{i+1}},s;x,s) & x\in(\ls_{\dir_i},\ls_{\dir_{i+1}}], i \in \Z_{< 0},
         \end{cases}
         \end{aligned}
     \end{equation}
     where we have used the shorthand notation $\xi_i = \xi_i^s$ and $\ls_\xi = \ls_\xi^s$ for our particular choice of $s$. Here, in the $i = 0$ and $i = -1$ cases, we interpret the empty sum $\sum_{j = 0}^{-1}$ as $0$. 
     \end{lemma}
     \begin{proof}
     First, assume that $x \in (\ls_{\dir_i},\ls_{\dir_{i+1}})$ for some $i$ such that $\ls_{\dir_i} > -\infty$ (this is always true for $i = 0$ by definition). Then, for all $\ve > 0$ sufficiently small, \eqref{eq:s2} gives us
     \[
     \xi_{x,s}^b = \xi_{\ls_{\dir_i} + \ve,s}^b = \xi_i.
     \]
   By Lemma \ref{lem:all_coalesce}, all semi-infinite geodesics with color $\dir_{i}$ coalesce, so for all sufficiently small $\ve > 0$,  Lemma \ref{lem:Buse=W} implies
         \[
         b(x,s)-b(\ls_{\dir_i} + \ve,s)=W^{\dir_i}\left( \ls_{\dir_i}+\ve,s;x,s\right).
         \]
         Then, using continuity of $b$ (Corollary \ref{cor:cts_bounded}) and continuity of the Busemann function (Proposition \ref{prop:Buse_basic_properties}\ref{itm:general_cts}), by sending $\ve \searrow 0$, we obtain
        \be \label{eq:bdiff0}
         b(x,s)-b(\ls_{\dir_i},s)=W^{\dir_i}\left( \ls_{\dir_i},s;x,s\right),\quad\text{for all }x \in [\ls_{\xi_i},\ls_{\xi_i + 1}),
        \ee
        where the equality is trivial for $x = \ls_{\xi_i}$. As a remark, the equality also holds for $x = \ls_{\xi_{i+1}}$ when $\ls_{\xi_{i+1}} < \infty$. 

        Next, we use induction to prove \eqref{4.32}. We see that the case when $x \in [\ls_{\xi_0},\ls_{\xi_1})$ is the $i = 0$ case of \eqref{eq:bdiff0}. Now, suppose, for some $i \in \Z_{\ge 0}$ that for all for $x \in [\ls_{\dir_i},\ls_{\dir_{i+1}})$,
         \[
         b(x,s)-b(\ls_{\dir_0},s)=\sum_{j=0}^{i-1} W^{\dir_j}\left(\ls_{\dir_j},s;\ls_{\dir_{j+1}},s \right)+W^{\dir_{i}}\left({\ls_{\dir_{i}},s;x,s}\right).
         \]
         We will also assume that $\ls_{\xi_{i+1}} < \infty$; otherwise, we terminate the induction. Then, by continuity, for $x = \ls_{\xi_{i+1}}$, we have  
\be \label{eq:bdiff1}
b(\ls_{\dir_{i+1}},s)-b(\ls_{\dir_{0}},s)=\sum^{i}_{j=0}W^{\dir_j}\left(\ls_{\dir_j},s;\ls_{\dir_{j+1}},s \right).
\ee
         Next, for $x \in [\ls_{\dir_{i+1}},\ls_{\dir_{i+2}})$, by the $i +1$ case of \eqref{eq:bdiff0}, we have that 
         \be \label{eq:bdiff2}
        b(x,s)-b(\ls_{\dir_{i+1}},s)=W^{\dir_{i+1}}\left( \ls_{\dir_{i+1}},s;x,s\right).
         \ee
         Then, the result comes by adding \eqref{eq:bdiff1} and 
         \eqref{eq:bdiff2}. A symmetric inductive argument holds for $x \in (\ls_{\xi_1},\ls_{\xi_{i+1}}]$ when $i \in \Z_{<0}$, starting from $i = -1$.
     \end{proof}

We now have the following lemma.
\begin{lemma}\label{lem:eternal_time_s}
    Let $b$ be an eternal solution, and assume there exists $s \in \R$ so that $\dirs^{b,s}$ contains at least two colors. Then, for all $x \in \R$,
    \[
    b(x,s) =(\Wf{\mathfrak{c}^s})(x,s),
    \]
    where $\mathfrak{c}^s:\Xs \to [-\infty,\infty)$ is defined by $\mathfrak{c}^s(\xi) = -\infty$ for all $\xi \notin \dirs^{s,b}$, and for $\xi \in \dirs^{s,b} = \{\xi_i\}_{i \in \mathcal I^s}$,
    \begin{equation}\label{eq2}
    \mathfrak{c}^s \left(\dir_i \right)=
        b(\ls_{\dir_0},s) + \begin{cases}
            \sum^{i-1}_{j=0}W^{\dir_j}\left(\ls_{\dir_j},s;\ls_{\dir_{j+1}},s \right)+W^{\dir_{i}}\left(\ls_{\dir_{i}},s;0,0\right) & i \in \Z_{\ge 0}\\
            -\sum_{j=i+1}^{-1}W^{\dir_{j}}\left(\ls_{\dir_{j}},s;\ls_{\dir_{j+1}},s\right)+W^{\dir_{i}}\left(\ls_{\dir_{i+1}},s;0,0\right)& i \in \Z_{< 0},
        \end{cases}
    \end{equation}
    and we use the same shorthands for $\xi_i$ and $\ls_{\xi}$ as in Lemma \ref{lem:s3}.
\end{lemma}
\begin{proof}
By Lemma \ref{lem:s3} and additivity of the Busemann functions, for all indices $i$ and  $x \in (\ls_{\dir_i},\ls_{\dir_{i+1}})$,
\[
b(x,s)=W^{\dir_i}(0,0;x,s)+\mathfrak{c}^s(\dir_i),
\]
and equality holds at the endpoints whenever they are finite. 
We show that for any index $j \neq i$, 
\[
W^{\dir_i}(0,0;x,s)+\mathfrak{c}^s(\dir_i) \ge W^{\dir_j}(0,0;x,s)+\mathfrak{c}^s(\dir_j),\qquad\forall \; x \in (\ls_{\xi_i},\ls_{\xi_{i+1}}).
\]
The extension to $x = \ls_{\xi_i}$ and $x = \ls_{\xi_{i+1}}$ when these are finite then follows by continuity. 
         First, by definition of $\mathfrak c^s$ \eqref{eq2}, for any index $j<i$,
        \begin{align*}
            &\quad \;\mathfrak{c}^s(\dir_i)+W^{\dir_i}(0,0;x,s) - \Bigl(\mathfrak{c}^s(\dir_j)+W^{\dir_j}(0,0;x,s)\Bigr) \\
           &=W^{\dir_i}(\ls_{\xi_{i}},s;x,s) + W^{\dir_{j}}(x,s;\ls_{\xi_{j+1}},s)  +   \sum_{\ell = j+1}^{i-1}W^{\dir_\ell} \left(\ls_{\dir_{\ell}},s;\ls_{\dir_{\ell+1}},s \right)
            \\&\ge W^{\dir_i}(\ls_{\xi_{i}},s;x,s) + W^{\dir_{j}}(x,s;\ls_{\xi_{j+1}},s)  +   \sum_{\ell = j+1}^{i-1}W^{\dir_j} \left(\ls_{\dir_{\ell}},s;\ls_{\dir_{\ell+1}},s \right) \\
        &= W^{\dir_i}( \ls_{\xi_{i}},s;x,s) + W^{\dir_j}(x,s;\ls_{\xi_{i}},s)  = W^{\dir_i}( \ls_{\xi_{i}},s;x,s) - W^{\dir_j}(\ls_{\xi_{i}},s;x,s) \ge 0,   
        \end{align*}
        where in both inequalities, we have used the monotonicity of the Busemann functions and the assumption $x > \ls_{\xi_i}$ for the last inequality. The intermediate equalities used additivity of the Busemann functions.

         Next, observe that for indices $j > i$,
         \begin{align*}
             &\quad \;\mathfrak{c}^s(\dir_i)+W^{\dir_i}(0,0;x,s) - \Bigl(\mathfrak{c}^s(\dir_j)+W^{\dir_j}(0,0;x,s)\Bigr) \\
             &= W^{\dir_i}(\ls_{\xi_{i+1}},s;x,s) - W^{\dir_j}(\ls_{\xi_{j}},s;x,s) - \sum_{\ell = i+1}^{j-1} W^{\dir_\ell}(\ls_{\dir_\ell},s;\ls_{\dir_{\ell+1}},s) \\
             &\ge W^{\dir_i}(\ls_{\xi_{i+1}},s;x,s) - W^{\dir_j}(\ls_{\xi_{j}},s;x,s) - \sum_{\ell = i+1}^{j-1} W^{\dir_j}(\ls_{\dir_\ell},s;\ls_{\dir_{\ell+1}},s) \\
             &= W^{\dir_i}(\ls_{\xi_{i+1}},s;x,s) - W^{\dir_j}(\ls_{\xi_{i+1}},s;x,s) \ge 0,
         \end{align*}
         where we used the definition of $\mathfrak c^s$ in the first equality, the monotonicity of the Busemann functions in the first inequality, additivity in the last equality, and monotonicity along with the assumption that $x < \ls_{\xi_{i+1}}$ in the last inequality (the assumption that an index $j > i$ exists implies that $\ls_{\xi_{i+1}} < \infty$).  
\end{proof}

    \begin{lemma}\label{lm:glb}
        Let $b$ be an eternal solution such that $\dirs^{b,s}$ contains at least two colors for some $s \in \R$. For this choice of $s$, let $\mathfrak{c}^s$ be as defined  in Lemma \ref{lem:eternal_time_s}. Then, for all $t > s$ and $x \in \R$,
        \begin{equation}
            b(x,t)=(\Wf{\mathfrak{c}^s})(x,t).
        \end{equation}
    \end{lemma}
    \begin{proof}
By the assumption that $b$ is an eternal solution, followed by Lemma \ref{lem:eternal_time_s}, then Lemma \ref{lem:interchange_sups},
\[
    b(x,t) = \kpzs_t(x;b(\cdot,s),s) =\kpzs_t\bigl(x; (\Wf{\mathfrak{c}^s})(\cdot,s), s\bigr)(x,t) = (W\star \mathfrak c^s)(x,t). \qedhere
\]
       \end{proof}

\subsection{Colors across times}
Lemma \ref{lm:glb} shows that en eternal solution can be represented as $W\star \varphi$ for $(x,t) \in \R \times \R_{>s}$. To show that this can be extended to the full plane, we need to look at the set of colors across all times. The following results says that if a color appears at some time then in also appears at any previous time. 
\begin{lemma}\label{lm:4}
    For all $s<t$, we have $\dirs^{b,t}\subseteq\dirs^{b,s}$.
\end{lemma}
\begin{proof}
Fix $x\in\R$. Let $\dir=\dir_{x,t}^b$. Then $g_{(x,t)}^{b,L}$ will intersect the time horizon $s$ at some spatial point $y\in\R$. By the consistency of geodesics in Lemma \ref{lem:geodesics_from_b}\ref{itm: consis}, $\dir = \dir_{x,t}^b = \dir_{y,s}^b$. Therefore, $\dir \in \dirs^{b,s}$ as well.  
\end{proof}

\noindent 
Now, define
\begin{equation} \label{eq:colors}
    \dirs^b:=\bigcup_{s \in \R} \dirs^{b,s}.
\end{equation}
We call the set $\dirs^b$ the \textbf{set of colors of $b$}.
For the dynamics of an eternal solution $b$, we define the \textit{extinction time} $\tex:\Xs\rightarrow \R$ by
\begin{equation}
    \tex(\dir):=\sup\{s:\dir\in\dirs^{b,s}\}.
\end{equation}
If the set on the right hand side is empty we set $\tex(\dir)=-\infty$. We say that $\dir\in\Xs$ goes extinct if $-\infty<\tex(\dir)<\infty$. We then call $\tex(\dir)$ its time of extinction.

For an eternal solution $b$, we define its \textbf{coloring map}  $\clr^b:\dirs^b\rightarrow \mathcal{B}(\R^2)$ by
\begin{equation}\label{cm}
\clr^b(\dir)=\{p\in\R^2:\dir^{b}_p=\dir\},    
\end{equation}
where $\mathcal{B}(\R^2)$ is the set of Borel subsets of $\R^2$.
Lemma \ref{lem:5} and \eqref{eq:s2} together show that
\begin{equation} \label{eq:Cb_rep}
    \clr^b(\dir)=\bigcup_{s<\tex(\dir)} (\ls^s_\dir,\rs^s_{\dir}]\times \{s\},
\end{equation}
where $-\infty \le \ls_\dir^s < \rs_\dir^s \le \infty$ for all $s < \tex(\dir)$.

With this notation in hand, we are ready to prove the main result of this section. 
\begin{proposition} \label{prop:exists_varphi} 
    For almost-every $\omega$, whenever $b \in \es_\omega$, the set $\dirs^b$ is countable. Furthermore, there exists a function $\varphi:\Xs \to [-\infty,\infty)$ with $\supp(\varphi) = \dirs^{b}$  such that $b = W\star \varphi$. For this choice of $\varphi$, we have that
    \be \label{eq:b=phi+W_thm}
    b(x,t) = \varphi(\xi) + W^\dir(0,0;x,t),\quad \forall \dir \in \dirs^{b},\text{ and } (x,t) \in \clr^b(\xi).
    \ee
\end{proposition}
\begin{proof} 
Lemmas \ref{lem:2} and \ref{lm:4} imply the countability of $\dirs^b$. We turn to proving the remaining parts of the proposition. First, we handle the trivial case where $\dirs^b = \varnothing$, where we see that we may take $\varphi \equiv -\infty$. 

Next, we handle the simple case where $\dirs^b$ has only one element. In this case, there exists $\xi \in \Xs$ so that $\xi_{x,t}^b = \xi$ for all $(x,t) \in \R^2$. Then, for $(x,t) \in \R^2$, Lemma \ref{lem:Buse=W} implies that
\begin{align*}
    b(x,t) - b(0,0) = W^{\dir}(0,0;x,t),
\end{align*}
so the result follows by taking $\varphi(\dir) = b(0,0)$, and $\varphi = -\infty$ on $\Xs \setminus \{\xi\}$.

We now assume that $|\dirs^b| \ge 2$ so that $|\dirs^{b,s}| \ge 2$ for some $s \in \R$. Then, for all $s < t$, Lemma \ref{lm:glb} states that 
\be \label{eq:bsup_W1}
b(x,t) = \sup_{\dir \in \dirs^{b,s}}[\mathfrak c^s(\xi) + W^\dir(0,0;x,t)],
\ee
where $\mathfrak c^s$ is defined in Lemma \ref{lem:eternal_time_s}, Equation \eqref{eq2}. 

Now, let $\xi \in \dirs^b$. For each $s < \tex(\dir)$, choose $y_s \in (\ls_{\xi_{i}^s}^s,\ls_{\xi_{i+1}^s}^s )$, where $i$ is chosen so that $\xi_i^s = \xi$. Then, for each $(x,t) \in \clr^b(\dir)$ where $s < t < \tex(\dir)$, we have $\xi_{x,t}^b = \xi_{y_s,s}^b = \xi$ by Lemma \ref{lm:4}. Then, by Lemma \ref{lem:Buse=W},
\be \label{eq:ysBuse}
b(x,t) - b(y_s,s) = W^\dir(y_s,s;x,t).
\ee
Next, by Lemma \ref{lem:s3}, definition of $\mathfrak c^s$ \eqref{eq2}, and additivity of the Busemann functions, we have 
\be \label{eq:bys_c}
b(y_s,s) - \mathfrak c^s(\xi) = W^\dir(0,0;y_s,s).
\ee
Adding the equations \eqref{eq:ysBuse} and \eqref{eq:bys_c}, we see that
\be \label{eq:bxt=c+W}
b(x,t) = \mathfrak c^s(\xi) + W^\dir(0,0;x,t),\quad \forall \; \dir \in \dirs^{b,s},\; s < t < \tex(\xi),\text{ and } (x,t) \in \clr^b(\dir).
\ee
In particular, $\mathfrak c^s(\xi)$ does not depend on $s$, so long as $s < \tex(\xi)$. Then, for $\xi \in \dirs^b$, we define $\varphi(\xi)$ to be the $s$-independent value  $\mathfrak c^s(\xi)$ when $s < \tex(\xi)$. We also define $\varphi(\xi) = -\infty$ for $\xi \notin \dirs^b$. Then, by \eqref{eq:bsup_W1}, for all $(x,t) \in \R^2$ and $s < t$, 
\be \label{1678}
    b(x,t) = \sup_{\dir \in \dirs^{b,s}}[\varphi(\xi) + W^\dir(0,0;x,t)] \le \sup_{\dir \in \Xs}[\varphi(\xi) + W^\dir(0,0;x,t)].
\ee
Next, let $(x,t) \in \R^2$ and $\xi \in \dirs^b$, and choose $s < \min(t,\tex(\dir))$. Then, $\xi \in \dirs^{b,s}$ and $\varphi(\xi) = \mathfrak c^s(\xi)$, so 
\be \label{1679}
\varphi(\xi) + W^\dir(0,0;x,t) = \mathfrak c^s(\dir) + W^\dir(0,0;x,t) \le b(x,t),
\ee
where we have used \eqref{eq:bsup_W1} in the inequality. By combining the upper bound in \eqref{1678} with the lower bound in \eqref{1679} (after taking the supremum over 
$\dir \in \Xs$), for all $(x,t) \in \R^2$, we obtain
\[
b(x,t) = \sup_{\xi \in \Xs}[\varphi(\xi) + W^\dir(0,0;x,t)] = (\Wf{\varphi})(x,t).
\]
Equation \eqref{eq:b=phi+W_thm} then follows from \eqref{eq:bxt=c+W} and the definition of $\varphi$. 
\end{proof}

    \section{Condition for real-valued eternal solutions} \label{sec:finiteness}
In this section, we describe exactly which functions $\varphi$ satisfy $(W\star \varphi)(x,t) <\infty$ for all $(x,t) \in \R^2$; i.e., which give rise to eternal solutions. The main result of thi section is Proposition \ref{prop:Wstarphi_bd}. We first prove the following lemmas. 
\begin{lemma} \label{lem:Buse_fixed_time_bd}
    Let $t \in \R$. Then, on a $t$-dependent event of probability one, there exists a constant $C > 0$ so that, for all $x \in \R$ and $\xi \in \Xs$,
    \[
    |W^\xi(0,t;x,t) - 2\mathcal P(\xi) x| \le 
    C + |\mathcal P(\xi)| + 3|x|.
    \]
\end{lemma}
\begin{proof}
    We prove the bound for $x > 0$, and the result for $x < 0$ follows a symmetric proof. By Proposition \ref{prop:Buse_basic_properties}, each $x \mapsto W^\xi(0,t;x,t)$ is a Brownian motion with drift $2\mathcal P(\xi)$ and diffusion coefficient $\sqrt 2$. For each $\xi \in \Xs$, define
    \[
    C_\xi^R = \sup_{x > 0}[ W^\xi(0,t;x,t) - (2\mathcal P(\xi) + 1)x], \quad\text{and}\quad C_\xi^L := -\inf_{x > 0}[ W^\xi(0,t;x,t) - (2\mathcal P(\xi) - 1)x].
    \]
    Note that $C_\xi^R$ is the supremum of a Brownian motion with negative drift, and $C_\xi^L$ is the negative of the infimum of Brownian motion with positive drift. Thus, these two quantities have the exponential distribution (see, for example, \cite{BM_handbook}, Equation 1.1.4 (1) on pg 251]). Therefore, by the Borel-Cantelli lemma, there exists a constant $C$ such that for all $\xi \in \{\dir \in \Xs:\mathcal P(\dir) \in \Z\}$, we have  $\max(C_\xi^R,C_\xi^R) \le C - 1 + |\mathcal P(\xi)|$ (here, we are using the fact that each fixed $\dir \in \Xs$ is in  $\Xi^c$ with probability one (Proposition \ref{prop:Buse_basic_properties}). For $\dir \in \Xs$, let $\lceil \dir \rceil$ denote the element $\mathcal P(\dir)^0 \in \Xs$. Then, by monotonicity of the Busemann functions, by comparing $\mathcal P(\xi)$ to $\lceil \mathcal P(\xi) \rceil$, for all $\xi \in \Xs$ and $x > 0$, we have  
    \begin{align*}
    W^\xi(0,t;x,t) &\le W^{\lceil \xi \rceil}(0,t;x,t) \le  C_{\lceil \xi \rceil} + (2\lceil \mathcal P(\xi) \rceil + 1)x  \\
    &\le C - 1 + |\lceil \mathcal P(\xi) \rceil|
 +   (2\lceil \mathcal P( \xi) \rceil + 1)x \le C + |\mathcal P(\xi)| + (2\mathcal P(\xi) + 3)x. \qedhere
 \end{align*}
The proof of the lower bound is symmetric. 
\end{proof}

\begin{lemma} \label{lem:Buse_global}
Fix $s \in \R$. Then, on an $s$-dependent event of probability one, there exist continuous functions $C(x,s,t)$ defined for $t > s$ and $x \in \R$ and $a(s,t)$ defined for $t > s$ so that, for all $x,y \in \R$, $\xi \in \Xs$, and $t > s$,
   \[
   L(x,s,t,\xi) 
       \le W^\xi(0,s;x,t) \le R(x,s,t,\xi),
   \] 
   where
    \begin{align*}
        L(x,s,t,\xi) &= -|\mathcal P(\xi)| - C(x,s,t)   \\
        & +  \begin{cases}
            (2\mathcal P(\xi) -a(s,t))x + (t-s)\left(\mathcal P(\xi) - \f{a(s,t)}{2}\right)^2  &x + (t-s)\mathcal P(\xi) > \f{a(s,t)}{2}(t-s) \\
            -\f{x^2}{t-s} &|x + (t-s)\mathcal P(\xi)| \le \f{a(s,t)}{2}(t-s) \\
            (2\mathcal P(\xi) + a(s,t))x + (t-s)\Bigl(\mathcal P(\xi) + \f{a(s,t)}{2}\Bigr)^2 &x+ (t-s)\mathcal P(\xi) < -\f{a(s,t)}{2}(t-s),
        \end{cases}
    \end{align*}
    and 
    \begin{align*}
     R(x,s,t,\xi) &= |\mathcal P(\xi)| + C(x,s,t) \\
     &+\begin{cases}
            (2\mathcal P(\xi) + a(s,t))x + t\left(\mathcal P(\xi) + \f{a(s,t)}{2}\right)^2  &x + (t-s)\mathcal P(\xi) \ge 0 \\
            (2\mathcal P(\xi) - a(s,t))x + (t-s)\Bigl(\mathcal P(\xi) + \f{a(s,t)}{2}\Bigr)^2 &x+ (t-s)\mathcal P(\xi) \le 0.
        \end{cases}
    \end{align*}
\end{lemma}
\begin{proof}
     For all $t > s$ and $x \in \R$, we have 
    \be \label{eq:W_eternal_1234}
    W^\xi(0,s;x,t) = \sup_{y \in \R}[W^\xi(0,s;y,s) + \Ll(y,s;x,t)].
    \ee
    By Lemma \ref{lem:Buse_fixed_time_bd}, there exists a constant $C_1$ so that for all $\xi \in \Xs$ and $x \in \R$, we have 
    \[
    |W^\xi(0,s;y,s) - 2\mathcal P(\xi) y| \le C_1 + |\mathcal P(\xi)| + 3|x|. 
    \]
    Next, by replacing the bound in Lemma \ref{lem:Landscape_global_bound}, with a cruder bound, there exists constants $C_2,C_3,C_4$ (depending on $s$) so that for all $x,y \in \R$ and $t > s$,
    \[
    \left|\Ll(y,s;x,t) + \f{(x-y)^2}{t-s}\right| \le C_2 + C_3 (t-s)^{1/3}(C_4 + |x| + |y| + |t|+ |s|).
    \]
    Then, from \eqref{eq:W_eternal_1234}, for all $t > s$ and $x \in \R$, we have 
    \[
         L(x,s,t,\xi) 
       \le W^\xi(0,s;x,t) \le R(x,s,t,\xi).
    \]
    where 
    \begin{align*}
        L(x,s,t,\xi) &:= -|\mathcal P(\xi)| - C(x,s,t) + \sup_{y \in \R}\left[2\mathcal P(\xi) y - a(s,t)|y| - \f{(x-y)^2} {t}\right], \\
        R(x,s,t,\xi) &:= |\mathcal P(\xi)| + C(x,s,t) + \sup_{y \in \R}\left[2\mathcal P(\xi) y + a(s,t)|y| - \f{(x-y)^2}{t}\right],\quad \text{with} \\
        C(x,s,t) &:= C_1 + |\mathcal P(\xi)| + C_2 + C_3(t-s)^{1/3}(C_4 + |x| + |t| + |s|),\quad\text{and} \\
        a(s,t) &:=  3 + C_3 (t-s)^{1/3}.
    \end{align*}
    The lemma then follows from computing these suprema. 
\end{proof}

\begin{proposition} \label{prop:Wstarphi_bd}
The following equivalence holds for almost every $\omega$: a function $\varphi:\Xs \to [-\infty,\infty)$ satisfies $(W\star \varphi)(x,t) < \infty$ for all $(x,t) \in \R^2$ if and only if the following two conditions hold:
\begin{enumerate} [label=(\roman*), font=\normalfont]
    \item \label{it:bd_cpct} Whenever $I \subseteq \Xs$ is such that $\mathcal P(I)$ is bounded,
    \[
    \sup_{\dir \in I}\varphi(\dir) < \infty.
    \]
    \item \label{it:growth} We have 
    \[
    \limsup_{|\mathcal P(\dir)| \to \infty} \f{\varphi(\dir)}{|\mathcal P(\dir)|^2} = -\infty.
    \]
\end{enumerate}
\end{proposition}

\begin{proof}
Assume first that \ref{it:bd_cpct}-\ref{it:growth} hold, and define $b = W\star \varphi$. Since each Busemann function is an eternal solution, Lemma \ref{lem:interchange_sups} implies that for each $s < t$ and $x,y \in \R$,
    \[
    b(x,t) = \sup_{y \in \R}[b(y,s) +\Ll(y,s;x,t)].
    \]
    Then, if $b(y,s) = \infty$, for some $(y,s) \in \R^2$, we have that $b(x,t) = \infty$ for all $t > s$ and $x \in \R$. Hence, we see that $b(x,t) < \infty$ for all $(x,t) \in \R^2$ if and only if $b(x,t) < \infty$ for all $(x,t) \in \R \times \R_{> 0}$. We thus turn our attention to $t > 0$.

    From the expression for $R(x,t,\xi):= R(x,0,t,\xi)$ in Lemma \ref{lem:Buse_global}, for each $(x,t) \in \R \times \R_{> t}$, as $|\mathcal P(\xi)| \to \infty$, the leading order term of $R(x,t,\xi)$ is $t \mathcal P(\xi)^2$. Then, by the conditions \ref{it:bd_cpct}-\ref{it:growth}, we have that $\varphi(\xi) \le C - 2t\mathcal P(\xi)^2$ for a $t$-dependent constant $C$. Hence, for each $t > 0$ and $x \in \R$, we have 
\begin{equation}\label{eq:eub}
     b(x,t) \le \sup_{\dir \in \Xs}[C - 2t \mathcal P(\xi)^2 + R(x,t,\xi)] < \infty.
\end{equation}
    Conversely, assume that one of \ref{it:bd_cpct}-\ref{it:growth} fails. First, we assume that \ref{it:bd_cpct} fails. Then, there is a sequence $\{\dir_k\}$ such that $\mathcal P(\dir_k)$ is bounded and $\varphi(\dir_k) \to \infty$. By boundedness of $\mathcal P(\xi_k)$, Lemma \ref{lem:Buse_global} implies that $W^{\xi_k}(0,0;x,t)$ is bounded from below for each $(x,t) \in \R^2$, so
    \begin{equation}\label{eq:lb}
         b(x,t) \ge \sup_{k \in \N}[\varphi(\dir_k) + W^{\dir_k}(0,0;x,t)] = \infty.
    \end{equation}
    Now, assume that \ref{it:growth} fails. Then, there exists a sequence $\xi_k \in \Xs$ such that $|\mathcal P(\xi_{k})| \to \infty$ as $k \to \infty$ and $\varphi(\xi_{k}) \ge -A\mathcal P(\xi_k)^2$ for all $k$ and some constant $A$. For any given $(x,t) \in \R \times \R_{> 0}$ and sufficiently large $k$, we have $|x + t \mathcal P(\xi_k)| > \f{a(t)}{2}t$, where $a(t)$ is the function $a(0,t)$ in Lemma \ref{lem:Buse_global}. Then, for each fixed $(x,t) \in \R \times \R_{> 0}$, the expression $L(x,t,\xi_k):= L(x,0,t,\xi_k)$ in Lemma \ref{lem:Buse_global} has leading order $t \mathcal P(\xi_k)^2$ as $k \to \infty$.  Hence, for all $t > A$ and $x \in \R$,
    \[
    b(x,t) = \sup_{\xi \in \Xs}[W^\xi(0,0;x,t) +  \varphi(\xi) ] \ge \liminf_{k \to \infty} [ W^{\xi_k}(0,0;x,t) -A \mathcal P(\xi_k)^2] = \infty. \qedhere
    \]
    \end{proof}

\section{Uniqueness of $\varphi$} \label{sec:unique}
For convenience of the reader, we recall the definition of the set $\Phi_\omega$ from \eqref{eq:Phi_omega_def} here:
\be \label{eq:fun_repeat}
\fun_\omega:=\Bigl\{\varphi\in \operatorname{USC}(\Xs):\limsup_{|\mathcal P(\xi)| \to \infty} \f{\varphi(\xi)}{\mathcal P(\xi)^2} = -\infty\Bigr\},
\ee
where $\operatorname{USC}(\Xs)$ is the space of all upper semicontinuous functions $\Xs \to [-\infty,\infty)$. An immediate consequence of Lemma \ref{eq:interchange_sup}, Proposition \ref{prop:Wstarphi_bd}, and Lemma \ref{lem:USC_bd_cpct} is that $W\star \varphi \in \es_\omega$ (that is, $W\star \varphi$ is an eternal solution) for all $\varphi \in \fun_\omega$.  To construct the homeomorphism in Theorem \ref{thm:mr}, we must show that each eternal solution can be uniquely represented as $W\star \varphi$ for $\varphi \in \Phi$. Proposition \ref{prop:exists_varphi} gives the existence of a function $\varphi:\Xs \to [-\infty,\infty)$ with countable support, but that function does not necessarily lie in $\operatorname{USC}(\Xs)$. There is a standard procedure for extending functions to be upper semicontinuous, and this procedure does not change $W\star \varphi$ (see Lemma \ref{lem:one_extension}). It remains to show there is no other way to extend the function $\varphi$ to be in $\operatorname{USC}(\Xs)$. This is proved in Proposition \ref{prop:uniq}. Before then, we must prove several intermediate lemmas. These are divided into several subsections. In Section \ref{sec:maxes}, we show several facts about maximizers of the function $\varphi(\dir) + W^\dir(0,0;x,t)$. This culminates in the proof of Corollary \ref{cor:o1}, which implies that any two functions in $\fun$ with the same eternal solution must agree on the set of colors for that eternal solution. Then, Section \ref{sec:ci} introduces competition interfaces. These are used to prove Lemma \ref{lm:o8}, which allows us to restrict the max of $\varphi(\dir) + W^\dir(0,0;x,t)$ to a bounded interval. Section \ref{sec:dominating} then introduces the concept of dominating colors, which are shown to be equivalent to the set of colors of semi-infinite geodesics. These concepts are tied together to give the uniqueness statement in Proposition \ref{prop:uniq}. We then prove Proposition \ref{prop:Legendre} as a corollary, which gives the unique function $\varphi$ as a Legendre-type transform involving $b$ and $W$. 

\subsection{Maximizers in the variational formula} \label{sec:maxes}

\begin{lemma}  \label{lem:compact_interval}
    Fix $\dir_1<\dir_2$ in $\Xs$. Then $I=[\dir_1,\dir_2]$ is  compact in $\Xs$. Furthermore, any compact subset of $\Xs$ is contained in one of these intervals. 
\end{lemma}
\begin{proof}
    As $\Xs$ is a metric space (Lemma \ref{lem:X_metric}), it is enough to verify that every sequence in $I$ has a convergent subsequence.  Let $\{\dir_n\}\subseteq I$. Then, $\{\mathcal P(\dir_n)\}\subseteq \mathcal{P}(I)$ has a subsequence that converges to some $\eta\in \R$. Then, there is a further subsequence $\{\mathcal P(\dir_{n_k})\}$ that either converges to $\eta$ from the right or from the left. By the characterization of convergence given in Remark \ref{rmk:convergence}, the lifted subsequence $\{\dir_{n_k}\}$ converges to some $\dir \in \Xs$. Suppose, by way of contradiction, that $\dir \notin I$. Then,  by the characterization of convergence in Remark \ref{rmk:convergence}, we must have that $\dir_{n_k} \notin I$ for sufficiently large $k$, a contradiction. Now, for a general compact set $K \subseteq \Xs$, we can see that $\mathcal P(K)$ is compact in $\R$, and so $K \subseteq [a,b]$ for some $a < b$ in $\R$. Then, $K \subseteq [a^-,b^+]$ in $\Xs$. 
\end{proof}

With Lemma \ref{lem:compact_interval} in mind, we will refer to an interval of the form $[\dir_1,\dir_2] \subseteq \Xs$ as a compact interval of $\Xs$. For $\varphi:\Xs \to [-\infty,\infty)$, we define  
\[
f_\dir(x,t):=\varphi(\dir) + W^{\dir}(0,0;x,t).
\]
\begin{lemma}\label{lm:o3}
   For almost-every $\omega$,  for all $\varphi\in\fun_\omega$ with $\supp(\varphi) \neq \varnothing$ and $(x,t)\in\R^2$, there exists $\dir^*\in\supp(\varphi)$ and a compact interval $I \subseteq \Xs$ such that
    \[
        f_{\dir^*}(x,t)=(\Wf{\varphi})(x,t) > \sup_{\dir \notin I} f_\dir(x,t).
    \]
\end{lemma}
\begin{proof}
    We will take the event of probability one to be the event on which the bounds in Lemma \ref{lem:Buse_global} hold for all $s \in \Z$.  Choose $\varphi \in \fun$ satisfying $\supp(\varphi) \neq \varnothing$, and choose $\dir_0 \in \Xs$ satisfying $\varphi(\dir_0) > -\infty$. Let $(x,t) \in \R^2$, and pick any integer $s$ less than $t$.  From the superquadratic decay of $\varphi$ in \eqref{eq:fun_repeat} and the bounds in Lemma \ref{lem:Buse_global}, for all $\dir \in \Xs$, we have 
    \[
    f_\dir(x,t) \le \varphi(\dir) + R(x,s,t,\dir) \le -A\mathcal P(\dir)^2.
    \]
    for some constant $A > 0$ depending on $s$ and $t$. Hence, there exists a compact interval $I \subseteq \Xs$ such that 
    \[
    \sup_{\dir \notin I}f_\dir(x,t) < f_{\dir_0}(x,t),\quad\text{and therefore,}\quad (W\star \varphi)(x,t) = \sup_{\dir \in I} f_\dir(x,t).
    \]
Then, there exists a sequence $\{\dir_n\}\subseteq I$ such that 
\[
(\Wf{\varphi})(x,t)=\lim_{n \to \infty}f_{\dir_n}(x,t).
\]
By the compactness of $I$, there exists $\dir^* \in I$ and a  subsequence $\{\dir_{n_k}\}\subseteq I$ such that $\dir_{n_k}\rightarrow \dir^*$. Then, by upper semicontinuity of $\varphi$, 
    \begin{equation}
        (\Wf{\varphi})(x,t)=\lim_{k \to \infty}[W^{\dir_{n_k}}(0,0;x,t)+\varphi(\dir_{n_k})]= W^{\dir^*}(0,0;x,t)+\lim_{k\to\infty}\varphi(\dir_{n_k})\leq f_{\dir^*}(x,t).
    \end{equation}
    Since $f_{\dir}(x,t) \le (\Wf{\dir})(x,t)$ for general $\dir \in \Xs$, the result follows. 
\end{proof}
With Lemma \ref{lm:o3} in mind, for $(x,t)\in\R^2$ and $\varphi\in\fun$, we define
 \begin{align} \label{eq:Gammaxt}
         \Gamma_{x,t}&=\{\dir \in \Xs:\Wf{\varphi}(x,t)=f_\dir(x,t)\},\quad\text{and} \\
         \dir^{L}_{x,t}&=\inf\Gamma_{x,t}. \label{eq:xixt}
     \end{align} 
         \begin{lemma}\label{lm:o9}
             For almost-every $\omega$, for all $\varphi \in \Phi_\omega$ and $(x,t)\in\R^2$, the set $\Gamma_{x,t} \subseteq\Xs$ is compact. In particular, $\dir^{L}_{x,t}\in \Gamma_{x,t}$.
         \end{lemma}
         \begin{proof}
            Since $\X$ is a metric space (Lemma \ref{lem:X_metric}), it suffices to show sequential compactness. By Lemma \ref{lm:o3}, the set $\Gamma_{x,t}$ is nonempty and contained in a compact interval $I$. Hence, every sequence $\{\dir_n\} \subseteq \Gamma_{x,t}$ has a convergent subsequence $\{\dir_{n_k}\}$ whose limit $\dir$ lies in $I$. It remains to show that $\dir \in \Gamma_{x,t}$.    By upper semicontinuity of $\varphi$ and continuity of the Busemann process, we have
\begin{equation}    \begin{aligned}f_{\dir}(x,t)&=\varphi(\dir) + W^{\dir}(0,0;x,t) \\
                 &\geq \limsup_{k \to \infty} [W^{\xi_{n_k}}(0,0;x,t)+\varphi(\xi_{n_k})]=\limsup_{n \to \infty} f_{\xi_{n_k}}(x,t)=(\Wf{\varphi})(x,t),
                 \end{aligned}
             \end{equation}
             where in the last step, we used the assumption that $\dir_{n_k} \in \Gamma_{x,t}$ so that $f_{\dir_{n_k}}(x,t) = (\Wf{\varphi})(x,t)$ for all $k$. Therefore, $\dir \in \Gamma_{x,t}$, as desired.
         \end{proof}

\begin{lemma}\label{lm:o1}
      Let $\varphi \in \Phi_\omega$ and $(x,t)\in\R^2$.
     Then,
     \begin{equation}
         g^{\Wf{\varphi},L}_{(x,t)}=g^{\dir^{L}_{x,t},L}_{(x,t)}.
     \end{equation}
\end{lemma}
\begin{proof}
Recalling Definitions \ref{def:b-geodesics} and \ref{def:SIG_xi}, for any $T<t$ and $\xi \in \Xs_\omega$,
\be \label{eq:max_1567}
\begin{aligned}
g^{\Wf{\varphi},L}_{(x,t)}(T) &= \inf \argmax_{y \in \R}[W\star \varphi(y,T) + \Ll(y,T;x,t)],\quad\text{and} \\
g_{(x,t)}^{\xi,L}(T) &= \inf \argmax_{y \in \R}[W^{\dir}(y,T) + \Ll(y,T;x,t) ].
\end{aligned}
\ee
Now, observe that for $\dir \in \Gamma_{x,t}$,
\be \label{eq:S58291}
\begin{aligned}
    (W\star \varphi)(x,t) &\ge (W\star \varphi)\bigl(g_{(x,t)}^{\dir,L}(T),T\bigr) + \Ll\bigl(g_{(x,t)}^{\dir,L}(T),T;x,t\bigr) \\
    &\ge \varphi(\dir) +W^\dir\bigl(0,0;g_{(x,t)}^{\dir,L}(T),T\bigr) + \Ll\bigl(g_{(x,t)}^{\dir,L}(T),T;x,t\bigr) \\
    &= \varphi(\xi) + \sup_{y \in \R}[W^\dir(0,0;y,T) + \Ll(y,T;x,t)] \\
    &= \varphi(\xi) + W^\dir(0,0;x,t)= (W\star \varphi)(x,t),
\end{aligned}
\ee
where the first inequality follows because $W\star \varphi$ is an eternal solution, the second inequality follows by definition of $W\star \varphi$, the first equality follows by the second line of \eqref{eq:max_1567}, the second equality follows because the Busemann functions are eternal solutions,  and the last equality is the assumption $\dir \in \Gamma_{x,t}$. We then see that all inequalities above are equalities. In particular, for $\dir \in \Gamma_{x,t}$,
\be \label{eq:78291}
(W\star \varphi)\bigl(g_{(x,t)}^{\dir,L}(T),T\bigr) + \Ll\bigl(g_{(x,t)}^{\dir,L}(T),T;x,t\bigr) = (W\star \varphi)(x,t)
\ee
Furthermore, by setting $\xi = \xi_{x,t}^{L}$ and using Lemma \ref{lm:o9}, we obtain
\[
g_{(x,t)}^{\xi_{x,t}^L,L}(T) \in \argmax_{y \in \R}[(W\star \varphi)(y,T) + \Ll(y,T;x,t)].
\]
Now, by the first line in \eqref{eq:max_1567}, it suffices to show that if $y_0 < g_{(x,t)}^{\xi_{x,t}^L,L}(T)$, then $y_0$ is not in the $\argmax$ set above. Assume, by way of contradiction, that 
\be \label{eq:7341}
(W\star \varphi)(y_0,T) + \Ll(y_0,T;x,t) = \max_{y \in \R}[(W\star \varphi)(y,T) + \Ll(y,T;x,t)] = (W\star \varphi)(x,t).
\ee
For this choice of $y_0$, Lemma \ref{lm:o3} implies that there exists $\dir^* \in \Xs_\omega$ so that 
\[
(W\star \varphi)(y_0,T) = \varphi(\dir^*) + W^{\dir^*}(0,0;y_0,T).
\]
Combining this with \eqref{eq:7341}, we have 
\be \label{eq:6945}
\begin{aligned}
    (W\star \varphi)(x,t) &= \varphi(\dir^*) + W^{\dir^*}(0,0;y_0,T) + \Ll(y_0,T;x,t) \\
    &\le \varphi(\dir^*) + \sup_{y \in \R}[W^{\dir^*}(0,0;y,T) + \Ll(y,T;x,t)]  \\
    &= \varphi(\dir^*) + W^{\dir^*}(0,0;x,t) \le (W\star \varphi)(x,t).
\end{aligned}
\ee
Hence, all inequalities above are equalities. In particular, $\varphi(\dir^*) + W^{\dir^*}(0,0;x,t) = (W\star \varphi)(x,t)$, so by definition of $\Gamma_{x,t}$ \eqref{eq:Gammaxt}, we must have that $\dir^* \in \Gamma_{x,t}$.

Next, by definition of $\xi_{x,t}^L$ \eqref{eq:xixt}, we must have that $\xi^* \ge \xi_{x,t}^L$, so $y_0 < g_{(x,t)}^{\xi_{x,t}^L,L}(T)  \le  g_{(x,t)}^{\xi^*,L}(T)$ by ordering of geodesics.
Then, by the second line of \eqref{eq:max_1567}, $y_0$ is not a maximizer of $y \mapsto W^{\xi^*}(0,0;y,T) + \Ll(y,T;x,t)$, and we have
\[
\begin{aligned}
    &\quad \; \varphi(\dir^*) +W^{\dir^*}(0,0;y_0,T) + \Ll(y_0,T;x,t) \\ &< \varphi(\xi^*) + \sup_{y \in \R}[W^{\dir^*}(0,0;y,T) + \Ll(y,T;x,t)] \\
    &=\varphi(\dir^*) + W^{\dir^*}(0,0;x,t) = (W\star \varphi)(x,t).
\end{aligned}
\]
This contradicts the first equality of \eqref{eq:6945}.
\end{proof}

For the following corollary, recall the definition of the set of colors of an eternal solution $\dirs^b$ \eqref{eq:colors}  and the coloring map $\clr^b(\xi)$ \eqref{cm}. 
 \begin{corollary} \label{cor:o1}
             
             For any $\varphi\in\fun_\omega$,
             \begin{equation}
                 (\Wf{\varphi})(x,t)=f_\dir(x,t), \qquad \forall \;  \dir\in\dirs^{W\star \varphi}\text{ and } (x,t) \in \clr^{\Wf{\varphi}}(\dir).
             \end{equation}
       \end{corollary}
\begin{proof}
    Let $\dir \in \dirs^{W\star \varphi}$ and let $(x,t) \in  \clr^{\Wf{\varphi}}(\dir)$. By definition of $\clr^{\Wf{\varphi}}(\dir)$ and Lemma \ref{lm:o1}, $\dir_{x,t}^L = \dir$. Then, by definition of $\Gamma_{x,t}$ \eqref{eq:Gammaxt} and Lemma \ref{lm:o9}, $(\Wf{\varphi})(x,t) = f_\dir(x,t)$. 
\end{proof}

\subsection{Competition interfaces} \label{sec:ci}
  In this section, we will define two notions of competition interfaces, which will then be shown to be equivalent. For $c = (c_1,c_2)\in\R^2$ and $\dir_1< \dir_2$ in $\Xs$, define
    \begin{equation}
        \begin{aligned}
            \tau_t^{R,c,\dir_1,\dir_2} &= \sup\{x \in \R: W^{\dir_1} (0,0;x,t) +c_1 = W^{\dir_2}(0,0;x,t) + c_2\},\quad\text{and}\\
            \tau_t^{L,c,\dir_1,\dir_2} &= \inf\{x \in \R: W^{\dir_1} (0,0;x,t) +c_1 = W^{\dir_2}(0,0;x,t) + c_2\}.
        \end{aligned}
    \end{equation}
\begin{lemma} \label{lem:7842}
    With probability one, for all $\dir_1 < \dir_2$ in $\Xs$ and $t \in \R$, the function $x \mapsto W^{\dir_2}(0,0;x,t) - W^{\dir_1}(0,0;x,t)$ is nondecreasing and satisfies
    \[
    \lim_{x \to \pm \infty} W^{\dir_2} (0,0;x,t) -W^{\dir_1}(0,0;x,t) = \pm \infty,
    \]
    In particular, $\tau_t^{L/R,c,\dir_1,\dir_2} \in \R$ for all $t \in \R$ and $c = (c_1,c_2) \in \R^2$. Furthermore, for all $x >\tau_t^{R,c,\dir_1,\dir_2}$, we have 
    \[
    W^{\dir_1}(0,0;x,t) + c_1 < W^{\dir_2}(0,0;x,t) + c_2,
    \]
    while for $x < \tau_t^{L,c,\dir_1,\dir_2}$,
    \[
    W^{\dir_1}(0,0;x,t) + c_1 > W^{\dir_2}(0,0;x,t) + c_2.
    \]
\end{lemma}
\begin{proof}
    This follows immediately from Proposition \ref{prop:Buse_basic_properties}\ref{itm:DL_Buse_gen_mont}-\ref{it:to_pm_inf}.  
\end{proof}

 Consistent with the notations in \cite{Rahman-Virag-21}, for $x_0 \in \R$, $t>s$ and a continuous function $f:\R\rightarrow [-\infty,\infty)$, we define
  \begin{equation}\label{eq:d}
     d_{(x_0,s)}(f;x,t)=\sup_{z \geq x_0} \left\{f(z)+\Ll(z,s; x,t) \right\}-\sup_{z \leq x_0} \left\{f(z)+\Ll(z,s; x,t)\right\}. 
  \end{equation}
It follows by a standard paths-crossing argument that $x \mapsto d_{(x_0,s)}(f;x,t)$ is nondecreasing (see \cite[Proposition 4.1]{Rahman-Virag-21}).  For an initial condition $f$ and $t > s$, let us define 
 \be \label{eq:taupm_def}
\begin{aligned}
& \intc^-_{f;(x_0,s)}(t)=\inf \{ x \in \R : d_{(x_0,s)}(f;x,t) \geq 0\},\\
& \intc^+_{f;(x_0,s)}(t)=\sup \{x \in \R: d_{(x_0,s)}(f;x,t) \leq 0 \},
\end{aligned}
\ee
while for $t = s$, we define
\be \label{intcat0}
 \intc^-_{f;(x_0,s)}(s) = \intc^+_{f;(x_0,s)}(s)=x_0
\ee
so that the interface starts at the point $(x_0,s)$. We  call $\intc^-_{f;(x_0,s)}$ (resp.\ $\intc^+_{f;(x_0,s)}$) the \textbf{\textit{leftmost} (resp.\ \textit{rightmost}) competition interface} from initial condition $f$. The monotonicity of $d_{(x_0,s)}(f;x,t)$ implies that  $\intc^-_{f;(x_0,s)}(t) \le \intc^+_{f;(x_0,s)}(t)$ for all $t \ge s$. \\
In the context of this work we will be interested in interfaces starting from $(x_0,s)\in\R^2$ associated with functions of the form
\begin{equation} \label{eq:fxi1xi2_def}
    f^{\dir_1,\dir_2}_{(x_0,s)}(x)=
    \begin{cases}
        W^{\dir_1}(x_0,s;x,s) & x\leq x_0\\
        W^{\dir_2}(x_0,s;x,s) & x\geq x_0,
    \end{cases}
\end{equation}
where $\dir_1<\dir_2$ in $\Xs$. 
The connection between the two types of interfaces is summarized in the following result.
\begin{lemma} [{\cite[Proposition 2.3]{Dunlap-Sorensen-2025} and \cite[Proposition 4.5]{Rahman-Virag-21}}]\label{lemma 4.1}
    Fix any $\dir_1<\dir_2$ and $c_1,c_2\in\R$. Then, for any $s\in\R$,
    \begin{equation} \label{tau_is_interface}
        \tau^{L,c,\dir_1,\dir_2}_t=\tau^-_{f;p}(t), \qquad \text{and} \qquad\tau^{R,c,\dir_1,\dir_2}_t=\tau^+_{f;p}(t)\qquad \text{for all } t\geq s,
    \end{equation}
    where $p=(\tau^{c,\dir_1,\dir_2}_s,s)$ and $f=f^{\dir_1,\dir_2}_{p}$. In particular, the function $t \mapsto \tau_t^{c,\dir_1,\dir_2}$ is continuous. 
\end{lemma}
\begin{proof}
Some justification is required to reduce it to the cited results, so we take care of that here. Once \eqref{tau_is_interface} is proved, the continuity is directly \cite[Proposition 4.5]{Rahman-Virag-21}.  Let $b$ denote the eternal solution
 \[
 b(x,t) = \max\{ W^{\dir_1} (0,0;x,t) +c_1, W^{\dir_2}(0,0;x,t) + c_2\}
 \]
Since each Busemann function is an eternal solution, we have 
\begin{align*}
    b(x,t) = \mathfrak h_t(x; f_s^{1,c_1},s) \vee \mathfrak h_t(x; f_s^{2,c_2},s) = \mathfrak h_t(x;f_s^{1,c_1} \vee f_s^{2,c_2},s),
\end{align*}
where, for $(x,t) \in \R^2$ and $i \in \{1,2\}$,
 \begin{equation} \label{eq:ft12_def}
    \begin{aligned}
    f^{i}_t(x)&=W^{\dir_i}(0,0;x,t)+c_i.
\end{aligned}
\end{equation}
Then, by \cite[Proposition 2.3]{Dunlap-Sorensen-2025} \footnote{The result in \cite{Dunlap-Sorensen-2025} is stated for functions $f_s^1,f_s^2$ satisfying $\lim_{|x| \to \infty} \f{f_s^1(x)}{x} = \theta > 0$ and $\lim_{|x| \to \infty} \f{f_s^2(x)}{x} = -\theta$, but the proof only requires the fact that $x \mapsto f_s^2(x)-f_s^1(x)$ is nondecreasing and $\lim_{x \to \pm \infty} f_s^2(x)-f_s^1(x) = \pm \infty$, which holds by Lemma \ref{lem:7842}.} for any $s \in \R$, we have that 
\[
\tau_t^{c,\xi_1,\xi_2} = \tau_{f_s^{1} \vee f_s^{2};p}^+(t), 
\]
where $p = (\tau_s^{c,\xi_1,\xi_2},s)$. It follows immediately from the definition \eqref{eq:taupm_def} that if we shift the initial function $f$ by a constant, the competition interface stays the same. Hence, to get to \eqref{tau_is_interface}, it suffices to show that $f_s^{1} \vee f_s^{2}$ is the same as the function $f_p^{\xi_1,\xi_2}$, modulo addition by a constant.  Set $x_0 = \tau_s^{c,\xi_1,\xi_2}$. By Lemma \ref{lem:7842}, we have 
\begin{align*}
f_s^{1}(x) \vee f_s^{2}(x) &= \begin{cases}
    W^{\xi_1}(0,0;x,s) + c_1, &x \le x_0 \\
    W^{\xi_2}(0,0;x,s) + c_2,& x \ge x_0.
\end{cases}
\end{align*}
By continuity, we have that $C := W^{\xi_1}(0,0;x_0,s) + c_1 = W^{\xi_2}(0,0;x_0,s) + c_2$. Hence, recalling the definition of $f_p^{\xi_1,\xi_2}$ \eqref{eq:fxi1xi2_def}, we have
\begin{align*}
f_s^{1}(x)\vee f_s^{2}(x) &= \begin{cases}
    W^{\xi_1}(0,0;x_0,s) +W^{\xi_1}(x_0,s;x,s) + c_1, &x \le x_0 \\
    W^{\xi_2}(0,0;x_0,s) + W^{\xi_2}(x_0,s;x,s) + c_2,& x \ge x_0 
\end{cases} =C + f_p^{\xi_1,\xi_2}(x),
\end{align*}
as desired. 
\end{proof}

\begin{lemma}\label{lm:o8}
Let $(x,t)\in\R^2$. Then for any $\dir_1<\dir_2$  and $c\in\R^2$, there exists $T_0 \in \R$ so that, for all $T < T_0$,
\begin{equation}
\label{eq:str_ineq_1}
    g^{\dir_1,R}_{(x,t)}(T) <\tau_T^{L,c,\dir_1,\dir_2} \le \tau_T^{R,c,\dir_1,\dir_2}<g^{\dir_2,L}_{(x,t)}(T).
\end{equation}
\end{lemma}
\begin{proof}
We prove the last inequality of \eqref{eq:str_ineq_1}, with the first following symmetrically, and the middle inequality following immediately from the definition. Let $b$ denote the eternal solution
\[
b(x,t)=\max\{ W^{\dir_1} (0,0;x,t) +c_1, W^{\dir_2}(0,0;x,t) + c_2\}.
\]
    Choose $y> \tau_t^{R,c,\dir_1,\dir_2}$. Then, by Lemma \ref{lem:7842}, $b(y,t) = W^{\dir_2}(0,0;y,t) + c_2 > W^{\dir_1}(0,0;x,t)+c_1$, and Lemma \ref{lm:o1} implies that
    $
    g_{(y,t)}^{\dir_2,L} = g_{(y,t)}^{b,L}$.
    We now claim that for all $T<t$,
    \be \label{eq:taut_true}
    \tau_T^{R,c,\dir_1,\dir_2}<g^{\dir_2,L}_{(y,t)}(T).
    \ee
    Assume, by way of contradiction, that $\tau_T^{R,c,\dir_1,\dir_2} \ge g^{\dir_2,L}_{(y,t)}(T)$.  Then, we have 
    \be \label{eq:7032}
\begin{aligned}
&\quad \;W^{\dir_1}\left(0,0;y,t\right)-\Ll \left( g^{\dir_2,L}_{(y,t)}(T),T; y,t\right) + c_1 =  W^{\dir_1}\left(0,0;y,t\right)-W^{\dir_2} \left( g^{\dir_2,L}_{(y,t)}(T),T; y,t\right) + c_1\\
&<W^{\dir_2}\left(0,0;y,t\right)+W^{\dir_2}\left( y,t; g^{\dir_2,L}_{(y,t)}(T),T \right)+c_2 =W^{\dir_2}\left(0,0;g^{\dir_2,L}_{(y,t)}(T),T \right)+c_2 \\
&\leq W^{\dir_1}\left(0,0;g^{\dir_2,L}_{(y,t)}(T),T \right)+c_1 = W^{\dir_1}(0,0;y,t) - W^{\dir_1}\Bigl(g^{\dir_2,L}_{(y,t)}(T),T;y,t \Bigr)+c_1 ,
\end{aligned}
\ee
where the first equality is the equality of the Busemann function and $\Ll$ along a semi-infinite geodesic (Lemma \ref{lem:W=landscape}, Equation \eqref{L =landscape}), the strict inequality follows by the assumption that $y > \tau_t^{R,c,\dir_1,\dir_2}$  and Lemma \ref{lem:7842}, the next equality is additivity of the Busemann functions, the last inequality is the assumption that $\tau_T^{R,c,\dir_1,\dir_2} \ge g^{\dir_2,L}_{(y,t)}(T)$ and Lemma \ref{lem:7842}, and the last equality is the additivity of the Busemann functions. Rearranging the first and last terms in \eqref{eq:7032}, we have 
\[
W^{\dir_1}\Bigl(g^{\dir_2,L}_{(y,t)}(T),T;y,t \Bigr) < \Ll \left( g^{\dir_2,L}_{(y,t)}(T),T; y,t\right),
\]
a contradiction to Lemma \ref{L =landscape}, Equation \eqref{eq:landscape_leq}. Hence, \eqref{eq:taut_true} holds, as desired.

Now by coalescence of $\dir_2$ geodesics (Lemma \ref{lem:all_coalesce}), we have that for $x \in \R$, there exists $T_0 < t$ so that for all $T < T_0$, $g^{\dir_2,L}_{(x,t)}(T) = g^{\dir_2,L}_{(y,t)}(T) $. The result then follows by combining this with \eqref{eq:taut_true}. 
\end{proof}

For $\varphi\in\fun$ and $A\subseteq \Xs$, define the function $\varphi_A:\Xs \to \R$ by
\begin{equation}
    \varphi_A(\varphi) =
    \begin{cases}
       \varphi(\dir) & \dir\in A\\
       -\infty  & \dir\notin A.
    \end{cases}
\end{equation}
For $\ve > 0$, recall the definition of the interval $I_\ve(\dir)$ \eqref{eq:o36}, and define
\begin{equation}\label{eq:o20}
\begin{aligned}
      I^{L,-}_\ve(\dir)&=\sup\{\dir' \in \Xs_\omega:\dir'<I_\ve(\dir)\}\\
      I^{L,+}_\ve(\dir)&=\inf\{\dir' \in \Xs_\omega:\dir'\in I_\ve(\dir)\}\\
       I^{R,+}_\ve(\dir)&=\inf\{\dir' \in \Xs_\omega:\dir'> I_\ve(\dir)\}\\
       I^{R,-}_\ve(\dir)&=\sup\{\dir' \in \Xs_\omega:\dir'\in I_\ve(\dir)\}.
\end{aligned}
\end{equation}
Here, the notation $\dir' < I_\ve(\dir)$ means that $\dir'< \wh \dir$ for all $\wh \dir \in I_\ve(\dir)$. As an example, if $\dir_0 = \eta_0^+\in\Xi^+$, then $I^{L,-}_\ve(\dir_0)=\eta_0^-$ and $I^{L,+}_\ve(\dir_0)=\dir_0.$ The elements in \eqref{eq:o20} will allow us to prove the results below without having to treat the different types of elements in $\Xs$ (i.e.\ whether in $\Xi^c,\Xi^-$ or $\Xi^+$).
\begin{figure}[htbp]
    \centering

    \begin{subfigure}[t]{0.48\textwidth}
        \centering
       \begin{minipage}[c][5cm][c]{\linewidth}
    \centering

    \begin{tikzpicture}

        \node[
            anchor=south west,
            inner sep=0
        ] (image) at (0,0) {%
            \includegraphics[
                page=1,
                width=\linewidth,
                height=4.8cm,
                keepaspectratio,
                trim=10 240 50 200,
                clip
            ]{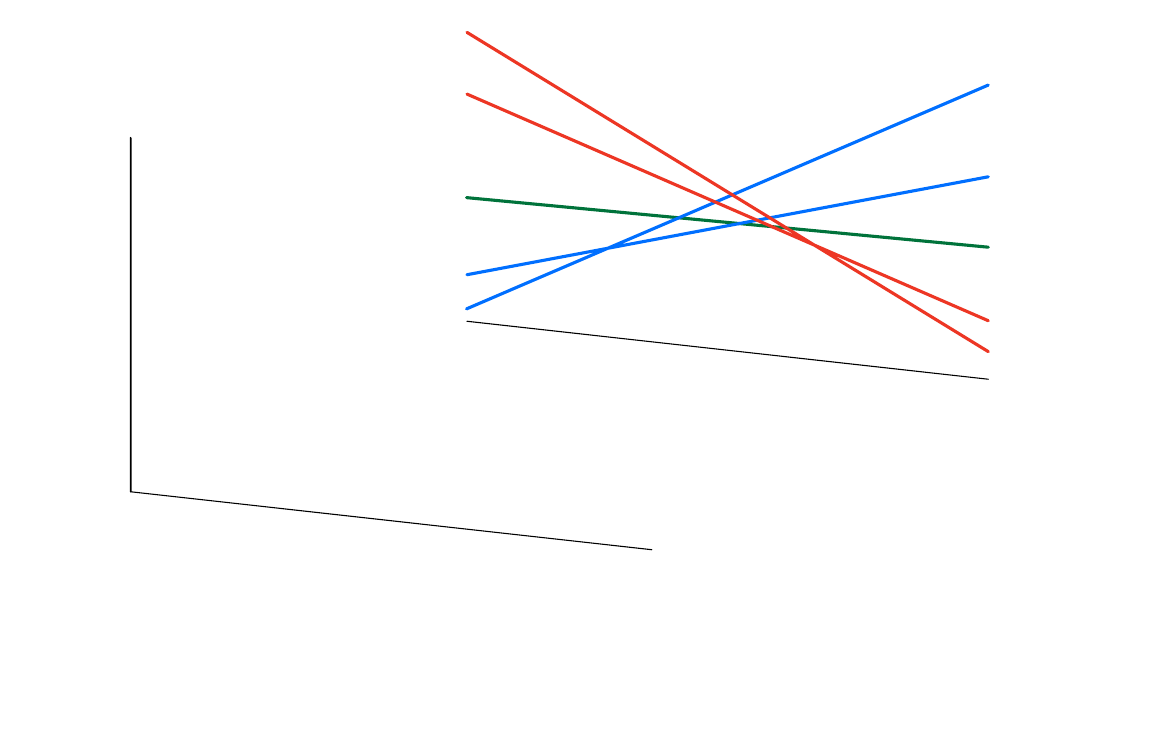}%
        };

        \begin{scope}[
            x={(image.south east)},
            y={(image.north west)}
        ]
            \coordinate (timearrow) at (0.15,0.4);
             \coordinate (spacearrow) at (0.15,0.2);
        \end{scope}

        \begin{scope}[
            shift={(timearrow)},
            rotate=25,
            transform shape
        ]
            \draw[
                ->
            ]
                (0,0) -- (1.5,0)
                node[midway, above] {\(\mathrm{time}\)};
        \end{scope}
        \begin{scope}[
            shift={(spacearrow)},
            rotate=-6,
            transform shape
        ]
            \draw[
                <->
            ]
                (0,0) -- (2.1,0)
                node[midway, above] {\(\mathrm{space}\)};
        \end{scope}
         \node at (7.1,2.4) {$t=0$};
         \node at (5.4,1.2) {$t=T<0$};
    \end{tikzpicture}
\end{minipage}
        \caption{At time $t=0$, the function $f_\dir(\cdot,0)$
        (green line) is dominated by colours of slope smaller than
        $\dir-\ve$ (red) and slope larger than $\dir+\ve$ (blue).}
    \end{subfigure}
    \hfill
    \begin{subfigure}[t]{0.48\textwidth}
        \centering
        \begin{minipage}[c][5cm][c]{\linewidth}
            \centering
            \includegraphics[
                page=2,
                width=\linewidth,
                height=4.8cm,
                keepaspectratio,
                trim=10 240 50 200,
                clip
            ]{interface_proof.pdf}
        \end{minipage}
        \caption{We pick the largest red slope and the smallest blue
        slope and raise them to the maximum of $\varphi$.}
    \end{subfigure}

    \medskip

    \begin{subfigure}[t]{0.48\textwidth}
        \centering
        \begin{minipage}[c][5cm][c]{\linewidth}
            \centering
             \begin{tikzpicture}

        \node[
            anchor=south west,
            inner sep=0
        ] (image) at (0,0) {%
            \includegraphics[
                page=3,
                width=\linewidth,
                height=4.8cm,
                keepaspectratio,
                trim=10 240 50 200,
                clip
            ]{interface_proof.pdf}%
        };

        \begin{scope}[
            x={(image.south east)},
            y={(image.north west)}
        ]
            
             \coordinate (spacearrow) at (0.2,0.15);
        \end{scope}

       
        \begin{scope}[
            shift={(spacearrow)},
            rotate=-6,
            transform shape
        ]
            \draw[
                <->
            ]
                (0,0) -- (2.1,0)
                node[midway, above] {\(J_{T,\dir}\)};
        \end{scope}
         \begin{scope}[
    x={(image.south east)},
    y={(image.north west)}
]
    \coordinate (interfaceTail) at (0.75,0.25);
    \coordinate (interfaceTip)  at (0.55,0.3);

    \draw[->, thick, red]
        (interfaceTail) -- (interfaceTip);

   \node[
    anchor=west,
    align=center,
    xshift=2pt,
    font=\small
] at (interfaceTail) {interface\\with\\the green};
\end{scope}
    \end{tikzpicture}            
        \end{minipage}
        \caption{Our choice of lifted curves (purple) ensures that
        the interfaces between the purple curves and the green one
        will: (a) cross and create the interval $J_{T,\dir}$; and
        (b) place every interface formed by a red (respectively, blue)
        curve and the green curve to the left (respectively, right)
        of $J_{T,\dir}$.}
    \end{subfigure}
    \hfill
    \begin{subfigure}[t]{0.48\textwidth}
        \centering
        \begin{minipage}[c][5cm][c]{\linewidth}
            \centering
            \includegraphics[
                page=4,
                width=\linewidth,
                height=4.8cm,
                keepaspectratio,
                trim=10 240 50 200,
                clip
            ]{interface_proof.pdf}
        \end{minipage}
        \caption{The spatial interval $J_{T,\dir}$ at time $t$
        ensures that colour $\dir$ (green) dominates every colour
        that is not in $I_\ve(\dir)$.}
    \end{subfigure}

    \caption{A sketch of the proof of Lemma~\ref{lm:o2}.}
    \label{fig:backwards}
\end{figure}

The following lemma is one of the key inputs in the proof of the uniqueness in Proposition \ref{prop:uniq} below. It allows us to restrict our attention to a bounded set of colors. 

\begin{lemma}\label{lm:o2} 
   The following holds for almost-every $\omega$. Let $\varphi\in\fun_\omega$ and let $\dir\in\supp(\varphi)$.  Then for every $\ve>0$ there exists   $T_0 = T_0(\ve,\dir)<0$ such that for all $T<T_0$, there exists an interval $J_{T,\dir}\subseteq \R$ \rm{(}depending on $\ve$\rm{)} such that 
    \begin{equation} \label{eq:o391}
        (\Wf{\varphi})(x,T)=(\Wf{{\varphi_{I_\ve}}})(x,T)>(\Wf{{\varphi_{I^c_\ve}}})(x,T), \qquad \forall x\in J_{T,\dir}.
    \end{equation}
    Moreover, for every $(x_0,t_0)\in\R^2$, there exists $T_0' = T_0'(\ve,\xi,x_0,t_0) < t_0$ so that for all $T < T_0'$, 
\begin{equation}\label{eq:o23}
        [g^{\dir,L}_{(x_0,t_0)}(T),g^{\dir,R}_{(x_0,t_0)}(T)]\subseteq J_{T,\dir}.
    \end{equation}
\end{lemma}
\begin{proof}
    Figure \ref{fig:backwards} gives a visual representation of the proof. To ease the notation, we abbreviate $I_\ve:=I_\ve(\dir)$.  For $t \in \R$, we define 
    \begin{equation}\label{eq:o14}
        \begin{aligned}
            \text{Int}^+(I_\ve,t)&=\sup\big\{\tau_t^{R,(\varphi(\dir'),\varphi(\dir)),\dir',\dir}:\dir'< I_\ve\big\}, \qquad \text{and}\\
        \text{Int}^-(I_\ve,t)&=\inf\big\{\tau_t^{L,(\varphi(\dir),\varphi(\dir')),\dir,\dir'}:\dir'>  I_\ve\big\}.
        \end{aligned}     
    \end{equation}
 Note that if $\text{Int}^+(I_\ve,t) < x < \text{Int}^-(I_\ve,t)$, then by Lemma \ref{lem:7842}, we would have that 
\be \label{eq:5621}
f_\dir(x,t) > f_{\dir'}(x,t),\quad \forall \; \dir' \in I_\ve^c.
\ee
By Lemma \ref{lm:o3}, there exists $\dir^* \in \Xs$ so that $(W\star \varphi)(x,t) = f_{\dir^*}(x,t)$. Then, if \eqref{eq:5621} holds, we would have that all such $\dir^*$ lie in $I_\ve$, and so
\be \label{eq:o954}
(W\star \varphi)(x,t) = (W\star \varphi_{I_\ve})(x,t) >(W\star \varphi_{I_\ve^c})(x,t),\quad \forall \; \text{Int}^+(I_\ve,t) < x < \text{Int}^-(I_\ve,t). 
\ee
Let $\varphi^{M}=\sup_{\eta \in \Xs}\varphi(\eta)$ and $\varphi_0=\varphi(\dir)$. By assumption that $\varphi \in \Phi$ and Lemma \ref{lem:USC_bd_cpct}, we have $\varphi^M < \infty$.  Note that for any $(x_0,t_0) \in \R^2$, Lemma \ref{lm:o8} implies that  there exists $T_0' = T_0'(\ve,\xi,x_0,t_0) < t_0$ so that for all $T < T_0'$,
    \be \label{eq:3781}
    \tau_T^{R,(\varphi^M,\varphi_0),I_\ve^{L,-},\dir} < g_{(x_0,t_0)}^{\dir,L}(T) \le g_{(x_0,t_0)}^{\dir,R}(T) < \tau_T^{L,(\varphi_0,\varphi^M),\dir,I_\ve^{R,+}}.
    \ee
    Then, set $T_0(\ve,\xi) := T_0'(\ve,\xi,0,0)$, and for $T < T_0$,
    Then, \eqref{eq:o391} and \eqref{eq:o23} follow immediately from \eqref{eq:o954}-\eqref{eq:3781} with $J_{T,\xi} = (\text{Int}^+(I_\ve,T),\text{Int}^-(I_\ve,T))$ once we show that, for all $t \in \R$, \begin{equation}\label{eq:o15}
    \begin{aligned}
        &\text{Int}^+(I_\ve,t)\leq \sup\big\{\tau_t^{R,(\varphi^M,\varphi_0),\dir',\dir}:\dir'< I_\ve\big\}\leq \tau_t^{R,(\varphi^M,\varphi_0),I_\ve^{L,-},\dir},\quad\text{and} \\
        &\text{Int}^-(I_\ve,t) \ge \inf\big\{\tau_t^{L,(\varphi_0,\varphi^M),\dir',\dir}:\dir'< I_\ve\big\} \geq \tau_t^{L,(\varphi_0,\varphi^M),\dir,I_\ve^{R,+}}.
    \end{aligned}
    \end{equation}
We prove the first line of \eqref{eq:o15} with the second following a symmetric proof. To get the first inequality, compare to the definition of $\text{Int}^+(I_\ve,t)$ in \eqref{eq:o14} and note that we just replaced $\varphi(\dir')$ with the larger value $\varphi^M$. Then, the inequality follows immediately from the fact that $x \mapsto W^\dir(0,0;x,T) - W^{\dir'}(0,0;x,T)$ is nondecreasing for $\dir < \dir'$ (Lemma \ref{lem:7842}). 

Now, to get the second inequality in the first line of \eqref{eq:o15}, fix $\dir' < \inf I_{\vep}$. Let $t<0$ and let $y > x:=\tau_t^{R,(\varphi^M,\varphi_0),I_\ve^{L,-},\dir}$.  Then, by Lemma \ref{lem:7842} and the fact that the Busemann functions are eternal solutions,
        \be \label{eq:89345}\begin{aligned}
            W^{\dir}(0,0;y,t)+\varphi_0 &> W^{I_\ve^{L,-}}(0,0;y,t)+\varphi^M=-W^{I_\ve^{L,-}}(y,t;0,0)+\varphi^M\\
            &=-\sup_{z \in \R}\left [ W^{I_\ve^{L,-}}(y,t;z,t)+\Ll(z,t;0,0)\right ]+\varphi^M
        \end{aligned}
        \ee
        Observe that $\varphi^M\geq \varphi_0$ which, in particular, implies that $W^{I_\ve^{L,-}}(0,0;0,0)+\varphi^M \geq W^{\dir}(0,0;0,0)+\varphi_0$.  Therefore, $\tau_0^{R,(\varphi^M,\varphi_0),I_\ve^{L,-},\dir} \geq 0$. 
        Then, by Lemma \ref{lemma 4.1}, we have that $g_{(0,0)}^{b,L}(t) \le x$, where $b$ is the eternal solution
        \[
        b(x,t) = \max\Bigl\{\varphi^M + W^{I_\ve^{L,-}}(0,0;x,t), \varphi_0 + W^{\dir}(0,0;x,t)\Bigr\}.
        \]
Then, by the monotonicity of Busemann geodesics in Lemma \ref{lem:SIG_facts_collect}\ref{it:xi_g_mont}, followed by Lemma  \ref{lm:o1},
        \be \label{eq:6943}
        g^{\dir',L}_{(0,0)}(t) \leq g^{I_\ve^{L,-},L}_{(0,0)}(t) = g_{(0,0)}^{b,L}(t) \leq x.
        \ee
        Then, by \eqref{eq:6943} followed by monotonicity of the Busemann process, then another application of  \eqref{eq:6943},
        \be \label{eq:67812}
        \begin{aligned}
            &\quad \;\sup_{z \in \R}\left [ W^{I_\ve^{L,-}}(y,t;z,t)+\Ll(z,t;0,0)\right ]=\sup_{z \leq x }\left [ W^{I_\ve^{L,-}}(y,t;z,t)+\Ll(z,t;0,0)\right ]\\
            & \le\sup_{z \leq x }\left [ W^{\dir'}(y,t;z,t)+\Ll(z,t;0,0)\right ] =\sup_{z \in \R }\left [ W^{\dir'}(y,t;z,t)+\Ll(z,t;0,0)\right ]= W^{\dir'}(y,t;0,0).
        \end{aligned}
        \ee
        Combining \eqref{eq:89345} and \eqref{eq:67812}, we have
        \begin{align*}
          &W^{\dir}(0,0;y,t)+\varphi_0> W^{\dir'}(0,0;y,t)+\varphi^M,
        \end{align*}
        and therefore $\tau_t^{R,(\varphi^M,\varphi_0),\dir',\dir} \le y$. Since this holds for all $y > x = \tau_t^{R,(\varphi^M,\varphi_0),I_\ve^{L,-},\dir}$, this completes the proof of the second inequality of the first line in \eqref{eq:o15}.  
    \end{proof}

\subsection{Dominating colors} \label{sec:dominating}
       Recall the definition of dominating colors in Definition \ref{def:dominating}.

\begin{lemma}\label{lm:o5}
   For all $\varphi \in \Phi_\omega$, the following hold.
    \begin{enumerate} [label=(\roman*), font=\normalfont]
    \item \label{it:o1}$\dirs^{\Wf{\varphi}}$ is dense in $\supp(\varphi)$.
        \item \label{it:o2}$\dirs^{\Wf{\varphi}}=\Dom(\varphi)$.
    \end{enumerate}
\end{lemma}
\begin{proof}
\textbf{Item \ref{it:o1}:} From Lemma \ref{lm:o2}, for every $\dir\in\supp(\varphi)$ and for every $\ve>0$ there exists $(x,t)\in\R^2$ such that
 \begin{equation}\label{eqo21}
        \Wf{\varphi}(x,t)=\Wf{{\varphi_{I_\ve}}}(x,t)>\Wf{{\varphi_{I^c_\ve}}}(x,t).
    \end{equation}
    By Lemma \ref{lm:o1}, this implies that $\dir^{L}_{x,t}\in I_\ve(\dir)$, which in turn shows that there exists $\dir'\in I_\ve(\dir)$ such that $\dir'\in\dirs^{\Wf{\varphi}}$. This shows that $\dirs^{\Wf{\varphi}}$ is dense in  $\supp(\varphi)$. 

    \medskip \noindent \textbf{Item \ref{it:o2}:} We start by showing that 
      $\Dom(\varphi) \subseteq \dirs^{\Wf{\varphi}}$. Let $\xi_0 \in \Dom(\varphi)$. By definition of this set, for all sufficiently small $\ve > 0$,
      \be \label{eq:dir0_min}
      \dir_0 = \inf\{\dir \in \Xs_\omega: \varphi(\dir) = (W\star \varphi_{I_{\ve}(\dir_0)})(0,0) = \sup_{\dir \in I_\ve(\dir_0)} \varphi(\dir)\},
      \ee
      where the second equality above is just the definition to aid the reader (note that $W^\dir(0,0;0,0) = 0$).  
       Next, for all such $\ve > 0$, when $t < 0$, we observe that 
       \be \label{eq:phi_var}
      \begin{aligned}
          \varphi(\dir_0) &= (W\star \varphi_{I_\ve(\dir_0)})(0,0)  \\
          &\ge (W\star \varphi_{I_\ve(\dir_0)})\bigl(g_{(0,0)}^{\dir_0,L}(t),t\bigr) + \Ll\bigl(g_{(0,0)}^{\dir_0,L}(t),t;0,0\bigr) \\
          &\ge \varphi(\dir_0) + W^{\dir_0}\bigl(0,0;g_{(0,0)}^{\dir_0,L}(t),t\bigr) + \Ll\bigl(g_{(0,0)}^{\dir_0,L}(t),t;0,0\bigr) \\
          &= \varphi(\dir_0) + \sup_{y \in \R}[W^{\dir_0}(0,0;y,t) + \Ll(y,t;0,0)] = \varphi(\dir_0). 
      \end{aligned}
      \ee
      The first equality used \eqref{eq:dir0_min}, the first inequality used the fact that $W\star \varphi_{I_\ve(\dir_0)}$ is an eternal solution (Lemma \ref{lem:interchange_sups} and Proposition \ref{prop:Wstarphi_bd}),  the second inequality is the variational definition of $W\star \varphi_{I_\ve(\dir_0)}$, the second equality comes by the definition of the semi-infinite geodesics as maximizers, and the third uses the fact that Busemann functions are eternal solutions, along with the fact that $W^\dir(0,0;0,0) = 0$. Comparing the first and last term, we see that all inequalities in \eqref{eq:phi_var} are equalities. In particular, by comparing the second and third lines, we obtain
      \be \label{eq:8912}
        f_{\dir_0}\bigl(g_{(0,0)}^{\dir_0,L}(t),t\bigr) = \varphi(\dir_0) + W^{\dir_0}(0,0;g_{(0,0)}^{\dir_0,L}(t),t) = (W\star \varphi_{I_\ve(\dir_0)})\bigl(g_{(0,0)}^{\dir_0,L}(t),t\bigr).
      \ee
      Now, if $\dir \in I_{\ve}(\dir_0)$ with $\dir < \dir_0$, then $\varphi(\dir) < \varphi(\dir_0)$ by \eqref{eq:dir0_min}. Then, using \eqref{eq:phi_var} in the first equality below, we have 
      \begin{align*}
          &\quad \; (W\star \varphi_{I_\ve(\dir_0)})\bigl(g_{(0,0)}^{\dir_0,L}(t),t\bigr) + \Ll\bigl(g_{(0,0)}^{\dir_0,L}(t),t;0,0\bigr)
        \\&= \varphi(\dir_0)  > \varphi(\dir) 
          = \varphi(\dir) + W^\dir(0,0;0,0) \\
          &\ge \varphi(\dir) + W^\dir\bigl(0,0;g_{(0,0)}^{\dir_0,L}(t),t\bigr) + \Ll\bigl(g_{(0,0)}^{\dir_0,L}(t),t;0,0\bigr).
      \end{align*}
      Rearranging, we have 
      \be \label{eq:6892}
            f_{\dir}\bigl(g_{(0,0)}^{\dir_0,L}(t),t\bigr) = \varphi(\dir) + W^\dir\bigl(0,0;g_{(0,0)}^{\dir_0,L}(t),t\bigr) < (W\star \varphi_{I_\ve(\dir_0)})\bigl(g_{(0,0)}^{\dir_0,L}(t),t\bigr).
      \ee
      Hence, by comparing \eqref{eq:8912} and \eqref{eq:6892}, we see that
      \be \label{78924}
      \dir_0 = \inf\Bigl\{\dir \in I_\ve(\dir_0):f_{\dir}\bigl(g_{(0,0)}^{\dir_0,L}(t),t\bigr) = (W\star \varphi_{I_\ve(\dir_0)})\bigl(g_{(0,0)}^{\dir_0,L}(t),t\bigr)  \Bigr\}.
      \ee
By Lemma \ref{lm:o2} (using both \eqref{eq:o391} and \eqref{eq:o23}), there exists $T < 0$ so that  
\begin{align*}
(\Wf{\varphi})\bigl(g_{(0,0)}^{\dir_0,L}(T),T\bigr)=\Wf{\varphi_{I_\ve(\dir_0)}}\bigl(g_{(0,0)}^{\dir_0,L}(T),T\bigr)>\Wf{\varphi_{I_\ve(\dir_0)^c}}\bigl(g_{(0,0)}^{\dir_0,L}(T),T\bigr).    
\end{align*}
Hence, for $t = T$, we may replace the inf over $I_\ve(\dir_0)$ in \eqref{78924} with $ \Xs_\omega$.  Then, by the definition \eqref{eq:xixt}, we have that $\dir_0 = \dir_{x,t}^L$ for $(x,t) = \bigl(g_{(0,0)}^{\dir_0,L}(T),T\bigr)$. Then, by Lemma \ref{lm:o1}, $\dir_0 \in \dirs^{\Wf{\varphi}}$, as desired.

Next, we show that $\dirs^{\Wf{\varphi}} \subseteq \Dom(\varphi)$. We will assume that $\dir_0 \notin \Dom(\varphi)$ and show that $\dir_0 \notin \dirs^{\Wf{\varphi}}$. Lemma \ref{lm:o1} implies that it suffices to take  arbitrary $(x,t) \in \R^2$ and show that $\dir_0 \neq \dir_{x,t}^L$. Considering \eqref{eq:o17}-\eqref{eq:o18} which defines $\Dom(\varphi)$, there are two cases to consider. 

\medskip \noindent \textbf{Case 1:}
There exists a sequence $\{\dir_n\}$ with $\dir_n\rightarrow\dir_0$ such that $\dir_n>\dir_0$ and $\varphi(\dir_n)>\varphi(\dir_0)$  for every $n$. Then, by the strong continuity of the Busemann process in Proposition \ref{prop:Buse_basic_properties}\ref{itm:general_cts}, for sufficiently large $n$,
\[
    f_{\dir_0}(x,t)<f_{\dir_n}(x,t).
\]
Then, $\dir_0 \notin \Gamma_{x,t}$, so Lemma \ref{lm:o9} implies that $\dir_0\neq \dir_{x,t}^L$.

\medskip \noindent \textbf{Case 2:} there exists a sequence $\{\dir_n\}$ with $\dir_n \to \dir_0$ such that $\dir_n<\dir_0$ and $\varphi(\dir_n)\ge \varphi(\dir_0)$ for every $n$. Then again by continuity of the Busemann process, for $n$ large enough
\[
    f_{\dir_0}(x,t)\le f_{\dir_n}(x,t).
\]
As $\dir_n < \dir_0$, it follows that $\dir_0\neq \dir^{L}_{x,t}$.
\end{proof}

\begin{proposition}\label{prop:uniq}
    The following holds for almost-every $\omega$. Let $b$ be an eternal solution. Then, there exists a unique $ \varphi \in \Phi_\omega$ such that $b = W\star \varphi$. 
\end{proposition}
\begin{proof}
    First, let $\wt \varphi$ be the function constructed in Proposition \ref{prop:exists_varphi} that satisfies $b = W\star \wt \varphi$ and whose support is the countable set $\dirs^b$. We first claim that any upper semicontinuous function $ \varphi \in \Phi$ that satisfies $W\star \varphi = b$ must be an extension of $\wt \varphi$; that is, $ \varphi(\xi) =  \wt \varphi(\xi)$ for all $\xi \in \dirs^{b}$. 

    To see this, let $\xi \in \dirs^{b}$, and choose $(x,t) \in \mathfrak C^b(\xi)$. Then, by applying Proposition \ref{prop:exists_varphi}, followed by the assumption $W\star \wt \varphi = W\star  \varphi$, then using Corollary \ref{cor:o1}, we have that 
    \[
    \wt \varphi(\xi)  = b(x,t) - W^\xi(0,0;x,t) = (W\star  \varphi)(x,t) - W^\xi(0,0;x,t) =  \varphi(\xi),
    \]
    which proves our claim. 

     Lemma \ref{lem:one_extension} shows that one such upper semicontinuous extension is given by
    \[
     \varphi(\xi) := \limsup_{\eta \to \xi} \wt \varphi(\eta).
    \]
     Let $\wh \varphi$ be another upper semicontinuous extension satisfying $W \star \wh \varphi = b$. Then, $\wh \varphi \ge \wt \varphi$ since $\wt \varphi(\dir) = -\infty$ whenever $\wt\varphi(\dir) \neq \wh \varphi(\dir)$. Along with the assumed upper semicontinuity of $\wh \varphi$, this implies that
    \be \label{eq:other_extension}
    \widehat \varphi(\xi) \ge \limsup_{ \eta \to \xi} \widehat \varphi(\eta) \ge \limsup_{\eta \to \xi} \wt \varphi(\eta) = \varphi(\xi). 
    \ee
 We will therefore assume, by way of contradiction, that $\widehat \varphi(\xi_0) > \varphi(\xi_0)$ for some $\xi_0 \in \Xs$.  Then, define a new function $\varphi'$ as follows:
\[
\varphi'(\xi) := \begin{cases}
    \wh \varphi(\xi) &\xi = \xi_0 \\
     \varphi(\xi) &\xi \neq \xi_0.
\end{cases}
\]
By upper semi-continuity of $ \varphi$, we immediately see that $\varphi'$ is upper semicontinuous at $\xi \neq \xi_0$, and for $\xi = \xi_0$, we have 
\be \label{phi'_ge}
\varphi'(\xi_0) >  \varphi(\xi_0) \ge \limsup_{\xi \to \xi_0}  \varphi(\xi) \ge\limsup_{\xi \to \xi_0, \xi \neq \xi_0} \varphi'(\xi),
\ee
where we used the upper semicontinuity of $\varphi$ in the second inequality. Thus, $\varphi'$ is in fact upper semicontinuous on all of $\Xs$. Furthermore, by definition of $\varphi'$ and \eqref{eq:other_extension}, we have $\varphi \le \varphi' \le \wh \varphi$, so since $W \star \varphi = W\star \wh \varphi = b$, we have $W \star \varphi' = b$ as well. By definition of $\varphi'$ and \eqref{phi'_ge}, $\xi_0 \in \Dom(\varphi')$, so Lemma \ref{lm:o5} implies that $\xi_0 \in \dirs^{b}$.  However, we previously showed that $ \varphi(\xi) = \wh \varphi(\xi)$ for all $\xi \in \dirs^{b}$, a contradiction to the assumption that $\wh \varphi(\xi_0) > \varphi(\xi_0)$.  
\end{proof}

\begin{proposition} \label{prop:Legendre}
    The following holds for almost-every $\omega$. Let $b$ be an eternal solution, and let $\varphi$ be the unique element of $\Phi_\omega$ such that $b = W\star \varphi$, as in Proposition \ref{prop:uniq}. Then, 
     \begin{equation}
        \varphi(\dir)=\inf_{(x,t)\in\R^2}\{b(x,t)-W^\dir(0,0;x,t)\}, \qquad \forall \dir\in \Xs.
     \end{equation}
\end{proposition}

\begin{proof}
Define $\theta:\Xs \to [-\infty,\infty)$ by
\[
\theta(\dir) := \inf_{(x,t) \in \R^2}[b(x,t) - W^\dir(0,0;x,t)].
\]
Since $b$ and $W$ are both continuous (Corollary \ref{cor:cts_bounded} and Proposition \ref{prop:Buse_basic_properties}\ref{itm:general_cts}), $\theta$ is the infimum over continuous functions and is therefore upper semicontinuous. Now, note that, for all $\dir \in \Xs$ and $(x,t) \in \R^2$, 
\[
\varphi(\dir) + W^{\dir}(0,0;x,t) \le b(x,t).
\]
Hence, by rearranging and taking an inf over $(x,t) \in \R^2$, we see that $\varphi(\dir) \le \theta(\dir)$, and thus $b(x,t) = (W\star \varphi)(x,t) \le (W\star \theta)(x,t)$ for all $(x,t) \in \R^2$. On the other hand, $\theta(\dir) \le b(x,t) - W^\dir(0,0;x,t)$ for all $\dir \in \Xs$ and $(x,t) \in \R^2$, so
\[
(W\star \theta)(x,t) \le \sup_{\dir \in \Xs}[b(x,t) - W^{\dir}(0,0;x,t) + W^\dir(0,0;x,t)] = b(x,t).
\]
Thus, $W\star \varphi = W\star \theta$, so by upper semicontinuity of $\theta$, Proposition \ref{prop:uniq} implies that $\varphi = \theta$.
\end{proof}

\section{Proof of Theorem \ref{thm:mr}} \label{sec:main_proof}
In this section, we prove Theorem \ref{thm:mr}. We first prove some topological preliminaries. Denote the metric on $\Xs$ from Lemma \ref{lem:X_metric} by $d_{\Xs}$. Furthermore, we equip the space $\Rm := [-\infty,\infty)$ with the following metric
\begin{equation}
    d_{\Rm}(x,y)=|\arctan(x)-\arctan(y)|, \qquad \forall x,y\in\Rm,
\end{equation}
where $\arctan(-\infty)=-\frac{\pi}{2}$. This choice of a metric guarantees that $ d_{\Rm}(x,x+\ve)\leq \ve$.

We now recall the Painlev\'e-Kuratowski (PK) topology (see, for example, \cite[Chapter 4.B]{RockafellarWets1998}) on the set of closed sets in  a topological space $Y$. For a sequence $F_n\subseteq Y$ of closed sets, we define
\begin{align*}
    \limsup_{n \to \infty} F_n &= \Bigl\{x \in Y: \exists \text{ a subsequence }n_k \text{ such that } \forall k,\; \exists\; x_{n_k} \in F_{n_k} \text{ with }\lim_{k \to \infty} x_{n_k} = x \Bigr\},\quad\text{and} \\
    \liminf_{n \to \infty} F_n &= \Bigl\{x \in Y: \exists N \in \N \text{ such that }\forall \;n \ge N,\exists\; x_n \in F_n\text{ such that }\lim_{n \to \infty} x_n = x  \Bigr\}.
\end{align*}
We say that $F_n\subseteq Y$ converges to $F\subseteq Y$ in the PK topology if 
\begin{equation}
    \limsup F_n=\liminf F_n=F.
\end{equation}
When $Y$ is a compact metric space then the PK topology is equivalent to the Hausdorff topology associated with the Hausdorff metric \cite[Example 4.13]{RockafellarWets1998}. 
We are now in a position to define the topology on the space $\fun$ defined in \eqref{eq:Phi_omega_def}. For $\varphi\in\fun$ we define the hypograph of $\varphi$ to be 
\begin{equation}
    \text{hypo}(\varphi)=\{(x,y):x\in\Xs,-\infty\leq y\leq \varphi(x)\}.
\end{equation}
 It is a standard fact that for $\varphi\in\operatorname{USC}(\Xs,[-\infty,\infty))$, the set $\text{hypo}(\varphi)$ is closed.  As every closed and bounded subset of $\Xs$ is compact, arguments similar to the ones in \cite[Chapter 4.I]{RockafellarWets1998} can be used to construct a metric $d_{\text{PK}}$ compatible with the hypograph topology on $\operatorname{USC}(\Xs,[-\infty,\infty))$. Furthermore, one can readily see that $\operatorname{hypo} \varphi_n \to \operatorname{hypo} \varphi$ is equivalent to the following two conditions:
 \begin{align}
     &\text{Whenever $\xi_n \to \xi$, we have $\limsup_{n \to \infty} \varphi_n(\xi_n) \le \varphi(\xi)$} \label{lephi} \\
     &\text{For each $\dir \in \Xs$, there exists a sequence $\xi_n \to \xi$ so that $\lim_{n \to \infty}\varphi_n(\xi_n) = \varphi(\xi)$.} \label{recov_phi}
 \end{align}

 \begin{definition} \label{def:fun_convergence}
     Let $\mathcal T_\fun$ denote the topology on the space $\fun$ such that $\varphi_n \to \varphi$ if the following two conditions hold: 
\begin{enumerate}
    \item \label{it:c1} For every $R\in\R$ there exists $I(R,\varphi)\subseteq\Xs$ compact and $N_0(R,\varphi)>0$  such that\ 
    \begin{equation}
        \frac{\varphi_n(\dir)}{\mathcal P(\dir)^2}\leq R, \qquad \forall \dir\in I^c, n\geq N_0.
    \end{equation}
    \item \label{it:c2}$\text{hypo}(\varphi_n)\xrightarrow{\operatorname{PK}}\text{hypo}(\varphi)$,  as $n\rightarrow \infty$.
\end{enumerate}
 \end{definition}
 \begin{remark} \label{rmk:topology}
     One can see that Definition \ref{def:fun_convergence} indeed defines a topology as follows: Consider the homeomorphism $F:\R \to (-1,1)$ defined by $F(x) = \f{x}{\sqrt{1 + x^2}}$. Letting $\Xi' = F(\Xi)$, we can construct a homeomorphic copy of the space $\Xs$ by defining 
     \[
     \Xs' = \{\eta^0: \eta \in (-1,1) \cap (\Xi')^c\} \cup \{\eta^\sig: \eta \in (-1,1) \cap \Xi', \sigg \in \{-,+\}\}.
     \]
     Lift $F$ to a bijection $\wt F:\Xs \to \Xs'$, and define the topology on $\Xs'$ so that $\wt F$ is a homeomorphism.  Let $\mathcal P':\Xs'\to (-1,1)$ be the analogous projection map in $\Xs'$. Next, we compactify this space by setting $\overline{\Xs}' = \Xs' \cup\{-1,1\}$. Then, the convergence in Definition \ref{def:fun_convergence} is equivalent to convergence in the PK topology of the hypographs of  $\varphi_n'$ to the hypograph of $\varphi'$, where for each $\varphi:\Xs \to [-\infty,\infty)$, $\varphi':\overline{\Xs}' \to [-\infty,\infty)$ is the function
     \[
     \varphi'(\xi) = \begin{cases} (1-\mathcal P'(\xi)^2)\varphi(\wt F^{-1}(\xi)) &\xi \in \Xs'\\
     -\infty &\xi \in \{-1,1\}
     
     \end{cases}
     \]
     Indeed, convergence of the hypographs away from $\xi = \pm 1$ is equivalent to Condition \eqref{it:c2}, and convergence near $\xi = \pm 1$ is equivalent to Condition \eqref{it:c1} because $\overline{\Xs}'$ is compact and \eqref{lephi} gives us that for all sequences $\dir_n$ converging to $1$ or $-1$,
     \[
     \limsup_{n \to \infty}(1 - \mathcal P'(\dir_n)^2) \varphi(\wt F^{-1}(\dir_n)) = -\infty.
     \]
     But since $F^{-1}(y) = \f{y}{\sqrt{1-y^2}}$, this is equivalent to saying that for all sequences $\dir_n$ converging to $-\infty$ or $+\infty$, we have 
     \[
        \limsup_{n \to \infty}\f{\varphi_n(\dir_n)}{\mathcal P(\dir_n)^2} = -\infty.
     \]
     In particular, using similar ideas from Lemma \ref{lem:X_metric}, the space $\overline{\Xs}'$ can be seen to be metrizable. Then, PK convergence of hypographs is metrizable, the convergence given in Definition \ref{def:fun_convergence} induces defines a metric topology on $\Phi$.   
 \end{remark}

We shall need the following result that gives a sufficient condition for a family of functions in $\fun$ to be sequentially compact.
\begin{lemma}\label{lm:o4}
 Let $\mathcal{F}\subseteq \fun$ be such that 
 \begin{enumerate} [label=\rm(\roman{*}), ref=\rm(\roman{*})]
     \item Uniform boundedness.
     \begin{equation}\label{eq:o41}
     \sup_{\varphi\in\mathcal{F}}\sup_{\dir\in\Xs}\varphi(\dir)<\infty    
     \end{equation}
     \item For any $R\in\R$ there exists $I_R\subseteq \Xs$ compact such that\ 
     \begin{equation}\label{eq:o42}
         \sup_{\varphi\in\mathcal{F}}\sup_{\dir\in I^c_R}\frac{\varphi(\dir)}{\mathcal P(\dir)^2}<R.
     \end{equation}
 \end{enumerate}
 Then $\mathcal{F}$ is sequentially compact i.e.\ any sequence in $\mathcal{F}$ has a convergent subsequence. 
\end{lemma}
 \begin{proof}
     Let $\{\varphi^n\}\subseteq \mathcal{F}$. For large enough $M_0$, there exist a nested collection of intervals $(I_{M})_{M \ge M_0}$ such that $I_M \subseteq I_{M+1} \subseteq \Xs$ for all $M \ge M_0$ and such that, on $I_M$, the following holds
     \begin{equation}\label{eq:o39}
        \sup_{\varphi\in\mathcal{F}}\sup_{\dir\in I^c_M}\frac{\varphi(\dir)}{\mathcal P(\dir)^2}<-M.
     \end{equation}
  Let $C>\sup_{\varphi\in\mathcal{F}}\sup_{\dir\in\Xs}\varphi(\dir)$. Set 
  \begin{equation}
      A^M_n:=\text{hypo}(\varphi^n)\cap (I_M\times[-\infty,C])
  \end{equation}
  Note that 
  \begin{equation}\label{eq:o47}
      \text{$A^M_n$ is the hypograph of the upper semi-continuous function $\varphi^n|_{I_M}$.}
  \end{equation}
 We claim that there exists a subsequence $\{a^M_n\}$ such that
  \begin{equation}\label{eq:o38}
      A^M_{a^M_n}\xrightarrow{\operatorname{PK}} A^M\subseteq I_M\times[-\infty,C].
  \end{equation}
 where $A^M\subseteq I_M\times[-\infty,C]$ is compact. Indeed, the PK topology on a metric space $Y$ is equivalent to the Hausdorff metric on compact sets in $Y$, whenever $Y$ is compact. This is the case here as  $I_M\times[-\infty,C]$ is compact. But when the metric space $Y$ is compact, any set of closed sets in $Y$ is compact in  the Hausdorff metric and thus one can find a subsequence that is convergent to some compact set in  $I_M\times[-\infty,C]$. We conclude \eqref{eq:o38}. We can now find a subsequence $\{a^{M+1}_n\}\subseteq \{a^M_n\}$ such that\
 \begin{equation}
     A^{M+1}_{a^{M+1}_n}\xrightarrow{\operatorname{PK}} A^{M+1}\subseteq I_{M+1}\times[-\infty,C].
 \end{equation}
 In particular, note that from \eqref{eq:o38}
 \begin{equation}
     A^{M+1}\cap I_{M}\times[-\infty,C]= A^{M}.
 \end{equation}
 Moreover, by \eqref{eq:o39} we see that
 \begin{equation}\label{eq:o40}
     (\dir,c_\dir)\in A^{M+1}\cap I_M^c\times[-\infty,C]\implies c_\dir<-\mathcal P(\dir)^2 M.
 \end{equation}
 Using a diagonal argument  we find a subsequence $c_n$ such that\
 \begin{equation}
     \text{hypo}(\varphi^{c_n})\xrightarrow{\operatorname{PK}} A\subseteq \Xs\times \Rm.
 \end{equation}
 Moreover, by \eqref{eq:o47} $A=\text{hypo}(\varphi^*)$ for some $\varphi^*\in\operatorname{USC}(\Xs)$.  But from \eqref{eq:o40}, it follows that 
 \begin{equation}
     \limsup_{|\dir|\rightarrow\infty} \frac{\varphi^*(\dir)}{\mathcal P(\dir)^2}=-\infty,
 \end{equation}
 which implies $\varphi^*\in\fun$. This concludes the proof. 
 \end{proof}

We are now ready to prove the main result of the paper
\begin{proof}[Proof of Theorem \ref{thm:mr}] 
    The function $\hm$ is a bijection by Proposition \ref{prop:uniq}. Equation \eqref{eq:hom_max} follows immediately from the definition of $\Wf{\varphi}$, and  and the expression for the inverse in \eqref{eq:hom_inv} is shown in Proposition \ref{prop:Legendre}. 
    
    It remains to show that $\hm$ and its inverse are continuous.  Let $\{\varphi_n\}\subseteq\fun$ be a sequence such that  $\varphi_n\xrightarrow{\operatorname{\fun}}\varphi$. 
     Let $b_n=\hm[\varphi_n]$ and $b=\hm[\varphi]$. We must  show that $b_n\rightarrow b$ uniformly on compacts. Without loss of generality, we take our compact set to be $K = [-M,M]^2$ for some $M  \in \Z$. We also make the following preliminary observations. By Lemma \ref{lem:Buse_global}, there exist constants $A,B > 0$ (depending on $M$) so that
     \be \label{eq:Wbd187}
     W^\dir(0,0;x,t) = W^\dir(0,0;0,-M) +  W^\dir(0,-M;x,t) \le A \mathcal P(\dir)^2 + B,\quad \forall\; (x,t) \in K,\dir \in \Xs.
     \ee
     Additionally, Condition \eqref{it:c1} of the definition of convergence $\varphi_n \to \varphi$ implies that 
     \be \label{eq:phin_bd167}
     \varphi_n(\dir) \le -2A\mathcal P(\dir)^2,\quad\text{when } \mathcal |P(\dir)| \text{ is sufficiently large}. 
     \ee
    Now, to show that $b_n \to b$, we handle two cases separately.

     \medskip \noindent \textbf{Case 1: $\varphi \equiv - \infty$}. Let $R < 0$. By \eqref{eq:Wbd187}-\eqref{eq:phin_bd167}, there exists a compact interval $I \subseteq \Xs$ and $N \in \N$ so that for all $n \ge N$, 
     \be \label{eq:phi+WleR}
     \sup_{\dir \notin I} [\varphi_n(\dir) +W^\dir(0,0;x,t)] \le R.
     \ee
     On the other hand, continuity of the process $(\dir,x,t) \mapsto W^\dir(0,0;x,t)$ (Proposition \ref{prop:Buse_basic_properties}\ref{itm:general_cts}) implies that  $W^\dir(0,0;x,t)$ is bounded for $(x,t) \in K$ and $\dir \in I$. Since $I$ is compact, by the hypograph convergence of $\varphi_n \to \varphi \equiv - \infty$, there exists $N \in \N$ so that for all $n \ge N$,
     \be \label{eq:phioutsideI}
     \sup_{\dir \in I} \varphi_n(\dir) \le R -\sup_{(x,t) \in K,\dir \in I} W^\dir(0,0;x,t).
     \ee
     Combining \eqref{eq:phi+WleR}-\eqref{eq:phioutsideI}, we see that $b_n = W\star \varphi_n$ converges to $-\infty$ uniformly on $K$.

     \medskip \noindent \textbf{Case 2: $\varphi \not \equiv - \infty$}
     We first show that $K$ intersects $\clr^b(\dir)$ for only finitely many $\dir \in \dirs^b$. By the limits in Lemma \ref{lem:SIG_facts_collect}\ref{it:limits_to_infty}, there exists $x_1 < -M$ and $x_2 > M$ so that 
     \[
     g_{(x_1,M)}^{b,L}(s) < -M < M < g_{(x_2,M)}^{b,L}(s),\quad\forall\; s \in [-M,M].
     \]
     Using the monotonicity in Lemma \ref{lem:SIG_facts_collect}\ref{it:xi_g_mont} and the consistency of geodesics in Lemma \ref{lem:geodesics_from_b}\ref{itm: consis}, we see that
     \[
     g_{(x_1,M)}^{b,L}(s) < g_{(x,t)}^{b,L}(s) < g_{(x_2,M)}^{b,L}(s),\quad\forall (x,t) \in K, s \in [-M,t]. 
     \]
     Setting
     \[
     y_1 :=  g_{(x_1,M)}^{b,L}(-M),\quad \text{and}\quad y_2 := g_{(x_2,M)}^{b,L}(-M),
     \]
     we see that 
     \[
     \{\xi_{x,t}^b:(x,t) \in K\} \subseteq \{\xi_{y,-M}^b: y \in [y_1,y_2]\},
     \]
     and by Lemma \ref{lem:2}, the set on the right-hand side is finite. Hence, it suffices to show uniform convergence of $b_n \to b$ on the set $K \cap \clr^b(\dir)$ for each $\dir \in \dirs^b$. Letting $\ve > 0$, we need to show that there exists $N \in \N$ so that
     \be \label{eq:unif_conv_up_lbd}
        -\ve \le b_n(x,t) - b(x,t) \le \ve,\quad \forall\;(x,t) \in K \cap \clr^b(\dir) , n \ge N.
     \ee
     We prove the upper and lower bounds separately. By the convergence $\text{hypo}(\varphi_n) \to \text{hypo}(\varphi)$,  \eqref{recov_phi} implies that there exists a sequence $\dir_n \to \dir$ such that $\varphi_n(\dir_n) \to \varphi(\dir)$. By the uniform convergence of the Busemann functions in Proposition \ref{prop:Buse_basic_properties}\ref{itm:general_cts}, there exists $N \in \N$ so that for all $n \ge N$,
\[
b_n(x,t) \ge \varphi_n(\dir_n) + W^{\dir_n}(0,0;x,t) \ge \varphi(\dir) + W^\dir(0,0;x,t) -\ve = b(x,t) -\ve,  \quad \forall \; (x,t) \in \clr^b(\dir) \cap K,
\]
where we used Corollary \ref{cor:o1} in the last equality. This gives the lower bound in \eqref{eq:unif_conv_up_lbd}.

Now, assume by way of contradiction, that the upper bound in \eqref{eq:unif_conv_up_lbd} does not hold. Then, there exists a subsequence $n_k$ and a sequence $\{(x_k,t_k)\} \subseteq K$ so that, for all $k \in \N$,
\be \label{eq:xik}
\varphi_{n_k}(\xi_k) + W^{\xi_k}(0,0;x_k,t_k) = b_{n_k}(x_k,t_k)  \ge  b(x_k,t_k) + \ve,
\ee
where $\xi_k = \xi_{x_k,t_k}^b$ (using Corollary \ref{cor:o1} in the first equality). We first show that $\xi_k$ must be bounded. Since $(x_k,t_k) \in K$,  if $|\mathcal P(\xi_k)| \to \infty$ along some subsequence, then \eqref{eq:Wbd187}-\eqref{eq:phin_bd167},  implies that
\[
\varphi_{n_k}(\xi_k) + W^{\xi_k}(0,0;x_k,t_k) \le -2A \mathcal P(\xi_k)^2 + A \mathcal P(\xi_k)^2 + B \to -\infty,
\]
a contradiction to \eqref{eq:xik}. Hence, $\dir_k$ is bounded, as desired. Since $K$ is compact, there exists a subsequential limit  $(\dir_k,x_k,t_k) \to (\dir,x,t)$ (where we have assumed the convergence holds along $k$ without loss of generality). By the hypograph convergence $\text{hypo}(\varphi_n) \to \text{hypo}(\varphi)$, \eqref{lephi}, continuity of the Busemann process, and continuity of finite eternal solutions (Corollary \ref{cor:cts_bounded}),   
\[
\limsup_{k \to \infty} \varphi_{n_k}(\xi_k) + W^{\xi_k}(0,0;x_k,t_k) \le \varphi(\dir) + W^\dir(0,0;x,t) \le b(x,t) = \lim_{k \to \infty} b(x_k,t_k),
\]
a contradiction to \eqref{eq:xik}. 

 We are left to show that if a sequence $\varphi_n\nrightarrow \varphi$ then $b_n := \hm[\varphi_n] \nrightarrow \hm[\varphi] :=b$. For $\varphi_n$ not to converge to $\varphi$ in $\fun$, one of the conditions of Definition \ref{def:fun_convergence} must fail. We consider two cases:

 \medskip \noindent \textbf{Case 1: Condition \eqref{it:c1} fails.}
     Then there exists $R>0$,  a subsequence  $n_k$, and a sequence $\dir_k$ such that\ $|\mathcal P(\dir_{k})|\rightarrow \infty$ and $\varphi_{n_k}(\dir_{k})\geq -\mathcal P(\dir_{k})^2R$ for all $k \in \N$.  Then, by the $s = 0$ case of Lemma \ref{lem:Buse_global} gives us that $W^{\dir_{k}}(0,0;0,t) \ge (t-1)\mathcal P(\dir_{k})^2 - B$ for all $t > 0$, where $B$ is a constant. Then, for $t > R + 1$, we have 
\[
\liminf_{k \to \infty} b_{n_k}(0,t) \ge \liminf_{k \to \infty} \varphi_{n_k}(\dir_k) + W^{\dir_k}(0,0;0,t) = \infty,
\]
so $b_n \nrightarrow b$. 

\medskip \noindent
     \textbf{Case 2: Condition  \eqref{it:c1} holds but \eqref{it:c2} fails.} 
     We consider two further subcases.

     \medskip \noindent \textbf{Case 2.1:}
     Assume there exists a sequence $\{\dir_n\}$ such that $\varphi_n(\dir_n) \to \infty$. Then, 
     \begin{equation}
          b_n(0,0)\geq \varphi_n(\dir_n) + W^{\dir_n}(0,0;0,0) = \varphi_n(\dir_n) \rightarrow \infty, 
     \end{equation}
     so $b_n\nrightarrow b$.
     
     \medskip \noindent \textbf{Case 2.2:} We assume that
     \begin{equation}\label{eq:o44}
     \sup_{\dir\in\Xs, n \in \N}\varphi_n(\dir)<\infty.     
     \end{equation}
     As condition  \eqref{it:c1} and \eqref{eq:o44} hold,  Lemma \ref{lm:o4} implies that $\{\varphi_{n}\}$ is sequentially compact. There must be some subsequential limit $\varphi'$ that is not equal to $\varphi$ since we assumed that $\varphi_n \nrightarrow \varphi$. Since $\varphi_{n_k} \to \varphi'$ along some subsequence, we showed earlier in this proof that $\hm[\varphi_{n_k}] \to \hm[\varphi']$. But since $\hm$ is a bijection, $\hm[\varphi'] \neq \hm[\varphi] = b$, so $b_n \nrightarrow b$.  
     \end{proof}

\section{Proof of Theorem \ref{thm:shocks_coal}} \label{sec:coal_proofs}
In this section, we prove the results about the coloring map in Theorem \ref{thm:shocks_coal}.  We prove each of the items separately. For these proofs, recall Definition \ref{def:color} that $\dir_{x,t}^b$ is the color of the geodesic $g_{(x,t)}^{b,L}$. We also recall the definition in \eqref{eq18} that $\dirs^{b,t} = \{\dirs_{x,t}^b: x \in \R\}$ for each $t \in \R$.

\begin{proof}[Proof of Theorem \ref{thm:shocks_coal}\ref{it:Cb_set}] 
Note that $\dirs^b$ is nonempty if and only if $b$ is a finite eternal solution, so we will assume this is the case for the remainder of the proof. The equality \eqref{eq:Cb_intro} was previously shown in \eqref{eq:Cb_rep}. The fact that $\clr^b(\dir)$ is unbounded follows from Lemma \ref{lm:4}: specifically, for $\dir \in \dirs^b$. we have $\dir \in \dirs^{b,s}$ for all $s < \tex(\dir)$, so for each $s < \tex(\dir)$, there exists $x \in \R$ so that $(x,s) \in \clr^b(\dir)$. 

Next, we show that if $\rs_\dir^{s'} < \infty$ for some $\dir \in \dirs^b$ and $s' < \tex(\dir)$, then $\rs_\dir^{s} < \infty$ for all $s < \tex(\dir)$. Since $\rs_{\dir}^{s'} < \infty$, the monotonicity of colors in Lemma \ref{lem:5} implies that there exists $\dir' > \dir$ so that $\dir' \in \dirs^{b,s'}$. Then, by Lemma \ref{lm:4}, $\dir' \in \dirs^{b,s}$ for all $s < s'$, so using the monotonicity of colors again, we must have $\rs_{\dir}^s < \infty$ for $s < s'$. On the other hand,  if $\rs_\dir^s = \infty$ for some $s \in (s',\tex(\dir))$, then there exists $M \in \R$ so that $\dir_{x,s}^b =\dir$ for all $x \ge M$. By Lemma \ref{lem:SIG_facts_collect}\ref{it:some_g}, this implies that $g_{(x,s)}^{b,L} = g_{(x,s)}^{\dir,L}$ for all $x \ge M$. By Lemma \ref{lem:SIG_facts_collect}\ref{it:limits_to_infty}, $x' := g_{(x,s)}^{b,L}(s') > \rs_\dir^{s'}$ for sufficiently large $x'$. But then, by the consistency of geodesics in Lemma \ref{lem:geodesics_from_b}\ref{itm: consis}, $\dir_{x',t'}^b = \dir$, a contradiction to $x' > \rs_\dir^{s'}$. Hence, we indeed have that $\rs_\dir^s < \infty$ for all $s < \tex(\dir)$. By a symmetric proof, if $\ls_\dir^{s'} > -\infty$ for some $s' < \tex(\dir)$, then $\ls_\dir^s > -\infty$ for all $s < \tex(\dir)$. 

Lastly, we prove the continuity of $s \mapsto \ls_\dir^s$ and $s \mapsto \rs_\dir^s$, which along with the definition that $\ls_\dir^s < \rs_\dir^s$ also proves that $\clr^b(\dir)$ is connected. Suppose, by way of contradiction and without loss of generality, that $t \mapsto \rs_\dir^t$ is discontinuous at some $t_0 < \tex(\dir)$. 
    We consider four cases.

    \medskip \noindent 
    \textbf{Case 1:} There exists $\ve > 0$ and a sequence $t_n\uparrow t_0$ such that  $\rs^{t_0}_{\dir}-\rs^{t_n}_{\dir}>\varepsilon$ for all $n\in\N$. By definition of $\rs_\dir^{t_0}$ in \eqref{cm}-\eqref{eq:Cb_rep}, $\dir_{\rs_{\dir}^{t_0},t_0}^b = \dir$. For each $n \in \N$, let $y_n := g^{b,L}_{(\rs^{t_0}_\dir,t_0)}(t_n)$. By the consistency of geodesics in Lemma \ref{lem:geodesics_from_b}\ref{itm: consis}, $\dir_{y_n,t_n}^b = \dir$.  On the other hand, since $g^{b,L}_{(\rs^{t_0}_\dir,t_0)}$ is continuous, for $n$ large enough it must be that 
    \[
        y_n  > \rs_\dir^{t_0} - \ve > \rs_\dir^{t_n}.
    \]
    Thus, $\dir_{y_n,t_n}^b >\dir$, a contradiction. 

    \medskip \noindent 
    \textbf{Case 2:} There exists $\ve > 0$ and a sequence $t_n\uparrow t_0$ such that  $\rs^{t_n}_{\dir}-\rs^{t_0}_{\dir}>\varepsilon$ for all $n\in\N$. Pick $x_0\in(\rs^{t_0}_\dir,\rs^{t_0}_\dir+\varepsilon/2)$ so that $\dir_{x_0,t_0}^b > \dir$.  On the other hand,  continuity of geodesics implies that $g^{b,L}_{(x_0,t_0)}(t_n)\le \rs^{t_n}_\dir$ for sufficiently large $n$. Again using the consistency of geodesics in Lemma \ref{lem:geodesics_from_b}\ref{itm: consis},  this implies that $\dir_{x_0,t_0}^b \le \dir$, a contradiction.

    \medskip \noindent \textbf{Case 3:} There exists $\ve > 0$ and a sequence $t_n\downarrow t_0$ such that  $\rs^{t_0}_{\dir}-\rs^{t_n}_{\dir}>\varepsilon$ for all $n\in\N$. Set $x  = \rs^{t_0}_{\dir}-\f{\vep}{2}$. Then $x > \rs_\dir^{t_n}$ for all $n$, and so $\dir_{x,t_n}^b > \dir$. Recalling that 
    \[
    g_{(x,t_n)}^{b,L}(t_0 - 1) = \inf \arg\max_{y \in \R}[b(y,t_0 - 1) + \Ll(y,t_0-1;x,t_n)],
    \]
    the finiteness of the eternal solution together with Lemma \ref{lem:general_restrict_max} implies that the sequence $\{g_{(x,t_n)}^{b,L}(t_0 - 1)\}_{n \in \N}$ is bounded. Thus, there exists a subsequence along which $g_{(x,t_{n})}^{b,L}(t_0 - 1)$ is convergent. Then, by Lemma \ref{lem:precompact}, there exists a further subsequence along which the graph of the restricted geodesic $g_{(x,t_{n})}^{b,L}|_{[t_0-1,t_n]}$ converges with respect to the Hausdorff distance to a geodesic $\wt g$ whose upper endpoint is $(x,t_0)$. Hence, for all sufficiently large $n$, we must have $g_{(x,t_n)}^{b,L}(t_0) \le x + \f{\ve}{4} < \rs_\dir^{t_0}$, and just as in the previous cases, the consistency of geodesics implies that $\dir_{x,t_n}^b \le \dir$, a contradiction.

    \medskip \noindent 
     \textbf{Case 4:} There exists a sequence $t_n\downarrow t_0$ such that  $\rs^{t_n}_{\dir}-\rs^{t_0}_{\dir}>\varepsilon$ for all $n\in\N$. Set  $x= \rs^{t_0}_{\dir}+\frac \vep2$. Then for all $n$, $\dir_{x,t_n}^b \le \dir $. However, by a symmetric argument to the previous case, for sufficiently large $n$,  $g^{b,L}_{(x,t_n)}(t_0) > \rs_\dir^{t_0}$, and hence $\dir_{x,t_n}^b > \dir$, a contradiction.   
\end{proof}

\begin{proof}[Proof of Theorem \ref{thm:shocks_coal}\ref{it:ext_pt}]
First, we claim that if $\dir \in \dirs^b$ and $\rs_\dir^s = \infty$ for some $s < \tex(\dir)$, then $\dir \in \dirs^{b,t}$ for all $t > s$ and therefore $\tex(\dir) = \infty$. The assumption that $\rs_\dir^s = \infty$ implies that there exists $M \in \R$ so that $\dir_{x,s}^b = \dir$ for all $x \ge M$. Let $t > s$ be arbitrary. By Lemma \ref{lem:SIG_facts_collect}\ref{it:limits_to_infty},
\[
\lim_{x \to \infty} g_{(x,t)}^{b,L}(s) \ge \infty,
\]
and so $g_{(x,t)}^{b,L}(s) \ge M$ for sufficiently large $x$. By the consistency of geodesics in Lemma \ref{lem:geodesics_from_b}\ref{itm: consis}, $\dir_{x,t}^b = \dir$ for such $x$, and so $\dir \in \dirs^{b,t}$. By a symmetric proof, we see that if $\ls_\dir^s = -\infty$ for some $s \in \R$, then $\tex(\dir) = \infty$ as well. 

Now, assume that $\ls_\dir^s > -\infty$ and $\rs_\dir^s < \infty$ for some $s \in \R$, and assume, by way of contradiction, that $\tex(\dir)=\infty$. We will show that we can construct a bi-infinite geodesic in the directed landscape, similar to the proof of Lemma 5.1 in our previous work \cite{bhat-bus-soren}. By Item \ref{it:Cb_set}, we have $-\infty < \ls_\dir^s < \rs_\dir^s < \infty$ for all $s \in \R$.
        Now take a sequence $t_n \geq s$ such that $\lim_{n \rightarrow \infty}t_n=\infty$, and define
        \[
        y_n:=\frac{\ls^{t_n}_{\dir}+\rs^{t_n}_{\dir}}{2}.
        \]
        Now consider the geodesics $g_{n}:=g^{b,L}_{(y_n,t_n)}$. Since each of these has color $\dir$, we have that
        \[
       \ls^{r}_{\dir}<g^{b,L}_{(y_n,t_n)}(r) \leq \rs^{r}_\dir,\quad\forall\; r < s.
        \]
         By compactness, there exists a subsequence  along which $g_{(y_n,t_n)}^{b,L}(s)$ and $g_{(y_n,t_n)}^{b,L}(s)$ converge. By Lemma \ref{lem:precompact}, there exists a subsequence $\{k^1_{\ell}\}$ such that \
       \begin{equation}
           g_{k^1_l}|_{[s,s+1]}\rightarrow \Tilde{g}_1,
       \end{equation}
       uniformly, where the limiting object $\Tilde{g}_1$ is a geodesic on $[s,s+1]$. By a similar argument, we can find a subsequence  $\{k^2_{\ell}\}\subseteq \{k^1_\ell\}$ such that \ 
       \begin{equation}
           g_{k^2_{\ell}}|_{[s,s+2]}\rightarrow \Tilde{g}_2,
       \end{equation}
       uniformly,  where $\Tilde{g}_2$ is a geodesic on $[s,s+2]$ and $\Tilde{g}_2|_{[s,s+1]}=\Tilde{g}_1$.  Using a diagonalization argument, we can find a sequence $g_{k^\ell_{\ell}}$ such that \ $g_{k^\ell_{\ell}}\rightarrow \Tilde{g}_\infty$ as $\ell\rightarrow \infty$, uniformly on compact sets. We now define
       \begin{equation}
           g_{\infty}(r)=
           \begin{cases}
               \Tilde{g}_\infty(r) & r\geq s\\
               g^{\dir,L}_{(g_{\infty}(s),s)}(r) & r<s.
           \end{cases}
       \end{equation}
       We claim $g_\infty$ is a bi-infinite geodesic. 
       To do this, we show that for $s_1 < s_2$,
       \be \label{eq:L=W_final}
       \Ll\bigl(g_\infty(s_1),s_1;g_\infty(s_2),s_2\bigr) = W^\dir\bigl(g_\infty(s_1),s_1;g_\infty(s_2),s_2\bigr).
       \ee
       Then, the fact that $g_\infty$ is a bi-infinite geodesic follows from the additivity of the Busemann functions and \eqref{eq:landscape_leq}. By Equation \eqref{L =landscape} of Lemma \ref{lem:W=landscape}, we have that \eqref{eq:L=W_final} holds when $s_2 < s$. By additivity and \eqref{eq:landscape_leq}, it then suffices to show that, for all $r' > s$,
       \be \label{L=W_f2}
       \Ll(g_{\infty}(s),s;r',g_\infty(r')) = W^{\dir}(g_{\infty}(s),s;r',g_\infty(r')).
       \ee
       To see this, recall $g_k =g_{(y_k,t_k)}^{\dir,L}$, and then Lemma \ref{lem:W=landscape} implies that for all $k \ge r'$,
       \[
       \Ll(g_{\infty}(s),s;r',g_k(r')) = W^{\dir}(g_{\infty}(s),s;r',g_k(r')).
       \]
       Then, \eqref{L=W_f2} follows by taking limits and using continuity of the Busemann functions. This contradicts the non-existence of bi-infinite geodesics in the directed landscape \cite[Proposition 34]{Bha24}. 

      We finish by proving \eqref{eq:extinction}. Since $\ls_\dir^s <\rs_\dir^s$ for all $\dir \in \dirs^b$ and $s < \tex(\dir)$, it suffices to derive a contradiction if we assume that
\[
-\infty \le a := \liminf_{s\uparrow t} \ls_\dir^s < \limsup_{s \uparrow t} \rs_\dir^s =: b \le \infty. 
\]
Choose $a',b' \in (a,b)$ so that $a' < b'$. Then, there exists a sequence of times $s_n \uparrow t$ so that, for all $n \in \N$,
\be \label{eq:bdab}
\ls_\dir^{s_n} < a' < b' < \rs_\dir^{s_n}.
\ee
Now, choose $x \in (a',b')$. Then, $\dir_{x,\tex(\dir)}^b \neq \dir$ by definition of $\tex(\dir)$. However, by continuity of geodesics and since $s_n \uparrow t$, there exists $N \in \N$ so that for all $n \ge N$,
\[
g_{(x,\tex(\dir))}^{b,L}(s_n) \in (a',b').
\]
Then, by \eqref{eq:bdab} and the consistency of geodesics in Lemma \ref{lem:geodesics_from_b}\ref{itm: consis}, $\dir_{x,\tex(\dir)}^b = \dir$, a contradiction.
\end{proof}

\begin{proof}[Proof of Theorem \ref{thm:shocks_coal}\ref{it:locally_finite}]
     We first show that no two colors have the same extinction point. Assume, without loss of generality that $\dir_1< \dir_2$ but $\pex(\dir_1)=\pex(\dir_2)=(x_{\text{ext}},t_{\text{ext}}
    )$. We will show that we can construct four disjoint semi-infinite geodesics from $(x_{\text{ext}},t_{\text{ext}}
    )$, giving a contradiction to the non-existence of `4-stars' in the directed landscape shown in \cite[Theorem 1.5]{Dauvergne-23}.  Let $\{t_n\}_{n\geq 1}$ be a sequence of times such that $t_n < t_{\text{ext}}(\dir)$ for all $n$ and $t_n \rightarrow t_{\text{ext}},$ as $n \rightarrow \infty.$ Consider the sequences $\{w_n\}_{n \geq 1},\{x_n\}_{n \geq 1}, \{y_n\}_{n \geq 1}, \{z_n\}_{n \geq 1}$ defined as follows
    \[
    w_n:=\ls^{t_n}_{\dir_1}-\frac 1n, x_n:=\frac{\ls_{\dir_1}^{t_n}+\rs_{\dir_1}^{t_n}}{2},y_n:=\frac{\ls_{\dir_2}^{t_n}+\rs_{\dir_2}^{t_n}}{2},  z_n:=\rs_{\dir_2}^{t_n}+\frac 1n.
    \]
    Note that as the set  of colors of the leftmost geodesics is locally finite (Lemma \ref{lem:2}), and since the function $t \mapsto \mathfrak a_\xi^t$ is continuous by Item \ref{it:Cb_set}, there is a subsequence $n_k$ such that $g^{b,L}_{(w_{n_k},t_{n_k})}$ (resp.\ $g^{b,L}_{(z_{n_k},t_{n_k})}$) will have the same color $\dir_{\ell}$ (resp.\ $\dir_r$). For ease of notation, we will assume this holds for the whole sequence, instead along the subsequence $n_k$. Since these geodesics all have different colors, and any two leftmost $b$-geodesics coalesce if they ever meet (Lemma \ref{lem:geodesics_from_b}\ref{itm: consis}), for each $n$ the leftmost $b$ geodesics starting from $(w_n,t_n),(x_n,t_n),(y_n,t_n),(z_n,t_n)$ are disjoint. Therefore, for all $s \leq t_n$, 
    \be \label{eq:geod_distinct}
  g^{b,L}_{(w_n,t_n)}(s)<  g^{b,L}_{(x_n,t_n)}(s)<g^{b,L}_{(y_n,t_n)}(s)<g^{b,L}_{(z_n,t_n)}(s).
    \ee
    By Lemma \ref{lem:precompact}, each of the four sequences of geodesics $g^{b,L}_{(w_n,t_n)}, g^{b,L}_{(x_n,t_n)}, g^{b,L}_{(y_n,t_n)},g^{b,L}_{(z_n,t_n)}$ converges to a semi-infinite geodesic starting from $(\xex,\tex)$ in the sense of the local Hausdorff topology on paths. Call these four geodesics $g_1,g_2,g_3,g_4$. By way of contradiction, assume that these geodesics are not disjoint. Then, without loss of generality, we may assume that there exists $t < \tex$ such that $g_1(t) = g_2(t)$.
    Fix $\delta \in (0,\f{\tex - t}{2})$. Since geodesics do not form interior bubbles (see \cite[Theorem 1]{Bha24}), it must be the case that the restrictions $g_1|_{[t - \delta,t+\delta]}$ and $g_2|_{[t - \delta,t+\delta]}$ are unique geodesics between their endpoints. By Lemma \ref{lem:overlap}, this implies  that $g^{b,L}_{(w_n,t_n)}|_{[t-\delta,t+\delta]}$ and$ g^{b,L}_{(x_n,t_n)}|_{[t-\delta,t+\delta]}$ converge to  $g_1|_{[t - \delta,t+\delta]}$ and $g_2|_{[t - \delta,t+\delta]}$, respectively, in the overlap sense. Then, also using \eqref{eq:geod_distinct}, for sufficiently large $n$,
\[
g_1(t) = g_{(w_n,t_n)}^{b,L}(t) < g_{(x_n,t_n)}^{b,L}(t) = g_2(t),
\]
a contradiction to $g_1(t) = g_2(t)$. Hence, no two colors have the same extinction point. The existence of three semi-infinite geodesics from each extinction point follows the same construction above. 

     Now, we show that the set of extinction points is locally finite. Assume by way of contradiction that there exists $M>0$ such that in the open square $S_M:=(-M,M)^2$ there exist infinitely many extinction points. Take $T$ large enough such that
    \begin{equation}
    \begin{aligned}
      &g^{b,L}_{-M-T,M}(s)<-M\qquad \forall s\in[-M,M],\quad\text{and}\\
      &g^{b,L}_{M+T,M}(s)>M\qquad \forall s\in[-M,M].
    \end{aligned}
    \end{equation}
    Then, by order of geodesics, any geodesic leaving from $S_M$ will cross the time horizon $-M$ through the interval $I^{M,L}=[g^{b,L}_{-M-T,M}(-M),g^{b,L}_{M+T,M}(-M)]$.  We now claim that the number of disjoint geodesics leaving from $S_M$ is at least the number of extinction points in $S_M$. Indeed, for each extinction point $\pex(\dir) \in S_M$, we can find a point $(y(\dir),s(\dir)) \in S_M$ sufficiently close so that $\dir_{y(\dir),s(\dir)}^b = \dir$. By the consistency in Lemma \ref{lem:geodesics_from_b}\ref{itm: consis}, if $\dir_{x,t}^b \neq \dir_{y,s}^b$ for $(x,t),(y,s) \in \R^2$, the geodesics $g_{(x,t)}^{b,L}$ and $g_{(y,s)}^{b,L}$ cannot meet.  Thus for distinct values $\dir$, the associated geodesics with color $\dir$ starting near $p(\dir)$ must be distinct.  We conclude that  there exist infinitely many disjoint geodesics started from $S_M$ and crossing the interval $I^{M,L}$. But this does not occur in the directed landscape (see, for example, \cite[Lemma B.14]{Busa-Sepp-Sore-22a}).
\end{proof}

   \appendix
\section{Auxiliary results}
 Recall the set $\Xs = \Xs_\omega$ defined in \eqref{eq:X_space}, along with the projections $\mathcal P$ and $\sn$ defined in \eqref{eq:P_s_funs}.

\begin{lemma} \label{lem:X_metric}
    The space is $\Xs$ is metrizable.
\end{lemma}
\begin{proof}
    This can be derived directly through Urysohn's metrization theorem, but here we construct the metric explicitly. Let $\{\dir_n\}_{n\in\N}=\Xi$ be the set of exceptional directions. That we can index this set by $\N$ is a consequence of that fact that $\Xi$ is countable. Define the   function $l$ on pairs of the form $(\dir_n^-,\dir_n^+)$  by
    \begin{equation}
        l(\dir_n^-,\dir_n^+)=2^{-n}.
    \end{equation}
    We now construct  $d$ from $l$ by
    \begin{equation}
        d(\dir_1,\dir_2)=\sum_{\substack{\dir_1\wedge \dir_2\leq\dir^-_n<\dir_1\vee \dir_2}} l(\dir_n^-,\dir_n^+)
    \end{equation}
    In words, one can imagine an edge between any pair $(\dir^-_n,\dir^+_n)$ and a weight function $l$ giving weight $2^{-n}$ to that edge. The distance between two points in $\Xs$ is the total weight on all edges between them. Clearly $d$ is finite and positive. As $\Xi$ is dense,  $d(\dir_1,\dir_2)=0$ if and only if $\dir_1=\dir_2$. The triangle inequality follows immediately, so $d$ is indeed a metric. To see that it generates the topology on $\Xs$ note that as $\Xs$ is first countable, it is sequential and it is therefore enough to show that any sequence converging with respect to the topology implies convergence in the metric and vice versa. Note that  a sequence $\dir_n$ converges to a point $\dir\in\Xi^+$ in the topology on $\Xs$ if and only if $\mathcal{P}(\dir_n)\downarrow\mathcal P(\dir)$ if and only if $d(\dir_n,\dir)\rightarrow 0$. The convergence to points in $\Xi^-$  and $\Xi^c$ is shown analogously. Hence, the topology on $\Xs$ is generated by $d$.      
\end{proof}

           \begin{lemma}[\cite{Directed_Landscape}, Corollary 10.7] \label{lem:Landscape_global_bound}
There exists a random constant $C_{\text{DL}}$ such that for all $v = (x,s;y,t) \in \Rup$, we have 
\[
\Bigl|\Ll(x,s;y,t) + \f{(x - y)^2}{t - s}\Bigr| \le C_{\text{DL}} (t - s)^{1/3} \log^{4/3} \Bigl(\f{2(\|v\| + 2)}{t - s}\Bigr)\log^{2/3}(\|v\| + 2),
\]
where $\|v\|$ is the Euclidean norm.
\end{lemma}

\begin{lemma}[\cite{Dauvergne-Sarkar-Virag-2022}, Lemma 3.1] \label{lem:precompact}
The following holds on a single event of probability one. Let $(p_n;q_n) \to (p,q) = (x,s;y,t) \in \Rup$, and let $g_n$ be any sequence of geodesics from $p_n$ to $q_n$. Then, the sequence of graphs $\graph g_n$ is precompact in the Hausdorff metric, and any subsequential limit of $\graph g_n$ is the graph of a geodesic from $p$ to $q$. 
\end{lemma}

\begin{lemma}[\cite{Dauvergne-Sarkar-Virag-2022}, Lemma 3.3] \label{lem:overlap}
The following holds on a single event of full probability. Let $(p_n;q_n) = (x_n,s_n;y_n,u_n) \in \Rup \to (p;q) = (x,s;y,u) \in \Rup$, and let $g_n$ be any sequence of geodesics from $p_n$ to $q_n$. Suppose that either
\begin{enumerate} [label=\rm(\roman{*}), ref=\rm(\roman{*})]  \itemsep=3pt
    \item \label{uniqn} For all $n$, $(x_n,s_n;y_n,u_n) \in \Rup \cap \Q^4$, or
    \item \label{uniqueg} There is a unique geodesic $g$ from $p$ to $q$.
    \end{enumerate}
    Then, the \textbf{overlap}
    \[
    O(g_n,g) := \{t \in [s_n,u_n]\cap [s,u]: g_n(t) = g(t)\}    
    \]
    is an interval for all $n$ whose endpoints converge to $s$ and $u$. 
\end{lemma}

\begin{lemma}\label{lem:f-ay_bd}
    The following holds on a single event of probability one. Let $s \in \R$, and let $f:\R  \to \R\cup \{-\infty\}$ be a function such that $\kpzs_t(0;f,s) < \infty$ for some $t > s$. Then, for for every $a > \f{1}{t-s}$,
    \[
    \sup_{y \in \R}[f(y) - ay^2] < \infty.
    \]
\end{lemma}
\begin{proof}
    Since $\f{1}{t-s} < a$, Lemma \ref{lem:Landscape_global_bound} implies that 
    $\Ll(y,s;0,t) \ge -ay^2 - C$ for a random constant $C$ depending on $t$ and $s$. Then,
    \[
    \sup_{y \in \R}[f(y) - ay^2] \le C + \sup_{y \in \R}[f(y)  + \Ll(y,s;0,t)] = C + \kpzs_t(0;f,s) < \infty. \qedhere
    \]
\end{proof}

\begin{lemma} \label{lem:general_restrict_max}
    The following holds on a single event of probability one. Let $s \in \R$, and let $f:\R \to [-\infty,\infty)$ be any function which satisfies $f(y^\star) > -\infty$ for some $y^\star \in \R$ and $\sup_{y \in \R}[f(y) - ay^2] < \infty$ for all $a > 0$. Then, for any compact subset $K \subseteq \R \times \R_{>s}$, there exists $A = A(s,K) > 0$ so that for all $(x,t) \in K$,
    \be \label{eq:restrict_max}
    \kpzs_t(x;f,s) = \sup_{y \in [-A,A]}[f(y) + \Ll(y,s;x,t)] > \sup_{|y| > A}[f(y) + \Ll(y,s;x,t)].
    \ee
    Furthermore, the function 
    $
    (x,t) \mapsto \kpzs_t(x;f,s)$
    is continuous on $\R \times \R_{>s}$. 
\end{lemma}
\begin{proof}
   We first prove \eqref{eq:restrict_max}. By the assumption that $f(y^\star) > -\infty$ and by continuity of the directed landscape, we have 
    \[
    M_K := \inf_{(x,t) \in K}[f(y^\star) + \Ll(y^\star,s;x,t)] > -\infty
    \]
    By compactness of $K$, we may choose 
    $c \in \bigl(0, \inf\{\f{1}{t-s}: (x,t) \in K\}\bigr)$. Since $K$ is a compact set of $\R \times \R_{>s}$, $t-s$ is also bounded away from $0$ for $(x,t) \in K$, so by Lemma \ref{lem:Landscape_global_bound}, there exists a constant $C_1 = C_1(s,K)$ so that 
    \[
    \Ll(y,s;x,t) \le -cy^2 + C_1,\quad\text{for all }y \in \R \text{ and }(x,t) \in K.
    \]
    Then, setting $a = \f{c}{2}$, our assumption implies there is another constant $C_2 > 0$ so that $f(y) \le ay^2 + C_2$ for all $y \in \R$. Then, choose $A = A(C_1,C_2)$ sufficiently large so that $A > |y^\star|$ and
    \[
     C_1 + C_2 + \sup_{y > |A|}[(a-c)y^2] < M_K.
    \]
    Then, for all $(x,t) \in K$, recalling the definition of $M_K$, and since $A > |y^\star|$,
    \[
    \sup_{y > |A|}[f(y) + \Ll(y,s;x,t)] \le C_1 + C_2 + \sup_{y > |A|}[(a-c)y^2] < M_K \le \sup_{y \in [-A,A]}[f(y) + \Ll(y,s;x,t)],
    \]
    thus proving \eqref{eq:restrict_max}. To see the continuity, let $(x,t) \in \R \times \R_{>s}$, and let $(x_n,t_n)$ be a sequence converging to $(x,t)$. By \eqref{eq:restrict_max}, there exists $A > 0$ so that for all sufficiently large $n$, 
    \be \label{eq:restrict_to_A}
\begin{aligned}
    \kpzs_t(x;f,s) &= \sup_{y \in [-A,A]}[f(y) + \Ll(y,s;x,t)],\quad \text{and}\\
    \kpzs_{t_n}(x_n;f,s) &= \sup_{y \in [-A,A]}[f(y) + \Ll(y,s;x_n,t_n)].
\end{aligned}
    \ee
    By continuity of the directed landscape, the function $\Ll(y,s;x_n,t_n)$ converges to $ \Ll(y,s;x,t)$, uniformly over $y \in [-A,A]$. Hence, for $\ve > 0$, and all sufficiently large $n$, 
    \[
    -\ve + \sup_{y \in [-A,A]}[f(y) + \Ll(y,s;x,t)] \le\sup_{y \in [-A,A]}[f(y) + \Ll(y,s;x_n,t_n)] \le \ve + \sup_{y \in [-A,A]}[f(y) + \Ll(y,s;x,t)].
    \]
    Hence, by \eqref{eq:restrict_to_A}, we have that  $\lim_{n \to \infty}\kpzs_{t_n}(x_n;f,s) = \kpzs_t(x;f,s)$, as desired. 
\end{proof}

\begin{corollary} \label{cor:cts_bounded}
    Let $b$ be a finite eternal solution to the KPZ fixed point. Then, $(x,t) \mapsto b(x,t)$ is continuous, and for each $s \in \R$ and $(x,t) \in \R$, the set of maximizers of $y \mapsto b(y,s) + \Ll(y,s;x,t)$ is nonempty and bounded. 
\end{corollary}
\begin{proof}
    Recall that a finite eternal solution satisfies
    \[
   -\infty <  b(x,t) = \sup_{y \in \R}[b(y,s) + \Ll(y,s;x,t)] < \infty
    \]
    for all $s < t$ and $x \in \R$. Then, by Lemma \ref{lem:f-ay_bd}, for any $s \in \R$ and $a > 0$,
    \[
    \sup_{y \in \R}[b(y,s) - ay^2] < \infty. 
    \]
    The result now follows from  Lemma \ref{lem:general_restrict_max}. 
\end{proof}

\begin{lemma} \cite[Lemma A.12]{bhat-bus-soren}\label{lem:geodesics_from_b}
Let $b:\R^2 \to \R$ be an eternal solution for the directed landscape. For $(x,t) \in \R^2$, define $g_{(x,t)}^{b,L/R}$ as in Definition \ref{def:b-geodesics}.
Then, the following hold
\begin{enumerate} [label=(\roman*), font=\normalfont]
\item \label{itm:geod} $g_{(x,t)}^{b,L}:(-\infty,t] \to \R$ \rm{(}resp. $g_{(x,t)}^{b,R}:(-\infty,t] \to \R$\rm{)} is a semi-infinite geodesic that is the leftmost (resp. rightmost) directed landscape geodesic between any two of its points. In particular, $g_{(x,t)}^{b,L}$ and $g_{(x,t)}^{b,R}$ are continuous functions $(-\infty,t] \to \R$.
\item \label{itm: consis} Let $s < r < t$ and $x \in \R$, and let $
w = g_{(x,t)}^{b,L}(r)$.
Then,
\[
g_{(x,t)}^{b,L}(s) = g_{(w,r)}^{b,L}(s),
\]
 In particular, recalling Definition \ref{def:color}, we have $\dir_{x,t}^b = \dir_{w,r}^b$. The analogous statement holds when replacing $L$ with $R$.
\item \label{itm: general}
More generally, let $t = s_0 > s_1 > \cdots$ be any decreasing sequence with $\lim_{n \to \infty} s_n = -\infty$. Set $g(t) = x$, and for each $i \ge 1$, let $g(s_i)$ be \textit{any} maximizer of $z \mapsto b(z,s_{i}) + \Ll(z,s_i; g(s_{i - 1}),s_{i - 1})$ over $z \in \R$. Then, pick \textit{any} point-to-point geodesic of $\Ll$ from $(g(s_{i}),s_{i})$ to $(g(s_{i-1}),s_{i-1})$, and for $s_{i} < s < s_{i-1}$, let $g(s)$ be the location of this geodesic at time $s$. Then, $g:(-\infty,t] \to \R$ is a semi-infinite geodesic. 
\end{enumerate}
\end{lemma}

\begin{lemma} \label{lem:one_extension}
    Let $\wt \varphi$ be any function $\Xs\rightarrow [-\infty,\infty)$.   Define the function $ \varphi:\Xs \to [-\infty,\infty]$ by 
    \be \label{eq:extension}
    \varphi(\xi) := \limsup_{\eta \to \xi}\wt \varphi(\eta) =  \lim_{\ve \downarrow 0} \sup_{0 \le |\eta - \xi| \le \ve} \wt \varphi(\eta),
    \ee
    where we have added the second equality to clarify that we allow $\eta = \xi$ in the supremum. 
    Then, $\varphi \in \operatorname{USC}(\Xs)$, and $W\star  \varphi = W\star \wt \varphi$.
\end{lemma}
\begin{remark}
    This lemma allows the case where $(W\star \varphi)(x,t) = \infty$ for some $(x,t) \in \R^2$, although we shall always work with functions such that $W\star \varphi < \infty$ in practice. 
\end{remark}
\begin{proof}
    This is a standard fact about upper semicontinuous extensions, and we provide a proof for completeness.   For all $\xi \in \Xs$, 
    \[
    \limsup_{\eta \to \xi} \varphi(\eta) = \limsup_{\eta \to \xi}\limsup_{\zeta \to \eta} \wt \varphi(\zeta) \le \lim_{\ve \downarrow 0}\sup_{|\eta -\xi| \le \f{\ve}{2}} \sup_{|\zeta - \eta| \le \f{\ve}{2}} \wt \varphi(\zeta) \le \lim_{\ve \downarrow 0} \sup_{|\zeta - \xi| \le \ve} \wt \varphi(\zeta) =  \varphi(\xi).
    \]
Thus, $ \varphi$ is upper semicontinuous. 

 It follows immediately from the definition that $ \varphi(\xi) \ge \wt \varphi(\xi)$ for all $\xi \in \Xs_\omega$, which immediately gives us $(W\star \varphi)(x,t) \ge (W\star \wt \varphi)(x,t)$ for all $(x,t) \in \R^2$. To show the opposite inequality, note that by definition, for each $\xi_0 \in \Xs$, there exists a sequence $\xi_n \to \xi_0$ such that $\wt \varphi(\xi_n) \to  \varphi(\xi_0)$ as $n \to \infty$. Then, for each $(x,t) \in \R^2$, by continuity of the Busemann process, 
    \begin{align*}
     \varphi(\xi_0) + W^{\xi_0}(0,0;x,t) &= \lim_{n \to \infty} [ \wt \varphi(\xi_n) + W^{\xi_n}(0,0;x,t)]  \\
    &\le \sup_{\xi \in \Xs_\omega}[\wt \varphi(\xi) + W^{\xi}(0,0;x,t)] = (W\star \wt \varphi)(x,t),
    \end{align*}
    and the result now follows by taking the supremum of the left-hand side over $\xi_0 \in \Xs_\omega$.
\end{proof}

\begin{lemma} \label{lem:USC_bd_cpct}
    If $\varphi:\Xs_\omega\rightarrow[-\infty,\infty)$ is upper semicontinuous, then it is bounded on compact sets. 
\end{lemma}
\begin{proof}
    Suppose not. Then, there exists a bounded sequence $\xi_n$ so that $\varphi(\xi_n) \to \infty$. By compactness, there exists a subsequence $n_k$ so that $x_{n_k} \to x$ for some $x$. But this is a contradiction, because upper semicontinuity implies 
    \[
    \infty = \limsup_{k \to \infty} \varphi(x_{n_k}) \le  \varphi(x) < \infty. \qedhere
    \]
\end{proof}
\begin{lemma}\label{lm:mincon}
    
For almost every $\omega$, $(\varphi_1,\varphi_2)\in\fun_\omega\times\fun_\omega$ is a continuity point of $(\fun_\omega,\wedge)$ if and only if $\supp(\varphi_1)\cap \supp(\varphi_2)=\varnothing$.
\end{lemma}
\begin{proof}
 Assume first that $\supp(\varphi_1)\cap \supp(\varphi_2)=\varnothing$. This implies that $\varphi_1\wedge\varphi_2\equiv -\infty$.  Assume that $(\varphi^n_1,\varphi^n_2)\rightarrow (\varphi_1,\varphi_2)$. Condition \eqref{it:c1} of Definition \ref{def:fun_convergence} automatically holds for $\varphi_1^n \wedge \varphi_2^n$ from the convergence $\varphi_1^n \to \varphi_1$. Since $\supp(\varphi_1)\cap \supp(\varphi_2) = \varnothing$,  for any $\dir\in\supp(\varphi_1)$,  $\varphi_2^n(\dir)\rightarrow-\infty$ so that $\varphi_1^n(\dir)\wedge \varphi_2^n(\dir)\rightarrow-\infty$. By a symmetric argument for $\dir\in\supp(\varphi_2)$,   we conclude that 
 \begin{equation}
     \lim_{n\rightarrow\infty} \varphi_1^n(\dir)\wedge \varphi_2^n(\dir)=-\infty, \qquad \forall \dir\in \Xs.
 \end{equation}
   Furthermore, by the characterization of hypograph convergence in \eqref{lephi}-\eqref{recov_phi}, we see that Condition \eqref{it:c2} of Definition \ref{def:fun_convergence} holds, and $(\varphi_1,\varphi_2)$ is a continuity point.

 Assume now that $\supp(\varphi_1)\cap \supp(\varphi_2)\neq\varnothing$. We now claim that  as the space $\Xs$ has no isolated points, one can find two sequences  $\{D_1^n\}$ and $\{D_2^n\}$ where $D_1^n,D_2^n\subseteq \Xs$ such that $D_1^n\cap D_2^n=\varnothing$ and $\varphi^n_i$ supported on $D_i^n$ such that $\varphi^n_i\xrightarrow{\mathrm{\mathcal{T}_\fun}} \varphi_i$. Indeed let $D_i^n$ be such that
 \begin{equation}
     \mathcal{P}(D_i^n)=\frac{2\Z+(i-1)}{n}, \qquad i\in\{1,2\},
 \end{equation}
 and for each $\dir\in D^n_i$, let
 \begin{equation}
 \varphi^n_i(\dir)=
 \begin{cases}
      \sup_{\dir'\in (\dir-n^{-1},\dir+n^{-1})}\varphi_i(\dir') & \dir\in D^n_i\\
      -\infty & \dir\in (D^n_i)^c.
 \end{cases}
 \end{equation}
 Since the support of each $\varphi_i^n$ has no limit points for fixed $n$, we see that $\varphi_i^n$ is upper semicontinuous and $\varphi_i^n \in \Phi$. Since $\varphi_i \in \Phi$, we readily see that Condition \eqref{it:c1} of Definition \ref{def:fun_convergence} holds. By upper semicontinuity of $\varphi_i$, we readily see that for any $\dir_n \to \dir \in \Xs$, we have $\limsup_{n \to \infty} \varphi_n(\dir_n) \le \varphi(\dir)$. Furthermore, taking $\dir_n^i \to \dir$ with $\dir_n^i \in D_i^n$, we see that $\lim_{n \to \infty} \varphi_n(\dir_n^i) = \varphi(\dir)$. Hence, Item \eqref{it:c2} of Definition \ref{def:fun_convergence} holds by the equivalent conditions in \eqref{lephi}-\eqref{recov_phi}. Hence,  $\varphi_i^n\xrightarrow{\mathrm{\mathcal{T}_\fun}} \varphi_i$ and $\supp(\varphi^n_1)\cap \supp(\varphi^n_2)=\varnothing$. It follows that  $\varphi_1^n\wedge\varphi_2^n=-\infty$ for every $n$. But $\varphi_1\wedge\varphi_2\neq -\infty$ which shows that $(\varphi_1,\varphi_2)$ is not a continuity point, a contradiction.
\end{proof}
\begin{lemma}\label{lm:o6}
    For almost-every $\omega$, whenever $b\in\es_\omega$ satisfies $b \not \equiv -\infty$, there exists $\{b_n\}\subseteq\es$ such that $b_n\rightarrow-\infty$ and \ $b_n\wedge b\notin \es$ for all $n\in\N$. Additionally, there exists a sequence $(b_1^n,b_2^n) \to (-\infty,-\infty)$ such that $b_i^n \in \es$ for $i = 1,2$, $b_1^n \wedge b_2^n \notin \es$ for all $n$.  
\end{lemma}
\begin{proof}
    We first consider the case where $b$ is a finite eternal solution. We know by Lemma \ref{lem:eternal_time_s} that there exists an increasing sequence $\{x_i\}_{i\in\Z}\subseteq \R$, an nondecreasing sequence $\{\dir_i\}_{i \in \Z}$ of colors, and constants $\{\varphi_i\}_{i \in \Z}$ such that
    \begin{equation}
        b(x,0)=\varphi_i + W^{\dir_i}(0,0;x,0), \qquad \forall x\in[x_i,x_{i+1}].
    \end{equation}
    Let $M\rightarrow \infty$ be such that $M\in(x_{i_M},x_{i_{M+1}})$ for some $i_M\in\N$. The Busemann process has the following property: For a fixed $\dir_0\in \Xs$, $t\in\R$, any non-trivial interval $[a,b]\subseteq \R$ and $x_0\in(a,b)$ there exists $a<x^-<x_0<x^+<b$ and $\dir_0<\dir$ such that
    \begin{equation}\label{eq90}
    \begin{aligned}
          &W^{\dir_0}(x_0,t;x,t)<W^\dir(x_0,t;x,t), \qquad \forall x\in[a,x^-]\\           &W^{\dir_0}(x_0,t;x,t)<W^\dir(x_0,t;x,t), \qquad \forall x\in[x^+,b],
    \end{aligned}
    \end{equation}
    This can be seen from the queuing representation for the Busemann process \cite{CGM_Joint_Buse,Busanidiffuse,Seppalainen-Sorensen-21b}. From  \eqref{eq90} we see that  for each $M$, there exists $\dir^M\in\Xs$ and $c_M\in\R$ and $x_{i_M}<x_M^-<x_M^+<x_{i_M+1}$ such that
    \begin{equation}
    \begin{aligned}
          &c_M + W^{\dir^M}(0,0;x,0)<\varphi_{i_M} + W^{\dir_{i_M}}(0,0;x,0), \qquad \forall x\in[x_{i_M},x^-_M]\\           &c_M + W^{\dir^M}(0,0;x,0)>\varphi_{i_M}+W^{\dir_{i_M}}(0,0;x,0), \qquad \forall x\in[x^+_M,x_{i_{M+1}}].
    \end{aligned}
    \end{equation}
    Now set $b_n:= c_n + W^{\dir^n}(0,0;x,t)$, and assume without loss of generality that $n\in (x_{i_n},x_{i_{n+1}})$ (otherwise we can take a subsequence for which this holds). It follows that, along time horizon $t = 0$, $b\wedge b_n$ has the increments of  $W^{\dir_M}$ on $[x_{i_M}x^-_M]$ and the increments of  $W^{\dir_{i_M}}$ on $[x^+_M,x_{i_{M+1}}]$. But for eternal solutions, the colors are increasing on any time horizon by Lemnma \ref{lem:s3}. This shows that $b\wedge b_n\notin \es$. We are left to show that $b_n\rightarrow -\infty$. For $R>0$, let $\mathfrak{R}=[-R,R]^2$ be a rectangle on the plane with side length $R$. Then for $M_0$ large enough depending on $R$, for any $M>M_0$,
    \begin{equation}\label{eq100}
        R<\tau^{L,(c_M,c_{M+1}),\dir^M,\dir^{M+1}}_t, \qquad \forall t\in[-R,R].
    \end{equation}
    In words, the interface between $c_M + W^{\dir^M}(0,0;x,t)$ and $c_{M+1} + W^{\dir^{M+1}}(0,0;x,t)$ goes to the right of the rectangle $\mathfrak{R}$. Without loss of generality we may assume that
    \begin{equation}\label{eq101}
        \tau^{L,(c_M,c_{M+1}),\dir^M,\dir^{M+1}}_t<\tau^{L,(c_{M+1},c_{M+2}),\dir^{M+1},\dir^{M+2}}_t, \qquad \forall t\in[-R,R],
    \end{equation}
    otherwise we may take a subsequence $\{b_{n_k}\}$ for which this holds. Now, using the monotonicity of the Busemann process as well as \eqref{eq100}-\eqref{eq101} this shows that for any $(x,t)\in\mathfrak{R}$ we have the following monotone limit $b_n(x,t)\downarrow -\infty$. As the function $f=-\infty$ is continuous, Dini's Theorem implies that $b_n\rightarrow -\infty$ uniformly on $\mathfrak{R}$. This concludes the proof of the first statement. To see the second statement, pick any two eternal solutions $b_1,b_2$ such that $b_1\wedge b_2 \notin \es$ (we saw that these exist in the previous part). Then, setting $b_i^n = b_i - n$ gives the desired sequence.  
\end{proof}

\bibliographystyle{plain}
\bibliography{references}
\end{document}